\documentclass[a4paper,12pt]{article}
\usepackage[dvips]{graphicx}
\usepackage{epsf}
\usepackage{graphicx}
\usepackage{epsfig}
\usepackage{amsmath}
\usepackage[psamsfonts]{amssymb}
\usepackage[psamsfonts]{eucal}
\usepackage{amsthm}
\usepackage{textcomp}
\usepackage{mathptmx}
\usepackage[scaled]{helvet}
\usepackage[colorlinks=true]{hyperref}
\hypersetup{urlcolor=blue, citecolor=red}

\swapnumbers
\usepackage{diagrams}
\swapnumbers
\theoremstyle{plain}

\newtheorem{theo}[subsubsection]{Theorem}
\newtheorem{lemme}[subsubsection]{Lemma}
\newtheorem{coro}[subsubsection]{Corollary}
\newtheorem{prop}[subsubsection]{Proposition}
\theoremstyle{definition}
\newtheorem{defi}[subsubsection]{Definition}
\newtheorem{defis}[subsubsection]{Definitions}
\newtheorem{rmk}[subsubsection]{Remark}
\newtheorem{rmks}[subsubsection]{Remarks}
\newtheorem{ex}[subsubsection]{Example}
\newtheorem{exs}[subsubsection]{Examples}

\newcommand{\field}[1]{\mathbb{#1}}
\newcommand{\CC}{\field{C}}

\newcommand{\RR}{\field{R}}
\newcommand{\ZZ}{\field{Z}}
\def\id{\mathop{\rm id}\nolimits}
\def\Ad{\mathop{\rm Ad}\nolimits}
\def\ad{\mathop{\rm ad}\nolimits}
\def\orth{\mathop{\rm orth}\nolimits}

\def\d{\mathrm{d}}
\def\mathi{\mathrm{i}}
\def\GL{\mathrm{GL}}
\def\gl{\mathfrak{gl}}
\def\Aff{\mathrm{Aff}}
\def\aff{\mathfrak{aff}}
\def\Isom{\mathop{\rm Isom}\nolimits}
\def\SO{\mathop{\rm SO}\nolimits}
\def\so{\mathop{\mathfrak so}\nolimits}
\def\vect#1{\overrightarrow{\mathstrut#1}}
\begin{document}

\title
{Symmetries of Hamiltonian systems
\\on symplectic and Poisson manifolds}
\author{Charles-Michel Marle\\
Institut de Math{\'e}matiques de Jussieu\\
Universit{\'e} Pierre et Marie Curie\\
Paris, France
}

\maketitle
\tableofcontents
\section{Introduction}\label{sec:1}

This text presents some basic notions in symplectic geometry, Poisson geometry, Hamiltonian systems, Lie algebras and Lie groups actions on
symplectic or Poisson manifolds, momentum maps and their use for the reduction of Hamiltonian systems. It should be accessible to readers with a general knowledge of basic notions in differential geometry.
Full proofs of many results are provided.
\par\smallskip

\subsection{Contents of the paper} 
Symplectic and Poisson manifolds are defined in Sections 
\ref{SymplecticManifolds} and \ref{PoissonManifolds}, where their
basic properties are given, often with detailed proofs. Darboux theorem and the main results about the local structure of Poisson manifolds, however, are given without proof. Actions of a Lie group or of a Lie algebra on a smooth manifold and, when this manifold is endowed with a symplectic or a Poisson structure, symplectic, Poisson and Hamiltonian actions are introduced in Section 
\ref{SymplecticPoissonHamiltonian}.  
For Hamiltonian actions of a Lie group on a connected symplectic manifold, 
the equivariance of the momentum map with respect to an affine action of 
the group on the dual of its Lie algebra is proven, and the notion of 
symplectic cocycle is introduced. We prove (\ref{LieAlgebraCocycleIntegration})
that given a Lie algebra symplectic cocycle, there exists on the associated
connected and simply connected Lie group a unique corresponding Lie group
symplectic cocycle. The Hamiltonian 
actions of a Lie group on its cotangent bundle obtained by lifting the 
actions of the group on itself by translations on the left and on the right 
are fully discussed in Subsection \ref{ActionsOnCotangentBundle}. 
We prove that there exists a two-parameter family of deformations of these actions into a pair of mutually symplectically orthogonal Hamiltonian actions whose momentum maps are equivariant with respect to an affine action involving
any given Lie group symplectic cocycle (\ref{GeneralizedActionsOfGonTstarG}).   
The use of first integrals and, more generally, of momentum maps for the resolution of Hamiltonian dynamical systems, is discussed in 
Section~\ref{Reduction}.
For a system whose Hamiltonian is invariant under a Hamiltonian Lie algebra action, the Marsden-Weinstein reduction procedure can be used: through the use of Noether's theorem, this procedure leads to a reduced symplectic manifold 
on which a reduced Hamiltonian system can be solved in a first step. Another way of using the symmetries of the system rests on the use of the Euler-Poincar\'e equation. This equation can be written for classical Lagrangian mechanical systems when there exists a locally transitive Lie algebra action on their configuration space, or for the corresponding Hamiltonian systems when the
Lagrangian is hyper-regular. However, the Euler-Poincar\'e equation does not always lead to a reduction of the system: such a reduction occurs mainly when
the Hamiltonian can be expressed as the momentum map composed with a smooth function defined on the dual of the Lie algebra; the Euler-Poincar\'e equation is then equivalent to the Hamilton equation written on the dual of the Lie algebra. Finally in Section \ref{ExamplesDynamicalSystems} three classical examples  
are considered: the spherical pendulum, the motion of a rigid body around a fixed point and the Kepler problem. For each example the 
Euler-Poincar\'e equation is derived (for the Kepler problem a transitive Lie algebra action is obtained by adding the Lie algebra of the group of positive homotheties to the Lie algebra of the group of rotations around the attractive centre), the first integrals linked to symmetries are given. In this Section, the classical concepts of vector calculus on an Euclidean three-dimensional vector space (scalar, vector and mixed products) are used and their interpretation in terms of concepts such as the adjoint or coadjoint action 
of the group of rotations are explained.

\subsection{Further reading}
Of course this text is just an introduction. Several important parts of the theory of Hamiltonian systems are not discussed here, for example completely
integrable systems (although the three examples presented belong to that class
of systems), action-angle coordinates, monodromy, singular reduction, the Kolmogorov-Arnold-Moser theorem, symplectic methods in Hydrodynamics, $\ldots$ To extend his knowledge of the subject, the reader can consult  the books by 
Ralph Abraham and Jerry Marsden \cite{AbrahamMarsden78},
Vladimir Arnold \cite{arnold2}, 
Vladimir Arnold and Boris Khesin \cite{arnoldkhesin},
Patrick Iglesias-Zemmour \cite{Iglesias2000},
Camille Laurent-Gengoux, Anne Pichereau and Pol Vanhaecke
\cite{LaurentPichereauVanhaecke2013},
Yvette Kosmann-Schwarzbach (editor) \cite{Kosmann2013} on 
both the scientific and historical aspects of the
development of modern Poisson geometry, 
Izu Vaisman \cite{Vaisman94}.

\subsection{Notations}\label{Notations}
Our notations are those which today are generally used in differential geometry. For example, the tangent and the cotangent bundles to 
a smooth $n$-dimensional manifold $M$
are denoted, respectively,  by $TM$ and by $T^*M$, and their canonical
projections on $M$ by $\tau_M:TM\to M$ and by $\pi_M:T^*M\to M$. The space of
differential forms of degree $p$, \emph{i.e.} the space of smooth sections of
$\bigwedge^p(T^*M)$, the $p$-th exterior power of the cotangent bundle, is 
denoted by $\Omega^p(M)$. Similarly the space of multivectors of degree $p$, 
\emph{i.e.} the space of smooth sections of $\bigwedge^p(TM)$, the $p$-th
exterior power of the tangent bundle, is denoted by $A^p(M)$. By convention
$\Omega^p(M)=A^p(M)=0$ for $p<0$ or $p>n$, and 
$\Omega^0(M)=A^0(M)=C^\infty(M,\RR)$.  
The exterior algebras of differential forms and of multivectors are, respectively, $\Omega(M)=\oplus_{p=0}^n\Omega^p(M)$ and
$A(M)=\oplus_{p=0}^nA^p(M)$. Their main properties
are briefly recalled in Section 
\ref{Schouten-Nijenhuis}.
\par\smallskip

When $f:M\to N$ is a smooth map between two smooth manifolds $M$
and $N$, the natural lift of $f$ to the tangent bundles is denoted by
$Tf:TM\to TN$. The same notation $Tf:\bigwedge^pTM\to \bigwedge^pTN$ 
is used to denote its natural prolongation to the $p$-th exterior power of $TM$. The pull-back by $f$ of a smooth differential form 
$\alpha\in \Omega(N)$ is denoted by $f^*\alpha$. 
\par\smallskip

When $f:M\to N$ is a smooth diffeomorphism, the push-forward $f_*X$ of a a 
smooth vector field $X\in A^1(M)$ is the vector field $f_*X\in A^1(N)$ defined by
 $$f_*X(y)=Tf\Bigl(X\bigl(f^{-1}(y)\bigr)\Bigr)\,,\quad y\in N\,.$$ 
Similarly, the pull-back of a smooth vector field $Y\in A^1(N)$
is the vector field $f^*Y\in A^1(M)$ defined by
 $$f^*Y(x)=Tf^{-1}\Bigl(Y\bigl(f(x)\bigr)\Bigr)\,,\quad x\in M\,.$$ 
The same notation is used for the push-forward of any smooth tensor field on $M$ and the pull-back of any smooth tensor field on $N$.

\section{Symplectic manifolds}\label{SymplecticManifolds}

\subsection{Definition and elementary properties}

\begin{defi} A \emph{symplectic form} on a smooth manifold $M$ is a bilinear
skew-symmetric differential form $\omega$ on that manifold which satisfies the following two properties:

\begin{itemize}

\item{} the form $\omega$ is closed; it means that its exterior differential $\d\omega$ vanishes: $\d\omega=0$;
 
\item{} the rank of $\omega$ is everywhere equal to the dimension of $M$; it means that for each point $x\in M$ and each vector $v\in T_xM$, $v\neq 0$, there exists another vector $w\in T_xM$ such that $\omega(x)(v,w)\neq 0$.

\end{itemize}
Equipped with the symplectic form $\omega$, the manifold $M$ is called a \emph{symplectic manifold} and denoted $(M,\omega)$. One says also that $\omega$ determines a
\emph{symplectic structure} on the manifold $M$.
\end{defi}

\subsubsection{Elementary properties of symplectic manifolds}
Let $(M,\omega)$ be a symplectic manifold.

\par\smallskip\noindent
{\bf 1.\quad} For each $x\in M$ and each $v\in T_xM$ we denote by 
$\mathi(v)\omega:T_xM\to\RR$ the map $w\mapsto\omega(x)(v,w)$; it is a linear form on
the vector space $T_xM$, in other words an element of the cotangent space $T^*_xM$.
Saying that the rank of $\omega$ is everywhere equal to the dimension of $M$ amounts to say that the map $v\mapsto \mathi(v)\omega$ is an isomorphism of the tangent bundle $TM$ onto the cotangent bundle $T^*M$.

\par\smallskip\noindent
{\bf 2.\quad} Let $V$ be a finite-dimensional vector space, and $\eta:V\times V\to\RR$ be a skew-symmetric bilinear form. As above,  $v\mapsto \mathi(v)\eta$ is a linear map defined on $V$, with values in its dual space $V^*$. The \emph{rank} of $\eta$ is the dimension of the image of that map.
An easy result in linear algebra is that the rank of a skew-symmetric bilinear form is always an even integer. When $(M,\omega)$ is a symplectic manifold, for each $x\in M$ that result can be applied to the bilinear form
$\omega(x):T_xM\times T_xM\to \RR$, and we see that the dimension of
$M$ must be an even integer $2n$.

\par\smallskip\noindent
{\bf 3.\quad} The \emph{Darboux theorem}, due to the 
French mathematician
Gaston Darboux (1842--1917), states that in 
a $2n$-dimensional symplectic manifold $(M,\omega)$ any point 
has a neighbourhood
on which there exists local coordinates $(x^1,\ldots,\allowbreak x^{2n})$ 
in which the $(2n)\times(2n)$-matrix $(\omega_{i\,j})$ ($1\leq i,\ j\leq 2n$) 
of components of $\omega$ is a constant, skew-symmetric invertible matrix. 
We recall that
 $$\omega_{i\,j}=
 \omega\left(\frac{\partial}{\partial x^i},\frac{\partial}{\partial x^j}
  \right)\,.
 $$
These local coordinates can even be chosen in such a way that
 \begin{equation*}
   \omega_{i j}=\begin{cases}
                      1& \text{if $i-j=n$},\\
                     -1& \text{if $i-j=-n$},\\
                      0& \text{if $\vert i-j\vert\neq n$},
                \end{cases}
                \quad\quad 1\leq i,\ j\leq 2n\,.
 \end{equation*}
Local coordinates which satisfy this property are called \emph{Darboux local coordinates}.

\par\smallskip\noindent
{\bf 4.\quad} On the $2n$-dimensional symplectic manifold $(M,\omega)$, the $2n$-form
$\omega^n$ (the $n$-th exterior power of $\omega$) is a \emph{volume form} (it means that it is everywhere $\neq 0$). Therefore
a symplectic manifold always is \emph{orientable}.

\subsection{Examples of symplectic manifolds}

\subsubsection{Surfaces}
A smooth orientable surface embedded in an Euclidean 3-dimensional affine space, endowed with the area form determined by the Euclidean metric, is a symplectic manifold.
\par\smallskip

More generally, any $2$-dimensional orientable manifold, equipped with a nowhere vanishing area form, is a symplectic manifold.

\subsubsection{Symplectic vector spaces}
A  \emph{symplectic vector space} is a finite-dimensional real vector space $E$ equipped with a skew-symmetric bilinear form 
$\omega: E\times E\to \RR$ of rank equal to the dimension of $E$; 
therefore $\dim E$ is an even integer $2n$. Considered as a constant differential 
two-form on $E$, $\eta$ is symplectic, which allows us to consider 
$(E,\eta)$ as a symplectic manifold.
\par\smallskip

The canonical example of a symplectic vector space is the following. 
Let $V$ be a real 
$n$-dimensional vector space and let $V^*$ be its dual space. There exists
on the direct sum $V\oplus V^*$ a natural skew-symmetric
bilinear form 
 $$\eta\bigl((x_1,\zeta_1),(x_2,\zeta_2)\bigr)= \langle\zeta_1,x_2\rangle-\langle\zeta_2,x_1\rangle\,.$$
The rank of $\eta$ being $2n$, $(V\oplus V^*,\eta)$ is a symplectic vector space.
\par\smallskip

Conversely, any $2n$-dimensional symplectic 
vector space $(E,\omega)$ can be identified with the 
direct sum of any of its $n$-dimensional vector subspaces $V$ such that the symplectic form $\omega$ vanishes identically on $V\times V$, with its dual space $V^*$. In this identification, the symplectic form $\omega$ on $E$ becomes
identified with the above-defined symplectic form $\eta$ on $V\oplus V^*$.   

\subsubsection{Cotangent bundles}
Let $N$ be a smooth 
$n$-dimensional manifold. With the notations 
of \ref{Notations} for the canonical projections of
tangent or cotangent bundles onto their base manifold and for prolongation 
to vectors of a smooth map, we recall that the diagram
\newarrow{Mapsto}|--->
 $$
 \begin{diagram}
  T(T^*N)           &\rTo^{T\pi_N}  &TN\\
  \dTo^{\tau_{T^*N}}&               &\dTo_{\tau_N}\\
  T^*N             &\rTo^{\pi_N}   &N\\
 \end{diagram}
 $$ 
is commutative. For each $w\in T(T^*N)$, we can therefore write
 $$\eta_{N}(w)=\bigl\langle \tau_{T^*N}(w), T\pi_N(w)\bigr\rangle\,.
 $$
This formula defines a differential 1-form $\eta_{N}$ on the manifold
$T^*N$, called the \emph{ Liouville 1-form}. Its exterior differential $d\eta_N$ is a symplectic form, called the 
\emph{canonical symplectic form} on the cotangent bundle $T^*N$.
\par\smallskip

Let  $(x^1,\ldots,x^n)$ be a system of local coordinates on $N$, 
$(x^1,\ldots,x^n,\allowbreak
p_1,\ldots,p_n)$ be the corresponding system of local coordinates on 
$T^*N$. The local expressions of the Liouville form $\eta_N$ and of 
its exterior differential $d\eta_N$ are
 $$\eta_N=\sum_{i=1}^n p_i\,\d x^i\,,\quad 
   \d\eta_N=\sum_{i=1}^n \d p_i\wedge \d x^i\,.
 $$
We see that $(x^1,\ldots,x^n,\allowbreak
p_1,\ldots,p_n)$ is a system of Darboux local coordinates. Therefore any symplectic manifold is locally isomorphic to a cotangent bundle.

\subsubsection{The complex plane}
The complex plane $\CC$ is naturally endowed with a Hermitian form 
 $$\eta(z_1,z_2)=z_1\overline{z_2}\,,\quad z_1\ \hbox{and }z_2\in\CC\,,
 $$
where $\overline{z_2}$ is the conjugate of the complex number $z_2$. 
Let us write $z_1=x_1+iy_1$, $z_2=x_2+iy_2$, where $x_1,\ y_1,\ x_2,\ y_2$
are real, and separate the real and imaginary parts of $\eta(z_1,z_2)$. We get
 $$\eta(z_1,z_2)=(x_1x_2+y_1y_2) + i(y_1x_2-y_2x_1)\,.$$
The complex plane $\CC$ has an underlying structure of real, 
2-dimensional vector space, which can be identified with $\RR^2$, each complex number  
$z=x+iy\in\CC$  being identified with $(x,y)\in\RR^2$. The real and imaginary
parts of the Hermitian form $\eta$ on $\CC$ are, respectively, the Euclidean scalar product $g$ and the symplectic form $\omega$ on $\RR^2$
such that
 \begin{align*}
  \eta(z_1,z_2)&=(x_1x_2+y_1y_2) + i(y_1x_2-y_2x_1)\\
               &=g\bigl((x_1,y_1),(x_2,y_2)\bigr)+i\omega\bigl((x_1,y_1),(x_2,y_2))\,.
  \end{align*}

\subsubsection{K\"ahler manifolds}
More generally, a $n$-dimensional
\emph{K\"ahler manifold} (\emph{i.e.} a complex manifold of
complex dimension $n$ endowed with a Hermitian form whose imaginary part is a closed two-form), when 
considered as  a
real $2n$-dimensional manifold, 
is automatically endowed with a Riemannian metric 
and a symplectic form given, respectively, 
by the real and the imaginary parts of the Hermitian form.
\par\smallskip

Conversely, it is not always possible to endow a symplectic manifold
with a complex structure and a Hermitian form of which the given
symplectic form is the imaginary part. However, it is always possible to define, on a symplectic manifold, an \emph{almost complex
structure} and an \emph{almost complex 2-form} with which the
properties of the symplectic manifold become similar to those of a
K\"ahler manifold (but with change of chart functions which are not
holomorphic functions).
This possibility was used by \emph{Mikha{\"{\i}}l Gromov} \cite{Gromov85} 
in his theory of \emph{pseudo-ho\-lo\-mor\-phic curves}.

\subsection{Remarkable submanifolds of a symplectic manifold}

\begin{defis} Let $(V,\omega)$ be a symplectic vector space, and $W$ be a vector subspace of $V$. The \emph{symplectic orthogonal} of $W$ is the vector subspace
 $$\orth W=\{\,v\in V\,;\,\omega(v,w)=0\ \text{for all }w\in W\,\}\,.$$
The vector subspace $W$ is said to be

\begin{itemize}
\item{} \emph{isotropic} if $W\subset\orth W$,

\item{} \emph{coisotropic} if $W\supset\orth w$,

\item{} \emph{Lagrangian} if $W=\orth W$,

\item{} \emph{symplectic} if $W\oplus\orth W=V$.
\end{itemize}
\end{defis}

\subsubsection{Properties of symplectic orthogonality} 
The properties stated
below are  easily consequences of the above definitions                                        

\par\smallskip
\noindent
{\bf 1.\quad} For any vector subspace $W$ of the symplectic vector space $(W,\omega)$,
we have $\orth(\orth W)=W$.

\par\smallskip\noindent
{\bf 2.\quad} Let $\dim V=2n$. For any vector subspace $W$ of $V$, we have $\dim(\orth W)=\dim V-\dim W
=2n-\dim W$. 
Therefore, if $W$ is isotropic, $\dim W\leq n$; if $W$ is coisotropic,
$\dim W\geq n$; and if $W$ is Lagrangian, $\dim W=n$.

\par\smallskip\noindent
{\bf 3.\quad} Let $W$ be an isotropic vector subspace of $V$. The restriction to 
$W\times W$ of the symplectic form 
$\omega$ vanishes identically. Conversely, if $W$ is a vector subspace such that the restriction of $\omega$ to $W\times W$ vanishes identically, $W$ is isotropic. 

\par\smallskip\noindent
{\bf 4.\quad} A Lagrangian vector subspace of $V$ is an isotropic subspace whose dimension is the highest possible, equal to half the dimension of $V$.

\par\smallskip\noindent
{\bf 5.\quad} Let $W$ be a symplectic vector subspace of $V$. Since 
$W\cap\orth W=\{0\}$, the rank of the restriction to $W\times W$ of the form $\omega$ is equal to $\dim W$; therefore $\dim W$ is even and, equipped with the restriction of 
$\omega$, $W$ is a symplectic vector space. Conversely if, when equipped with the restriction of $\omega$, a vector subspace $W$ of $V$ is a symplectic vector space, 
we have $W\oplus\orth W=V$, and $W$ is a symplectic vector subspace of $V$ in the sense of the above definition.

\par\smallskip\noindent
{\bf 6.\quad} A vector subspace $W$ of $V$ is symplectic if and only if $\orth W$ is symplectic.

\begin{defis} Let $(M,\omega)$ be a symplectic manifold. For each $x\in M$, 
$\bigl(T_xM,\omega(x)\bigr)$ is a symplectic vector space.
A submanifold $N$ of $M$ is said to be

\begin{itemize}

\item{} \emph{isotropic} if for each $x\in N$, $T_xN$ is an isotropic vector subspace of the symplectic vector space
$\bigl(T_xM,\omega(x)\bigr)$,

\item{} \emph{coisotropic} if for each $x\in N$, $T_xN$ is a coisotropic vector subspace of $\bigl(T_xM,\omega(x)\bigr)$, 

\item{} \emph{Lagrangian} if for each $x\in N$, $T_xN$ is a Lagrangian vector subspace of $\bigl(T_xM,\omega(x)\bigr)$,

\item{} \emph{symplectic} if for each $x\in N$, $T_xN$ is a symplectic vector subspace of $\bigl(T_xM,\omega(x)\bigr)$.
\end{itemize}
\end{defis}

\subsection{Hamiltonian vector fields on a symplectic manifold}

Let $(M,\omega)$ be a symplectic manifold. We have seen that the map which associates to each vector $v\in TM$ the covector $\mathi(v)\omega$ is an isomorphism from $TM$ onto $T^*M$. So, for any given differential one-form 
$\alpha$, there exists a unique vector field $X$ such that $\mathi(X)\omega=\alpha$. We are therefore allowed to state the following definitions.

\begin{defis}
Let $(M,\omega)$ be a symplectic manifold and $f:M\to\RR$ be a smooth function.
The vector field $X_f$ which satisfies
 $$\mathi(X_f)\omega=-\d f$$
is called the \emph{Hamiltonian vector field} associated to $f$. The function $f$ is called a \emph{Hamiltonian} for the Hamiltonian vector field $X_f$.
\par\smallskip
A vector field $X$ on $M$ such that the one-form $\mathi(X)\omega$ is closed,
 $$\d\mathi(X)\omega=0\,,$$
is said to be \emph{locally Hamiltonian.} 
\end{defis}

\begin{rmks}
The function $f$ is not the unique Hamiltonian of the Hamiltonian vector field 
$X_f$: any function $g$ such that $\mathi(X_f)\omega=-dg$ is another Hamiltonian for $X_f$. Given a Hamiltonian $f$ of $X_f$, a function $g$ is another Hamiltonian for $X_f$ if and only if $\d(f-g)=0$, or in other words if and only if $f-g$ keeps a constant value on each connected component of $M$
\par\smallskip

Of course, a Hamiltonian vector field is locally Hamiltonian. The converse is not true when the cohomology space $H^1(M,\RR)$ is not trivial. 
\end{rmks}

\begin{prop}\label{LocallyHamiltonian} 
On a symplectic manifold $(M,\omega)$, a vector field $X$ is locally
Hamiltonian if and only if the Lie derivative 
${\mathcal L}(X)\omega$ of the symplectic form
$\omega$ with respect to $X$ vanishes:
 $${\mathcal L}(X)\omega=0\,.$$

The bracket $[X,Y]$ of two locally Hamiltonian vector fields $X$ and $Y$ is Hamiltonian, and has as a Hamiltonian the function $\omega(X,Y)$.
\end{prop}

\begin{proof}
The well known formula which relates the the exterior differential $\d$,
the interior product $\mathi(X)$ and the Lie derivative 
${\mathcal L}(X)$ with respect to the vector field $X$
 $${\mathcal L}(X)=\mathi(X)\d+\d\mathi(X)$$
proves that when $X$ is a vector field on a symplectic manifold $(M,\omega)$
 $${\mathcal L}(X)\omega = \d\,\mathi(X)\omega\,,$$
since $\d\,\omega=0$. Therefore $\mathi(X)\omega$ is closed if and 
only if ${\mathcal L}(X)\omega=0$.
\par\smallskip

Let $X$ and $Y$ be two locally Hamiltonian vector fields. We have
 \begin{equation*}
 \begin{split}
 \mathi\bigl([X,Y]\bigr)\omega&={\mathcal L}(X)\mathi(Y)\omega-\mathi(Y){\mathcal L}(X)\omega\\
                           &={\mathcal L}(X)\mathi(Y)\omega\\
                  &=\bigl(\mathi(X)\d +\d\,\mathi(X)\bigr)\mathi(Y)\omega\\
                  &=\d\,\mathi(X)\mathi(Y)\omega\\
                  &=-d\bigl(\omega(X,Y)\bigr)\,,
 \end{split}
 \end{equation*}      
which proves that $\omega(X,Y)$ is a Hamiltonian for $[X,Y]$.
\end{proof}

\subsubsection{Expression in a system of Darboux local coordinates}
Let $(x^1,\ldots,x^{2n})$ be a system of Darboux local coordinates. The symplectic form $\omega$ can be locally writen as
 $$\omega=\sum_{i=1}^n\d\,x^{n+i}\wedge\d\,x^i\,,$$
so we see that the Hamiltonian vector field $X_f$ associated to a smooth 
function $f$ can be locally written as
 $${X}_f=\sum_{i=1}^n\frac{\partial f}{\partial x^{n+i}}\,\frac{\partial}{\partial x^i}
                  -\frac{\partial f}{\partial x^{i}}\,\frac{\partial}{\partial x^{n+i}}\,.
 $$
A smooth curve $\varphi$  drawn in $M$ parametrized by the real variable $t$ is said to be a \emph{solution} of the 
differential equation determined by ${X}_f$, or an \emph{integral curve} of $X_f$,
if it satisfies the equation, called the 
\emph{Hamilton equation} for the Hamiltonian $f$,
 $$\frac{\d\varphi(t)}{\d t}=X_f\bigl(\varphi(t)\bigr)\,.$$
Its local expression in the considered system of Darboux local coordinates is 
\begin{equation*}
\left\{
\begin {aligned}
 \frac{\d x^{i}}{\d t}&=\frac{\partial f}{\partial x^{n+i}}\,,\\
 \frac{\d x^{n+i}}{\d t}&=-\frac{\partial f}{\partial x^{i}}\,,\\
 \end{aligned}
 \qquad(1\leq i\leq n)\,.
\right.
\end{equation*}

\begin{defi}\label{CanonicalLiftDiffeo}
Let $\Phi:N\to N$ be a diffeomorphism of a smooth manifold $N$ onto itself.
The \emph{canonical lift} of $\Phi$ to the cotangent bundle is the transpose
of the vector bundles isomorphism $T(\Phi^{-1})=(T\Phi)^{-1}:TN\to TN$. In other words, denoting by $\widehat \Phi$ the canonical lift of $\Phi$ to the cotangent bundle, we have for all $x\in N$, $\xi\in T^*_xN$ and 
$v\in T_{\Phi(x)}N$,
 $$\bigl\langle\widehat\Phi(\xi),v\bigr\rangle
   =\bigl\langle\xi, (T\Phi)^{-1}(v)\bigr\rangle\,.
 $$ 
\end{defi}

\begin{rmk}
With the notations of Definition \ref{CanonicalLiftDiffeo}, we have
$\pi_N\circ\widehat \Phi=\Phi\circ \pi_{N}$.
\end{rmk}

\subsubsection{The flow of a vector field} Let $X$ be a smooth vector field on a smooth manifold $M$. We recall that the \emph{reduced flow}
of $X$ is the map $\Phi$, defined on an open subset $\Omega$ of 
$\RR\times M$ and taking its values in $M$, such that for each $x\in M$
the parametrized curve $t\mapsto \varphi(t)=\Phi(t,x)$ is the maximal integral curve of the differential equation
 $$\frac{\d\varphi(t)}{\d t}=X\bigl(\varphi(t))$$
 which satisfies
$\varphi(0)=x$. 
For each $t\in \RR$, the set $D_t=\{x\in M;(t,x)\in\Omega\}$ 
is an open subset of $M$ and when $D_t$ is not empty the map 
$x\mapsto \Phi_t(x)=\Phi(t,x)$ is a diffeomorphism of $D_t$ onto
$D_{-t}$. 

\begin{defis}\label{CanonicalLiftDef}
Let $N$ be a smooth manifold, $TN$ and $T^*N$ be its tangent and cotangent bundles, $\tau_N:TN\to N$ and $\pi_N:T^*N\to N$ be their canonical projections. Let $X$ be a smooth vector field on $N$ and 
$\{\Phi^X_t\,;t\in\RR\}$ be its reduced flow. 

\par\smallskip\noindent
{\bf 1.\quad} The \emph{canonical lift} of $X$ to the tangent bundle $TN$ 
is the unique vector field $\overline X$ on $TM$ whose reduced flow 
$\{\Phi^{\overline X}_t\,;t\in \RR\}$ is the prolongation to vectors of the reduced flow of $X$. In other words, for each $t\in\RR$,
 $$\Phi^{\overline X}_t=T\Phi^X_t\,,$$
therefore, for each $v\in TN$,
 $$\overline X(v)=\frac{\d}{\d t}\bigl(T\Phi^X_t(v)\bigr)\bigm|_{t=0}\,.
 $$

\par\smallskip\noindent
{\bf 2.\quad}  The \emph{canonical lift} of $X$ to the cotangent bundle 
$T^*N$  is the unique vector field $\widehat X$ on $T^*M$ whose reduced flow
$\{\Phi^{\widehat X}_t\,;t\in \RR\}$ is the lift to the cotangent bundle
of the reduced flow $\{\Phi^X_t\,;t\in\RR\}$ of $X$. In other words, for each $t\in\RR$,
 $$\Phi^{\widehat X}_t=\widehat{\Phi^X_t}\,,$$ 
therefore, for each $\xi\in T^*N$,
$$\widehat X(\xi)=\frac{\d}{\d t}\bigl(\widehat{\Phi^{X}_t}(\xi)\bigr)\bigm|_{t=0}\,.
 $$
\end{defis}

\begin{rmk}
Let $X$ be a smooth vector field defined on a smooth manifold $N$. Its canonical lift
$\overline X$ to the tangent bundle $TN$ (\ref{CanonicalLiftDef}) is related to the
prolongation to vectors $TX:TN\to T(TN)$ by the formula
 $$\overline X=\kappa_N\circ TX\,,$$ 
where $\kappa_N:T(TN)\to T(TN)$ is the canonical involution of the tangent bundle to $TN$ (see \cite{Tulczyjew89}).   
\end{rmk}

\begin{prop}\label{PropCanonicalLifts}
Let $\Phi:N\to N$ be a diffeomorphism of a smooth manifold $N$ onto itself
and $\widehat \Phi:T^*N\to T^*N$ the canonical lift of $\Phi$ to the cotangent bundle. Let $\eta_N$ be the Liouville form on $T^*N$. We have
 $$\widehat\Phi^*\eta_N=\eta_N\,.$$

Let $X$ be a smooth vector field on $N$, and
$\widehat X$ be the canonical lift of $X$ to the cotangent bundle. We have
 $${\mathcal L}(\widehat X)(\eta_N)=0\,.$$
\end{prop}

\begin{proof}
Let $\xi\in T^*N$ and $v\in T_\xi(T^*N)$. We have
 $$\widehat\Phi^*\eta_N(v)=\eta_N\bigl(T\widehat\Phi(v)\bigr)
   =\bigl\langle \tau_{T^*N}\circ T\widehat\Phi(v),T\pi_N\circ T\widehat \Phi(v)\bigr\rangle\,.
 $$
But $\tau_{T^*N}\circ T\widehat\Phi=\widehat\Phi\circ \tau_{T^*N}$ and
$T\pi_N\circ T\widehat\Phi=T(\pi_N\circ\widehat\Phi)=T(\Phi\circ \pi_N)$.
Therefore
 $$\widehat\Phi^*\eta_N(v)=
    \bigl\langle \widehat\Phi\circ \tau_{T^*N}(v),
     T(\Phi\circ \pi_N)(v)\bigr\rangle
     =\bigl\langle\tau_{T^*N}(v),T\pi_N(v)\bigr\rangle=\eta_N(v)
 $$
since $\widehat\Phi=(T\Phi^{-1})^T$.
Now let $X$ be a smooth vector field on $N$, $\{\Phi^X_t\,;t\in\RR\}$
be its reduced flow, and $\widehat X$ be the canonical lift of $X$ to the cotangent bundle. We know that the reduced flow of $\widehat X$ is
$\{\widehat{\Phi^X_t}\,;t\in\RR\}$, so we can write
 $${\mathcal L}(\widehat X)\eta_N=\frac{\d}{\d t}
      \bigl({\widehat{\Phi^X_t}}^*\eta_N\bigr)\bigm|_{t=0}\,.
 $$
Since $\widehat{\Phi^X_t}^*\eta_N=\eta_N$ does not depend on $t$,
${\mathcal L}(\widehat X)\eta_N=0$.
\end{proof}

The following Proposition, which presents an important example of 
Hamiltonian vector field on a cotangent bundle, will be used when we will consider Hamiltonian actions of a Lie group on its cotangent bundle.

\begin{prop}\label{CanonicalLiftHam}
Let $N$ be a smooth manifold, $T^*N$ be its cotangent bundle, $\eta_N$ be the Liouville form and $\d\eta_N$ be the canonical symplectic form on $T^*N$.
Let $X$ be a smooth vector field on $N$ and $f_X:T^*N\to \RR$ be
the smooth function defined by
 $$f_X(\xi)=\Bigl\langle\xi,X\bigl(\pi_N(\xi)\bigr)\Bigr\rangle\,,
    \quad\xi\in T^*N\,.$$
On the symplectic manifold $(T^*N,\d\eta_N)$, the vector field $\widehat X$,
canonical lift to $T^*N$ of the vector field $X$ on $N$ in the sense defined above (\ref{CanonicalLiftDef}), is a Hamiltonian field which has the function $f_X$ as a Hamiltonian. In other words
 $$\mathi(\widehat X)\d \eta_N=-\d f_X\,.$$
Moreover,
 $$f_X=\mathi(\widehat X)\eta_N\,.$$
\end{prop}

\begin{proof}
We have seen (Proposition \ref{PropCanonicalLifts}) that 
${\mathcal L}(\widehat X)\eta_N=0$. Therefore
 $$\mathi(\widehat X)\d \eta_N={\mathcal L}(\widehat X)\eta_N-
   \d\mathi(\widehat X)\eta_N=-\d\mathi(\widehat X)\eta_N\,,$$
which proves that $\widehat X$ is Hamiltonian and admits
$\mathi(\widehat X)\eta_N$ as Hamiltonian. For each $\xi\in T^*N$
 $$\mathi(\widehat X)\eta_N(\xi)=\eta_N(\widehat X)(\xi)
   =\Bigl\langle\xi,T\pi_N\bigl(\widehat X(\xi)\bigr)\Bigr\rangle
   =\Bigl\langle\xi,X\bigl(\pi_N(\xi)\bigr)\Bigr\rangle=f_X(\xi)\,.\qedhere
 $$
\end{proof}

\subsection{The Poisson bracket}

\begin{defi} The \emph{Poisson bracket} of an ordered pair
$(f,g)$ of smooth functions defined on the symplectic manifold $(M,\omega)$ is 
the smooth function $\{f,g\}$ defined by the equivalent formulae
 $$\{f,g\}=\mathi({ X}_f)\,dg=-\mathi({X}_g)\,df=\omega({ X}_f,{ X}_g)\,,$$
where $X_f$ and $X_g$ are the Hamiltonian vector fields on $M$ with, respectively,
the functions $f$ and $g$ as Hamiltonian.
\end{defi}

\begin{lemme}\label{LemmaPoissonBracket}
Let $(M,\omega)$ be a symplectic manifold, let $f$ and $g$ be two smooth functions on $M$ and let $X_f$ and $X_g$ be the associated Hamiltonian vector fields.
The bracket $[X_f,X_g]$ is a Hamiltonian vector field which admits $\{f,g\}$ as Hamiltonian.
\end{lemme}

\begin{proof} This result is an immediate consequence of Proposition 
\ref{LocallyHamiltonian}.
\end{proof}

\begin{prop}
Let $(M,\omega)$ be a symplectic manifold. The Poisson bracket is a bilinear
composition law on the space $C^\infty(M,\RR)$ of smooth functions on $M$, which satisfies the following properties.

\noindent
{\rm 1.} It is skew-symmetric: $\{g,f\}=-\{f,g\}$.

\noindent
{\rm 2.} It satisfies the Leibniz identity with respect to the ordinary product 
of functions:
 $$\{f,gh\}=\{f,g\}h+g\{f,h\}\,.$$
{\rm 3.} It satisfies the Jacobi identity, which is a kind of Leibniz identity 
with respect to the Poisson bracket itself:
 $$\bigl\{f,\{g,h\}\bigr\}=\bigl\{\{f,g\},h\bigr\}+\bigl\{g,\{f,h\}\bigr\}\,,$$
which can also be written, when the skew-symmetry of the Poisson bracket 
is taken into account,
 $$\bigl\{\{f,g\},h\bigr\}+\bigl\{\{g,h\},f\bigr\}+\bigl\{\{h,f\},g\bigr\}=0\,.$$
\end{prop}

\begin{proof} The proofs of Properties (i) and (ii) are very easy and left to the reader. Let us proove Property (iii).

We have
 $$\bigl\{\{f,g\},h\bigr\}=\omega(X_{\{f,g\}}, X_h)
                          =-\mathi(X_{\{f,g\}})\mathi(X_h)\omega  
                          =\mathi(X_{\{f,g\}})\d h\,.
 $$
By Lemma \ref{LemmaPoissonBracket}, $X_{\{f,g\}}=[X_f,X_g]$ so we have
  $$\bigl\{\{f,g\},h\bigr\}=\mathi\bigl([X_f,X_g]\bigr)\d h
     ={\mathcal L}\bigl([X_f,X_g]\bigr)h\,.
 $$
 We also have

$$\bigl\{\{g,h\},f\bigr\}=-{\mathcal L}(X_f)\circ{\mathcal L}(X_g)h\,,\quad
   \bigl\{\{h,f\},g\bigr\}={\mathcal L}(X_g)\circ{\mathcal L}(X_f)h\,.$$
Taking the sum of these three terms, and taking into account the identity
 $${\mathcal L}\bigl([X_f,X_g]\bigr)={\mathcal L}(X_f)\circ {\mathcal L}(X_g)
                                      -{\mathcal L}(X_g)\circ {\mathcal L}(X_f)\,,
 $$
we see that the Jacobi identity is satisfied.
\end{proof}

\begin{rmks}\hfill
\par\smallskip\noindent
{\bf 1.\quad} In a system of Darboux local coordinates $(x^1,\ldots,x^{2n})$, the Poisson bracket can be written
 $$\{f,g\}=\sum_{i=1}^n\left(
           \frac{\partial f}{\partial x^{n+i}}\, \frac{\partial g}
           {\partial x^i}-
           \frac{\partial f}{\partial x^{i}}\, \frac{\partial g}
           {\partial x^{n+i}}\right)\,.
 $$
\noindent
{\bf 2.\quad} 
Let $H$ be a smooth function on the symplectic manifold $(M,\omega)$,
and $X_H$ be the associaled Hamiltonian vector field. By using the Poisson bracket, one can write in a very concise way the Hamilton 
equation for $X_H$. Let $t\mapsto \varphi(t)$ be any integral curve of $X_H$.
Then for any smooth function $f:M\to\RR$  
 $$\frac{\d f\bigl(\varphi(t)\bigr)}{\d t}=\{H,f\}\bigl(\varphi(t)\bigr)\,.$$
By succesively taking for $f$ the coordinate functions 
$x^1,\ldots,x^{2n}$ of a system of Darboux local coordinates, we recover the equations 
\begin{equation*}
\left\{
\begin {split}
 \frac{\d x^{i}}{\d t}&=\frac{\partial H}{\partial x^{n+i}}\,,\\
 \frac{\d x^{n+i}}{\d t}&=-\frac{\partial H}{\partial x^{i}}\,,\\
 \end{split}
 \qquad(1\leq i\leq n)\,.
\right.
 \end{equation*}
\end{rmks}

\section{Poisson manifolds}\label{PoissonManifolds}

\subsection{The inception of Poisson manifolds}
Around the middle of the XX-th century, several scientists
felt the need of a frame in which Hamiltonian differential equations 
could be considered,  more general than that of symplectic manifolds.
Paul Dirac for example proposed such a frame in his famous 1950 paper
\emph{Generalized Hamiltonian dynamics} \cite{Dirac50, Dirac64}. 
\par\smallskip

In many applications in which, starting from a symplectic manifold, 
another manifold is built by a combination of processes (products, quotients, restriction to a submanifold, $\ldots$), there exists on that
manifold a structure, more general than a symplectic structure, with which a vector field can be associated to each smooth function, and the bracket of two smooth functions can be defined. It was also known that on a (odd-dimensional) contact manifold one can define the bracket of
two smooth functions.
\par\smallskip

Several generalizatons of symplectic manifolds were defined and
investigated by Andr\'e Lichnerowicz during the years 1975--1980. He gave several names to these generalizations:  
\emph{canonical, Poisson, 
Jacobi} and \emph{locally conformally symplectic} manifolds \cite{Lichnerowicz77, Lichnerowicz79}.
\par\smallskip

In 1976 Alexander Kirillov published a paper entitled \emph{Local Lie algebras} \cite{Kirillov76} in which he determined all the possible structures 
on a manifold allowing the definition of a bracket with which 
the space of smooth functions becomes a local Lie algebra.
\emph{Local} means that the value taken by the bracket of two smooth functions at each point only depends of the values taken by these functions on an arbitrarily small neighbourhood of that point.
The only such structures are those called by Lichnerowicz
\emph{Poisson structures}, \emph{Jacobi structures} and 
\emph{locally conformally symplectic structures}. 
\par\smallskip

In what follows we will mainly consider Poisson manifolds.

\subsection{Definition and structure of Poisson manifolds}

\begin{defi}
A \emph{Poisson structure} on a smooth manifold $M$ 
is the structure determined by a bilinear, skew-symmetric 
composition law on the space of smooth functions, called the
\emph{Poisson bracket} and denoted by 
 $(f,g)\mapsto\{f,g\}$, satisfying the Leibniz identity
 $$\{f,gh\}=\{f,g\}h+g\{f,h\}$$
and the Jacobi identity
 $$\bigl\{\{f,g\},h\bigr\}+\bigl\{\{g,h\},f\bigr\}+
 \bigl\{\{h,f\},g\bigr\}=0\,.$$
A manifold endowed with a Poisson structure is called a \emph{Poisson manifold}.
\end{defi}

\begin{prop}
On a Poisson manifold $M$, there exists a unique smooth bivector field $\Lambda$,
called the \emph{Poisson bivector field} of $M$,
such that for any pair $(f,g)$ of smooth functions defined on $M$, the Poisson bracket $\{f,g\}$ is given by the formula
$$\{f,g\}=\Lambda(\d f,\d g)\,.$$ 
\end{prop}

\begin{proof}
The existence, uniqueness and skew-symmetry of $\Lambda$ are easy consequences 
of the the Leibniz identity and of the skew-symmetry
of the Poisson bracket. It does not depend on the Jacobi identity.
\end{proof}

\begin{rmk}
The Poisson bivector field $\Lambda$ determines the Poisson structure of $M$, since it allows the calculation of the Poisson bracket of any pair of smooth functions. For this reason a Poisson manifold $M$ is often denoted by $(M,\Lambda)$. 
\end{rmk}

\begin{defi}
Let $(M,\Lambda)$ be a Poisson manifold. We denote by $\Lambda^\sharp:T^*M\to TM$
the vector bundle homomorphism such that, for each $x\in M$ and each 
$\alpha\in T^*_xM$, $\Lambda^\sharp(\alpha)$ is the unique element in $T_xM$ 
such that, for any $\beta\in T^*_xM$,
 $$\bigl\langle\beta,\Lambda^\sharp(\alpha)\bigr\rangle=\Lambda(\alpha,\beta)\,.$$
The subset $C=\Lambda^\sharp(T^*M)$ of the tangent bundle $TM$ is called the
\emph{characteristic field} of the Poisson manifold $(M,\Lambda)$
\end{defi}

The following theorem, due to Alan Weinstein \cite{Weinstein83}, proves that,
loosely speaking, a Poisson manifold is the disjoint union of symplectic manifolds, arranged in such a way that  the union is endowed with a differentiable structure.

\begin{theo} Let $(M,\Lambda)$ be a Poisson manifold. Its characteristic field $C$
is a completely integrable generalized distribution on $M$. It means that $M$
is the disjoint union of immersed connected submanifolds, called the 
\emph{symplectic leaves}
of $(M,\Lambda)$, with the following properties: a leaf $S$ is such that, for each
$x\in S$, $T_xS=T_xM\cap C$; moreover, $S$ is maximal in the sense that any immersed connected submanifold $S'$ containing $S$ and such that for each $x\in S'$,
$T_xS'=T_xM\cap C$, is equal to $S$.
\par\smallskip

Moreover, the Poisson structure of $M$ determines, on each leaf $S$, a symplectic form $\omega_S$, such that the restriction to $S$ of the Poisson bracket of two smooth functions defined on $M$ only depends on the restrictions of these functions to $S$, and can be calculated as  the Poisson bracket of these restrictions, using the symplectic form $\omega_S$. 
\end{theo}

The reader may look at \cite{Weinstein83} or at \cite{LibermannMarle} for a 
proof of this theorem.

\subsubsection{The Schouten-Nijenhuis bracket}\label{Schouten-Nijenhuis} 
Let $M$ be a smooth $n$-dimensional manifold. We recall that 
the exterior algebra  $\Omega(M)$ of differential forms on $M$ 
is endowed with an associative composition law, the 
\emph{exterior product}, which associates to a pair $(\eta,\zeta)$, 
with $\eta\in\Omega^p(M)$ and 
$\zeta\in \Omega^q(M)$ the form 
$\eta\wedge\zeta\in\Omega^{p+q}(M)$, with the following properties.

\par\smallskip\noindent
{\rm 1.} When $p=0$, $\eta\in\Omega^0(M)\equiv C^\infty(M,\RR)$; the exterior
product $\eta\wedge\zeta$
is the usual product $\eta\zeta$ of the differential form $\zeta$ of degree $q$
by the function $\eta$.

\par\smallskip\noindent
{\rm 2.} The exterior product satisfies
 $$\zeta\wedge\eta=(-1)^{pq}\eta\wedge\zeta\,.
 $$
\noindent
{\rm 3.} When $p\geq 1$ and $q\geq 1$, $\eta\wedge\zeta$ evaluated on the $p+q$
vector fields $v_i\in A^1(M)$ ($1\leq i\leq p+q)$ is expressed as 
 \begin{equation*}\eta\wedge\zeta(v_1,\ldots,v_{p+q})
 =\sum_{\sigma\in{\mathcal S}_{(p,q)}}\varepsilon(\sigma)
 \eta(v_{\sigma(1)},\ldots,v_{\sigma(p)})\zeta(v_{\sigma(p+1)},\ldots,
 v_{\sigma(p+q)})\,.
 \end{equation*}
We have denoted by ${\mathcal S}_{(p,q)}$ the set of permutations
$\sigma$ of $\{\,1,\ldots,p+q\,\}$ which satisfy
 \begin{equation*}\sigma(1)<\sigma(2)<\cdots<\sigma(p)\quad\hbox{and}\quad
 \sigma(p+1)<\sigma(p+2)<\cdots<\sigma(p+q)\,,\end{equation*}
and set
 \begin{equation*}\epsilon(\sigma)=\begin{cases}
                    1&\text{if $\sigma$ is even},\\
                    -1&\text{if $\sigma$ is odd}.
                    \end{cases}
 \end{equation*}
\par\smallskip

The exterior algebra $\Omega(M)$ is endowed with a linear map
$\d:\Omega(M)\to\Omega(M)$ called the \emph{exterior differential}, with the following properties.
 
\par\smallskip\noindent
{\rm 1.} The exterior differential $\d$ is a graded map of degree 1, 
which means that $\d\bigl(\Omega^p(M)\bigr)\subset\Omega^{p+1}(M)$.

\par\smallskip\noindent
{\rm 2.} It is a \emph{derivation} of the exterior algebra $\Omega(M)$, 
which means that when $\eta \in\Omega^p(M)$ and $\zeta\in\Omega^q(M)$,
 $$\d(\eta\wedge\zeta)=(\d\eta)\wedge\zeta +(-1)^p\eta\wedge\d\zeta\,.$$
\noindent
{\rm 3.} It satisfies 
 $$\d\circ \d=0\,.$$
\par\smallskip

Similarly, the exterior algebra $A(M)$ of smooth multivector 
fields on $M$ is endowed with an associative composition law, 
the \emph{exterior product}, which associates to a pair $(P,Q)$, 
with $P\in A^p(M)$ and 
$Q\in A^q(M)$, the multivector field $P\wedge Q\in A^{p+q}(M)$.
It is defined by the formulae given above for the exterior product of differential
forms, the only change being the exchange of the
roles of $\Omega^p(M)$ and $A^p(M)$. Its properties are essentially 
the same as those of the exterior product of differential forms.
\par\smallskip

There is a natural pairing of elements of same degree in $A(M)$ and in $\Omega(M)$. 
It is first defined for decomposable elements: let 
$\eta=\eta_1\wedge\cdots\wedge \eta_p\in \Omega^p(M)$ and
$P=X_1\wedge\cdots\wedge X_p\in A^p(M)$. We set
 $$\langle\eta,P\rangle=\det\bigl(\langle\eta_i,X_j\rangle\bigr)\,.$$
Then this pairing can be uniquely extended to $\Omega^p(M)\times A^p(M)$ by bilinearity.
\par\smallskip

With any $P\in A^p(M)$ we can associate a graded endomorphism $\mathi(P)$ of the
exterior algebra of differential forms $\Omega(M)$, of degree $-p$, which means that when $\eta\in\Omega^q(M)$, $\mathi(P)\eta\in\Omega^{q-p}(M)$. This endomorphism, which extends to multivector fields the interior product of forms with a vector
field, is determined by the formula, in which $P\in A^p(M)$, $\eta\in \Omega^q(M)$
and $R\in A^{q-p}(M)$,
 $$\bigl\langle\mathi(P)\eta,R\bigr\rangle
   =(-1)^{(p-1)p/2}\langle\eta,P\wedge Q\rangle\,.
 $$
Besides the exterior product, there exists on the graded vector space $A(M)$ of multivector fields another bilinear composition law, which naturally extends to multivector fields the Lie bracket of vector fields. It
associates to $P\in A^p(M)$ and $Q\in A^q(M)$ an element denoted 
$[P,Q]\in A^{p+q-1}(M)$, called the \emph{Schouten-Nijenhuis bracket} 
of $P$ and $Q$.
The Schouten-Nijenhuis bracket $[P,Q]$ is defined by the following formula, which gives the expression of the corresponding graded endomorphism of $\Omega(M)$,
 $$\mathi\bigl([P,Q]\bigr)=\Bigl[\bigl[\mathi(P),\d\bigr],\mathi(Q)\Bigr]\,.
 $$
The brackets in the right hand side of this formula are the 
\emph{graded commutators} of graded endomorphisms of $\Omega(M)$. Let us recall that
if $E_1$ and $E_2$ are graded endomorphisms of $\Omega(M)$ of degrees $e_1$ and $e_2$ respectively, their graded commutator is
 $$[E_1,E_2]=E_1\circ E_2-(-1)^{e_1e_2}E_2\circ E_1\,.$$
The following properties of the Schouten-Nijenhuis bracket can be deduced 
from the above formulae. 

\par\smallskip\noindent
{\rm 1.} For $f$ and $g\in A^0(M)=C^\infty(M,\RR)$, $[f,g]=0$.

\par\smallskip\noindent
{\rm 2.} For a vector field $V\in A^1(M)$, $q\in\ZZ$ and $Q\in A^q(M)$,
the Schouten-Nijenhuis bracket $[V,Q]$ is the Lie derivative ${\mathcal L}(V)(Q)$.

\par\smallskip\noindent
{\rm 3.} For two vector fields $V$ and $W \in A^1(M,E)$, the Schouten-Nijenhuis bracket $[V,W]$ is the usual Lie bracket of these vector fields. 

\par\smallskip\noindent
{\rm 4.} For all $p$ and $q\in\ZZ$, $P\in A^p(M)$, $Q\in A^q(M)$,
 $$[P,Q]=-(-1)^{(p-1)(q-1)}[Q,P]\,.
 $$

\noindent
{\rm 5.} Let $p\in\ZZ$, $P\in A^p(M)$. The map $Q\mapsto [P,Q]$
is a derivation of degree $p-1$ of the graded exterior algebra
$A(M)$. In other words, for $q_1$ and $q_2\in\ZZ$, $Q_1\in
A^{q_1}(M)$ and $Q_2\in A^{q_2}(M)$,
 $$[P,Q_1\wedge Q_2]=[P,Q_1]\wedge
 Q_2+(-1)^{(p-1)q_1}Q_1\wedge[P,Q_2]\,.
 $$

\noindent
{\rm 6.} Let $p$, $q$ and $r\in\ZZ$, $P\in A^p(M)$, 
$Q\in
A^q(M)$ and $R\in A^r(M)$. The Schouten-Nijenhuis bracket
satisfies the \emph{graded Jacobi identity}
 \begin{align*}
 (-1)^{(p-1)(r-1)}\bigl[[P,Q],R\bigr]
 &+(-1)^{(q-1)(p-1)}\bigl[[Q,R],P\bigr]\\
 &+(-1)^{(r-1)(q-1)}\bigl[[R,P],Q\bigr]\\
 &=0\,.
 \end{align*}
For more information about the Schouten-Nijenhuis bracket, the reader may look at 
\cite{Koszul85} or \cite{Marle2008}.

\begin{prop}
Let $\Lambda$ be a smooth bivector field on a smooth manifold $M$. Then $\Lambda$
is a Poisson bivector field (and $(M,\Lambda)$ is a Poisson manifold) if and only if
$[\Lambda,\Lambda]=0$. 
\end{prop}

\begin{proof}
We define the vector bundle homomorphism $\Lambda^\sharp:T^*M\to TM$ by setting, for
all $x\in M$, $\alpha$ and $\beta\in T^*_xM$,
 $$\bigl\langle\beta,\Lambda^\sharp(\alpha)\bigr\rangle=\Lambda(\alpha,\beta)\,.$$
For any pair $(f,g)$ of smooth functions we set
 $$X_f=\Lambda^\sharp(df)\,,\quad \{f,g\}=\mathi(X_f)(\d g)=\Lambda(\d f,\d g)\,.$$
This bracket is a bilinear skew-symmetric composition law on $C^\infty(M,\RR)$ which satisfies the Leibniz identity. Therefore $\Lambda$ is a Poisson bivector field if and only if the above defined bracket of functions satisfies the Jacobi identity. 
\par\smallskip
Let $f$, $g$ and $h$ be three smooth functions on $M$. We
easily see that $X_f$ and $\{f,g\}$ can be expressed in terms of the Schouten-Nijenhuis bracket. Indeed we have
 $$X_f=-[\Lambda,f]=-[f,\Lambda]\,,\quad\{f,g\}=\bigl[[\Lambda,f],g\bigr]\,.$$
Therefore
 $$\bigl\{\{f,g\},h\bigr\}=\Bigg[\Bigl[\Lambda,\bigl[[\Lambda,f],g\bigr]\Bigr],h\Bigg]\,.
 $$
By using the graded Jacobi identity satisfied by Schouten-Nijenhuis bracket, we see that
 $$\Bigl[\Lambda,\bigl[[\Lambda,f],g\bigr]\Bigr]
    =-\bigl[[g,\Lambda],[f,\Lambda]\bigr]
      +2\Bigl[\bigl[ [\Lambda,\Lambda],f\bigr],g\Bigr]\,.
 $$
Using the equalities $X_f=-[\Lambda,f]=-[f,\Lambda]$ and
$X_g=-[\Lambda,g]=-[g,\Lambda]$ we obtain
 \begin{equation*}
  \begin{split}
    \bigl\{\{f,g\},h\bigr\}&=\bigl[[X_f,X_g],h\bigr]
    +2\Biggl[\Bigl[\bigl[ [\Lambda,\Lambda],f\bigr],g\Bigr],h\Biggr]\\
    &={\mathcal L}\bigl([X_f,X_g]\bigr)h
      +2\Biggl[\Bigl[\bigl[ [\Lambda,\Lambda],f\bigr],g\Bigr],h\Biggr]\,.   
\end{split}
\end{equation*}
On the other hand, we have
 $$\bigl\{\{g,h\},f\bigr\}=-{\mathcal L}(X_f)\circ{\mathcal L}(X_g)h\,,\quad
   \bigl\{\{h,f\},g\bigr\}={\mathcal L}(X_g)\circ{\mathcal L}(X_f)h\,.$$
Taking into account the equality
 $${\mathcal L}\bigl([X_f,X_g]\bigr)={\mathcal L}(X_f)\circ {\mathcal L}(X_g)
                                      -{\mathcal L}(X_g)\circ {\mathcal L}(X_f)
 $$
we obtain
 $$\bigl\{\{f,g\},h\bigr\}+\bigl\{\{g,h\},f\bigr\}+\bigl\{\{h,f\},g\bigr\}
   =2\Biggl[\Bigl[\bigl[ [\Lambda,\Lambda],f\bigr],g\Bigr],h\Biggr]\,.$$
By using the formula which defines the Schouten-Nijenhuis bracket, we check that for any
$P\in A^3(M)$
 $$\Bigl[\bigl[[P,f],g\bigr],h\Bigr]=P(df,dg,dh)\,.$$
Therefore
$$\bigl\{\{f,g\},h\bigr\}+\bigl\{\{g,h\},f\bigr\}+\bigl\{\{h,f\},g\bigr\}
  =2[\Lambda,\Lambda](df,dg,dh)\,,$$
so $\Lambda$ is a Poisson bivector field if and only if $[\Lambda,\Lambda]=0$. 
\end{proof}

\subsection{Some properties of Poisson manifolds}

\begin{defis} Let $(M,\Lambda)$ be a Poisson manifold.

\par\smallskip\noindent
{\bf 1.\quad} The \emph{Hamiltonian vector field} associated to a smooth function 
$f\in C^\infty(M,\RR)$ is the vector field $X_f$ on $M$ defined by
 $$X_f=\Lambda^\sharp(\d f)\,.$$
The function $f$ is called a \emph{Hamiltonian} for the Hamiltonian vector field
$X_f$.

\noindent 
{\bf 2.\quad} A \emph{Poisson vector field} is a vector field $X$ which satisfies
 $${\mathcal L}(X)\Lambda=0\,.$$
\end{defis}

\begin{ex} On a symplectic manifold $(M,\omega)$ we have defined the 
Poisson bracket of smooth functions. That bracket endows $M$ with a Poisson structure, said to be \emph{associated} to its symplectic structure.
The Poisson bivector field $\Lambda$ is related to the symplectic form $\omega$ by
 $$\Lambda(df,dg)=\omega(X_f,X_g)\,,\quad f\ \hbox{and }g\in C^\infty(M,\RR)\,.$$
The map $\Lambda^\sharp:T^*M\to TM$ such that, for any $x\in M$, $\alpha$ and $\beta\in T^*_xM$,
 $$\bigl\langle \beta,\Lambda^\sharp(\alpha)=\Lambda(\alpha,\beta)$$
is therefore the inverse of the map $\omega^\flat:TM\to T^*M$ such that, for any $x\in M$, $v$
and $w\in T_xM$,
 $$\bigl\langle\omega^\flat(v),w\bigr\rangle=-\bigl\langle i(v)\omega ,w\bigr\rangle
    =\omega(w,v)\,.$$ 
Hamiltonian vector fields for the symplectic structure of $M$ coincide with Hamiltonian vector fields for its Poisson structure.
The Poisson vector fields on the symplectic manifold $(M,\omega)$ are the locally Hamiltonian vector fields. However, on a general Poisson manifold, Poisson vector fields are more general than locally Hamiltonian vector fields: even restricted to an arbitrary small neighbourhood of a point, a Poisson vector field may not be
Hamiltonian.
\end{ex} 

\begin{rmks}\label{Casimirs}\hfill

\par\smallskip\noindent
{\bf 1.\quad} Another way in which the Hamiltonian vector field $X_f$ associated to a smooth function $f$ can be defined is by saying that, for any other smooth function $g$ on the Poisson manifold $(M,\Lambda)$,
 $$\mathi(X_f)(\d g)=\{f,g\}\,.$$

\noindent
{\bf 2.\quad} A smooth function $g$ defined on the Poisson manifold $(M,\Lambda)$ is said to be a  \emph{Casimir} if for any other smooth function $h$, we have
$\{g,h\}=0$. In other words, a Casimir is a smooth function $g$ whose associated
Hamiltonian vector field is $X_g=0$. On a general Poisson manifold, there may exist
Casimirs other than the locally constant functions.  

\par\smallskip\noindent
{\bf 3.\quad} A smooth vector field $X$ on the Poisson manifold $(M,\Lambda)$ is a Poisson vector field if and only if, for any pair $(f,g)$ of smooth functions,
 $${\mathcal L}(X)\bigl(\{f,g\}\bigr)=\bigl\{{\mathcal L}(X)f,g\bigr\}
    +\bigl\{f,{\mathcal L}(X)g\bigr\}\,.
 $$
Indeed we have
 \begin{equation*}
  \begin{split}
    {\mathcal L}(X)\bigl(\{f,g\}\bigr)&={\mathcal L}(X)\bigl(\Lambda(\d f,\d g)\bigr)\\
                                      &=\bigl({\mathcal L}(X(\Lambda)\bigr)(\d f,\d g)
                                    +\Lambda\bigl({\mathcal L}(X)(\d f), \d g\bigr)
                                    +\Lambda\bigl(\d f,{\mathcal L}(X)(\d g)\bigr)\\
       &=\bigl({\mathcal L}(X)(\Lambda)\bigr)(\d f,\d g)
         +\bigl\{{\mathcal L}(X)f,g\bigr\}
    +\bigl\{f,{\mathcal L}(X)g\bigr\}\,.
\end{split}
\end{equation*}

\noindent
{\bf 3.\quad} Any Hamiltonian vector field $X_f$ is a Poisson vector field. Indeed, if $f$ is a Hamiltonian for $X_f$, $g$ and $h$ two other smooth functions,we have according to the Jacobi identity
 \begin{equation*}
 \begin{split}
 {\mathcal L}(X_f)\bigl(\{g,h\}\bigr)&=\bigl\{f,\{g,h\}\bigr\}
   =\bigl\{\{f,g\},h\bigr\}+\bigl\{g,\{f,h\}\bigr\}\\
   &=\bigl\{{\mathcal L}(X_f)g,h\bigr\}+\bigl\{g,{\mathcal L}(X_f)h\bigr\}\,.
 \end{split}
\end{equation*}   
\noindent
{\bf4.\quad} Since the characteristic field of the Poisson manifold $(M,\Lambda)$
is generated by the Hamiltonian vector fields, any Hamiltonian vector field is everywhere tangent to the symplectic foliation. A Poisson vector field may not be tangent to that foliation.
\end{rmks}

\begin{prop}\label{FirstIntegrals}
Let $(M,\Lambda)$ be a Poisson manifold, $H\in C^\infty(M,\RR)$ be a smooth function
and $X_H=\Lambda^\sharp(\d H)\in A^1(M)$ be the associated Hamiltonian vector 
field. A smooth function $g\in C^\infty(M,\RR)$ keeps a constant value on each integral curve of $X_H$ if and only if $\{H,g\}=0$ identically. 
Such a function $g$ is said to be a
\emph{first integral} of $X_H$.
\par\smallskip

A specially important first integral of $X_H$, called the 
\emph{energy first integral}, is the Hamiltonian $H$.
\end{prop}

\begin{proof}
Let $\varphi:I\to M$ be an integral curve of $X_H$, defined on an open interval 
$I$ of $\RR$. For each $t\in I$
 $$\frac{\d\varphi(t)}{\d t}=X_H\bigl(\varphi(t)\bigr)\,.$$
The function $g\circ\varphi$ is smooth and satisfies
 $$\frac{\d(g\circ\varphi)(t)}{\d t}=\mathi(X_H)(\d g)
                 \bigl(\varphi(t)\bigr)=\{H,g\}\bigl(\varphi(t)\bigr)\,.
 $$
Since $I$ is connected, $g\circ\varphi$ keeps a constant value if and only if,
for each $t\in I$,
$\displaystyle\frac{\d(g\circ\varphi)(t)}{\d t}=0$, and the above equality 
proves that such is the case if and only if $\{H,g\}\bigl(\varphi(t)\bigr)=0$.
The indicated result follows from the fact that for any point $x\in M$, there exists an integral curve $\varphi:I\to M$ of $X_H$ and an element $t$ in $I$
such that $\varphi(t)=x$.
\par\smallskip

The skew-symmetry of the Poisson bracket implies $\{H,H\}=0$, therefore the Hamiltonian $H$ is a first integral of $X_H$. 
\end{proof} 

\begin{rmk}
Some Hamiltonian mechanical systems encountered in Mechanics, 
defined on a Poisson manifold $(M,\Lambda)$, have as Hamiltonian
a smooth function $H$ defined on $\RR\times M$ rather than on the
manifold $M$. Such a function $H$ is said to be a \emph{time-dependent
Hamiltonian}. The associated Hamiltonian vector field $X_H$ is no more
an ordinary vector field on $M$, \emph{i.e.} a smooth map $M\to TM$
wich associates to each $x\in M$ an element in $T_xM$, but rather
a \emph{time-dependent vector field}, \emph{i.e.} a smooth map
$X_H:\RR\times M\to TM$ such that, for each $t\in\RR$ and each $x\in M$
$X_H(t,x)\in T_xM$. For each fixed value of $t\in \RR$, the map
$x\mapsto X_H(t,x)$ is the Hamiltonian vector field on $M$ whose 
Hamiltonian is the function $H_t:M\to\RR$, defined by
 $$H_t(x)=H(t,x)\,,\quad x\in M\,.$$
Therefore 
 $$X_H(t,x)=\Lambda^\sharp\bigl(\d H_t\bigr)(x)\,,\quad x\in M\,,\ t\in\RR\,.
 $$  
A smooth parametrized curve $\varphi:I\to M$, defined on an open interval 
$I$ of $\RR$, is an integral curve of the time-dependent vector field $X_H$ if
for each $t\in I$ it satisfies the non-autonomous differential equation
 $$\frac{\d \varphi(t)}{\d t}=X_H\bigl(t,\varphi(t)\bigr)\,.$$
The time-dependent Hamiltonian $H:\RR\times M\to\RR$ is no more a first integral
of $X_H$ since, for each integral curve $\varphi:I\to M$ of $X_H$ and each 
$t\in I$,
 $$\frac{\d(H\circ\varphi)(t)}{\d t}
   =\frac{\partial H\bigl(t,\varphi(t)\bigr)}{\partial t}\,.
 $$
\end{rmk}

\begin{prop}\label{PoissonMaps} 
Let $(M_1,\Lambda_1)$ and $(M_2,\Lambda_2)$ be two Poisson manifolds and let
$\varphi:M_1\to M_2$ be a smooth map. The following properties are equivalent.

\begin{enumerate}
\item{} For any pair $(f,g)$ of smooth functions defined on $M_2$
 $$\{\varphi^*f,\varphi^*g\}_{M_1}=\varphi^*\{f,g\}_{M_2}\,.$$

\item{} For any smooth function $f\in C^\infty(M_2,\RR)$ the Hamiltonian vector fields
$\Lambda_2^\sharp(df)$ on $M_2$ and $\Lambda_1^\sharp\bigl(d(f\circ\varphi)\bigr)$ on 
$M_1$ are $\varphi$-compatible, which means that for each $x\in M_1$
 $$T_x\varphi\Bigl(\Lambda_1^\sharp\bigl(d(f\circ\varphi)(x)\bigr)\Bigr)
   =\Lambda_2^\sharp\Bigl(df\bigl(\varphi(x)\bigr)\Bigr)\,.
 $$

\item{} The bivector fields $\Lambda_1$ on $M_1$ and $\Lambda_2$ on $M_2$ are 
$\varphi$-compatible, which means that for each $x\in M_1$
 $$T_x\varphi\bigl(\Lambda_1(x)\bigr)
   =\Lambda_2\bigl(\varphi(x)\bigr)\,.
 $$
\end{enumerate}

A map $\varphi:M_1\to M_2$ which satisfies these equivalent properties is called a
\emph{Poisson map}.
\end{prop}

\begin{proof}
Let $f$ and $g$ be two smooth functions defined on $M_2$. For each $x\in M_1$,
we have
 \begin{equation*}
   \begin{split}
    \{\varphi^*f,\varphi^*g\}_{M_1}(x)&=\{f\circ\varphi,g\circ\varphi\}(x)
                                       =\Lambda_1(x)\bigl(\d(f\circ\varphi)(x),
                                                          \d(g\circ\varphi)(x)
                                                     \bigr)\\
     &=\Bigl\langle\d(g\circ\varphi)(x),
                  \Lambda_1^\sharp\bigl(\d(f\circ\varphi(x))\bigr)
      \Bigr\rangle\,.   
   \end{split}
 \end{equation*}
We have also
 \begin{equation*}
   \begin{split}
    \varphi^*\{f,g\}_{M_2}(x)&=\{f,g\}_{M_2}\bigl(\varphi(x)\bigr)\\
             &=\Bigl\langle\d g\bigl(\varphi(x)\bigr),
                           \Lambda_2^*\bigl(\d f\bigl(\varphi(x)\bigr)\bigr) 
               \Bigr\rangle\,.
   \end{split}
 \end{equation*}
These formulae show that Properties 1 and 2 are equivalent. 
\par\smallskip
We recall that $T_x\varphi\bigl(\Lambda_1(x)\bigr)$ is, by its very definition, the
bivector at $\varphi(x)\in M_2$ such that, for any pair $(f,g)$ of smooth functions on $M_2$
 $$T_x\varphi\bigl(\Lambda_1(x)\bigr)\Bigl(\d f\bigl(\varphi(x)\bigr),
                                           \d g\bigl(\varphi(x)\bigr)\Bigr)
     =\Lambda_1\bigl(\d(f\circ\varphi)(x),\d (g\circ\varphi)(x)\bigr)\,.
 $$
The above equalities therefore prove that Properties 2 and 3 are equivalent.
\end{proof}

Poisson manifolds often appear as quotients of symplectic manifolds, as indicated
by the following Proposition, due to Paulette Libermann \cite{Libermann83b}. 

\begin{prop}\label{QuotientPoisson} Let $(M,\omega)$ be a symplectic manifold and
let $\varphi:M\to P$ be a surjective submersion of $M$ onto a smooth manifold 
$P$ whose fibres are connected
(it means that for each $y\in P$, $\varphi^{-1}(y)$ is connected). The following properties are equivalent.

\begin{enumerate}

\item{} On the manifold $M$, the distribution $\orth(\ker T\varphi)$ is integrable.

\item{} For any pair $(f,g)$ of smooth functions defined on $P$, the 
Poisson bracket $\{f\circ\varphi,g\circ\varphi\}$ is constant on each fibre
$\varphi^{-1}(y)$ of the submersion $\varphi$ (with $y\in P$). 
\end{enumerate}

When these two equivalent properties are satisfied, there exists on $P$ a unique Poisson structure for which $\varphi:M\to P$ is a Poisson map (the manifold $M$ being endowed with the Poisson structure associated to its symplectic structure).
\end{prop}

\begin{proof} On the manifold $M$, $\ker T\varphi$ is a an integrable
distribution of rank $\dim M-\dim P$ whose integral submanifolds are the fibres of the submersion $\varphi$. Its symplectic orthog(onal $\orth(\ker T\varphi)$ is therefore a distribution of rank $\dim P$. Let $f$ and $g$ be two smooth functions defined on $M_2$. On $M_1$, the Hamiltonian vector fields $X_{f\circ\varphi}$
and $X_{g\circ\varphi}$ take their values in $\orth(\ker T\varphi)$. We have
 $$[X_{f\circ\varphi},X_{g\circ\varphi}]=X_{\{f\circ\varphi,g\circ\varphi\}}\,.$$   
Therefore $[X_{f\circ\varphi},X_{g\circ\varphi}]$ takes its values in
$\orth(\ker T\varphi)$ if and only if the function $\{f\circ\varphi,g\circ\varphi\}$
is constant on each fibre $\varphi^{-1}(y)$ of the submersion $\varphi$. The equivalence of Properties 1 and 2 easily follows.
\par\smallskip

Let us now assume that the equivalent properties 1 and 2 are satisfied. 
Since the map $\varphi:M\to P$ is a submersion with connected fibres, the map which associates to each function $f\in C^\infty(M_2,\RR)$ the function 
$f\circ\varphi$ is an isomorphism of $C^\infty(M_2,\RR)$ onto the subspace of
$C^\infty(M_1,\RR)$ made by smooth functions which are constant on each fibre of
$\varphi$. The existence and unicity of a Poisson structure on $M_2$ for which
$\varphi$ is a Poisson map follows.
\end{proof}

\begin{rmk}\label{PoissonPair}
Poisson manifolds obtained as quotients of symplectic manifolds often come by pairs.
Let us assume indeed that $(M,\omega)$ is a symplectic manifold and
that the above Proposition can be applied to a smooth surjective submersion with connected fibres $\varphi:M\to P$, and defines a Poisson structure
on $P$ for which $\varphi$ is a Poisson map. Since $\orth(\ker T\varphi)$ is integrable, it defines a foliation of $M$, which is said to be \emph{simple} when
the set of leaves $Q$ of that foliation has a smooth manifold structure such that the map $\psi:M\to Q$, which associates to each point in $M$ the leaf through this point, is a submersion. Then the maps $\varphi:M\to P$ and $\psi:M\to Q$ play similar parts, so there exists on $Q$ a unique Poisson structure for which $\psi$ is a Poisson map. Alan Weinstein \cite{Weinstein83} has determined the links which exist between the local structures of the two Poisson manifolds $P$ and $Q$
at corresponding points (that means, at points which are the images of the same point in $M$ by the maps $\varphi$ and $\psi$).
\end{rmk}

Several kinds of remarkable submanifolds of a Poisson manifold can be defined
\cite{Weinstein83}. The most important are the \emph{coisotropic} submanifolds, defined below.

\begin{defi}\label{coisotropicPoisson} 
A submanifold $N$ of a Poisson manifold $(M,\Lambda)$ is said to be
\emph{coisotropic} if for any point $x\in N$ and any pair 
$(f,g)$ of smooth functions defined on a neighbourhood $U$ of $x$ in $M$ whose restrictions to $U\cap N$ are constants, the Poisson bracket $\{f,g\}$ vanishes on $U\cap N$. 
\end{defi}

\subsection{Examples of Poisson manifolds}

\subsubsection{Symplectic manifolds} 
We have seen above that any symplectic manifold is a Poisson manifold.

\subsubsection{Dual spaces of finite-dimensional Lie algebras}
\label{CanonicalPoissonStructure} 
Let $\mathcal G$ be a finite-dimensional Lie algebra, and ${\mathcal G}^*$ its dual space. The Lie algebra $\mathcal G$ can be considered as the dual of ${\mathcal G}^*$, that means as the space of linear functions on ${\mathcal G}^*$, and the bracket of the Lie algebra $\mathcal G$ is a composition law on this space of linear functions. This composition law can be extended to the space $C^\infty({\mathcal G}^*,\RR)$
by setting
 $$\{f,g\}(x)=\Bigl\langle x,\bigl[\d f(x),\d g(x)\bigr]\Bigr\rangle\,,\quad 
          f\ \text{and}\ g\in C^\infty({\mathcal G}^*,\RR)\,,\quad 
          x\in{\mathcal G}^*\,.
 $$
This bracket on $C^\infty({\mathcal G}^*,\RR)$ defines a Poisson structure 
on  ${\mathcal G}^*$, called its \emph{canonical Poisson structure}.
It implicitly appears in the works of Sophus Lie, and was rediscovered by Alexander Kirillov \cite{Kirillov74}, Bertram Kostant
\cite{Kostant70} and Jean-Marie Souriau \cite{Souriau69}. Its existence can be seen as an application of Proposition \ref{QuotientPoisson}. Let indeed $G$ be the connected and simply connected Lie group whose Lie algebra is $\mathcal G$. We know that the cotangent bundle $T^*G$ has a canonical symplectic structure. One can check easily that for this symplectic structure, the Poisson bracket of two smooth functions defined on $T^*G$ and invariant with respect to the lift to $T^*G$ of the action of $G$ on itself by left translations, is too invariant with respect to that action. Application of Proposition 
\ref{QuotientPoisson}, the submersion
$\varphi:T^*G\to{\mathcal G}^*$ being the left translation which, for each $g\in G$, maps $T^*_gG$ onto $T^*_eG\equiv{\mathcal G}^*$, yields the above defined
Poisson structure on ${\mathcal G}^*$. If instead of translations on the left, we use translation on the right, we obtain on ${\mathcal G}^*$ the opposite Poisson structure. This illustrates Remark \ref{PoissonPair}, since, as we will see later,
each one of the tangent spaces at a point $\xi\in T^*G$ to the orbits of that point by the lifts to $T^*G$ of the actions of $G$ on itself by translations on the left and on the right, is the symplectic orthogonal of the other.
\par\smallskip

The symplectic leaves of ${\mathcal G}^*$ equipped with the above defined Poisson structure are the coadjoint orbits. 

\subsubsection{Symplectic cocycles}\label{SymplecticCocycle} 
A \emph{symplectic cocycle} of the Lie algebra
$\mathcal G$ is a skew-symmetric bilinear map 
$\widetilde\Theta:{\mathcal G}\times{\mathcal G}\to\RR$ which satisfies
 $$\widetilde\Theta\bigl([X,Y],Z\bigr)+\widetilde\Theta\bigl([Y,Z],X\bigr)
                            +\widetilde\Theta\bigl([Z,X],Y\bigr)=0\,.
 $$
The above defined canonical Poisson structure on ${\mathcal G}^*$ can be modified by means of a symplectic cocycle $\widetilde\Theta$ by defining the new bracket
(see for example \cite{LibermannMarle})
 $$\{f,g\}_{\widetilde\Theta}(x)=\Bigl\langle x,\bigl[\d f(x),\d g(x)\bigr]
                                  \Bigr\rangle
                                  -\widetilde\Theta\bigl(\d f(x),\d g(x))\,,
 $$
where $f$ and $g\in C^\infty({\mathcal G}^*,\RR)$, $x\in{\mathcal G}^*$.
This Poisson structure is called the \emph{modified canonical Poisson structure} by means of the symplectic cocycle $\widetilde\Theta$. We will see in Section 
\ref{ActionsOnCotangentBundle}
that the symplectic leaves of ${\mathcal G}^*$ equipped with this Poisson structure are
the orbits of an affine action whose linear part is the coadjoint action, with an additional term determined by $\widetilde\Theta$.

\section{Symplectic, Poisson and Hamiltonian actions}
\label{SymplecticPoissonHamiltonian}

\subsection{Actions on a smooth manifold}
Let us first recall some definitions and facts about actions of a Lie algebra 
or of a Lie group on a smooth manifold.

\begin{defi}
An \emph{action on the left} (resp. an \emph{action on the right}) 
of a Lie group $G$ on a smooth manifold $M$ is a smooth map 
$\Phi:G\times M\to M$ (respectively, $\Psi:M\times G\to M$) such that, for any $x\in M$,
$g_1$ and $g_2\in G$, $e\in G$ being the 
neutral element,

\begin{itemize}

\item{} for an action on the left
 $$\Phi\bigl(g_1,\Phi(g_2,x)\bigr)=\Phi(g_1g_2,x)\,,\quad \Phi(e,x)=x\,,$$

\item{} for an action on the right
  $$\Psi\bigl(\Psi(x, g_1),g_2\bigr)=\Psi(x, g_1g_2)\,,\quad \Psi(x,e)=x\,.$$
\end{itemize}
\end{defi}

\subsubsection{Consequences} 
Let $\Phi:G\times M\to M$ be an action on the left of the Lie group $G$ on the smooth manifold $M$. For each $g\in G$, we denote by 
$\Phi_g:M\to M$ the map
 $$\Phi_g(x)=\Phi(g,x)\,.$$
The map $g\mapsto\Phi_g$ is a groups homomorphism of $G$ into the group of smooth diffeomorphisms of $M$. In other words, for each $g\in G$, $\Phi_g$ is a diffeomorphism of $M$, and we have
 $$\Phi_{g}\circ\Phi_h=\Phi_{gh}\,,\quad (\Phi_g)^{-1}=\Phi_{g^{-1}}\,,\quad
    g\ \text{and}\ h\in G\,.$$ 
Similarly, let $\Psi:M\times G\to M$ be an action on the right of the Lie group $G$ on the smooth manifold $M$. For each $g\in G$, we denote by 
$\Psi_g:M\to M$ the map
 $$\Psi_g(x)=\Psi(x,g)\,.$$
The map $g\mapsto\Psi_g$ is a groups anti-homomorphism of $G$ into the group of smooth diffeomorphisms of $M$. In other words, for each $g\in G$, $\Psi_g$ is a diffeomorphism of $M$, and we have
 $$\Psi_{g}\circ\Psi_h=\Psi_{hg}\,,\quad (\Psi_g)^{-1}=\Psi_{g^{-1}}\,,\quad
    g\ \text{and}\ h\in G\,.$$ 

\begin{defi}
Let $\Phi:G\times M\to M$ be an action on the left (resp. let $\Psi:M\times G\to M$
be an action of the right) of the Lie group $G$ on the smooth manifold $M$. With each element $X\in {\mathcal G}\equiv T_eG$ (the tangent space to the Lie group $G$ at the neutral element) we associate the vector field $X_M$ on $M$ defined by 
 $$X_M(x)=\begin{cases}\displaystyle
           \frac{\d\Phi\bigl(\exp(sX),x\bigr)}{ds}\Bigm|_{s=0}&\text{if $\Phi$ 
                          is an action on the left,}\\
                       \displaystyle
           \frac{\d\Psi\bigl(x, \exp(sX)\bigr)}{ds}\Bigm|_{s=0}&\text{if $\Psi$ 
                          is an action on the right.}
          \end{cases}  
$$
The vector field $X_M$ is called the \emph{fundamental vector field} on $M$ associated to $X$.
\end{defi}

\begin{defi} An \emph{action} of a Lie algebra $\mathcal G$ on a smooth manifold
$M$ is a Lie algebras homomorphism $\varphi$ of $\mathcal G$ into the Lie algebra $A^1(M)$ of smooth vector fields on $M$ (with the Lie bracket of vector fields as composition law). In other words, it is a linear map $\varphi:{\mathcal G}\to A^1(M)$ such that
for each pair $(X,Y)\in{\mathcal G}\times{\mathcal G}$,
 $$\varphi\bigl([X,Y]\bigr)=\bigl[\varphi(X),\varphi(Y)\bigr]\,.$$
\end{defi} 

\begin{rmk} Let $G$ be a Lie group. There are two natural ways in which the 
tangent space $T_eG\equiv{\mathcal G}$ to the Lie group $G$
at the neutral element $e$ can be endowed with a Lie algebra structure.
\par\smallskip

In the first way, we associate with each element $X\in T_eG$ the \emph{left invariant} vector field $X^L$ on $G$ such that $X^L(e)=X$; its value at a point
$g\in G$ is $X^L(g)=TL_g(X)$, where $L_g:G\to G$ is the map $h\mapsto L_g(h)=gh$.
We observe that for any pair $(X,Y)$ of elements in $\mathcal G$ the Lie bracket
$[X^L, Y^L]$ of the vector fields $X^L$ and $Y^L$ on $G$ is left invariant, and we define the bracket $[X,Y]$ by setting $[X,Y]=[X^L,Y^L](e)$. This Lie algebra structure on ${\mathcal G}\equiv T_eG$ will be called the Lie algebra structure of
\emph{left invariant vector fields} on $G$.
\par\smallskip

In the second way, we choose the \emph{right invariant} vector fields on $G$
$X^R$ and $Y^R$, instead of the left invariant vector fields $X^L$ and $Y^L$. Since
$[X^R,Y^R](e)=-[X^L,Y^L](e)$, the Lie algebra structure on ${\mathcal G}\equiv T_eG$
obtained in this way, called the Lie algebra structure of \emph{right invariant vector fields}, is the \emph{opposite} of that of left invariant vector fields. 
We have therefore on $T_eG$ two opposite Lie algebras structures, both equally natural. Fortunately, the choice of one rather than the other as the Lie algebra 
$\mathcal G$ of $G$ does not matter because the map $X\mapsto -X$ is a Lie algebras isomorphism between these two structures.  
\end{rmk}

\begin{prop}\label{LeftRight} 
Let $\Phi:G\times M\to M$ be an action on the left
(resp. let $\Psi:M\times G\to M$ be an action on the right) of a Lie group $G$ on a smooth manifold $M$. We endow ${\mathcal G}\equiv T_eG$ with the Lie algebra structure of right invariant vector fields on $G$ (resp, with the Lie algebra structure of left invariant vector fields on $G$). The map $\varphi:{\mathcal G}\to A^1(M)$ 
(resp. $\psi:{\mathcal G}\to A^1(M)$) which associates to each element $X$ of the Lie algebra $\mathcal G$ of $G$ the corresponding fundamental vector field $X_M$, is an action of the Lie algebra 
$\mathcal G$ on the manifold $M$. This Lie algebra action is said to be \emph{associated} to the Lie group action $\Phi$ (resp. $\Psi$).
\end{prop}

\begin{proof}
Let us look at an action on the left $\Phi$. Let $x\in M$, and let 
$\Phi^x:G\to M$ be the map $g\mapsto\Phi^x(g)=\Phi(g,x)$.For any $X\in T_eG$ and $g\in G$, we have
 \begin{equation*}
   \begin{split}
      X_M\bigl(\Phi(g,x)\bigr)
            &=\frac{\d}{\d s}\,\Phi\bigl(\exp(sX),\Phi(g,x)\bigr)\bigm|_{s=0}
             =\frac{\d}{\d s}\,\Phi\bigl(\exp(sX)g,x\bigr)\bigm|_{s=0}\\
            &=\frac{\d}{\d s}\,\Phi\Bigl(R_g\bigl(\exp(sX)\bigr),x\Bigr)\Bigm|_{s=0}
             =T\Phi^x\circ TR_g(X)\,.
   \end{split}
  \end{equation*}
We see that for each $X\in T_xG$, the right invariant vector field $X^R$ on $G$ and the fundamental vector field $X_M$ on $M$ are compatible with respect to the map
$\Phi^x:G\to M$. Therefore for any pair $(X,Y)$ of elements in $T_eG$, we have
$[X,Y]_M=[X_M,Y_M]$. In other words the map $X\mapsto X_M$ is an action 
of the Lie algebra ${\mathcal G}=T_eG$ (equipped with the Lie algebra structure of right invariant vector fields on $G$) on the manifold $M$. 
\par\smallskip

For an action on the right $\Psi$, the proof is similar, ${\mathcal G}=T_eG$ 
being this time endowed with the Lie algebra structure of left invariant 
vector fields on $G$. 
\end{proof}

\begin{prop}\label{DirectImageFundamental} 
Let $\Phi:G\times M\to M$ be an action on the left
(resp. let $\Psi:M\times G\to M$ be an action on the right) of a Lie group $G$ on a smooth manifold $M$. Let $X_M$ be the fundamental vector field associated to an
element $X\in{\mathcal G}$. For any $g\in G$, the direct image $(\Phi_g)_*(X_M)$
(resp. $(\Psi_g)_*(X_M)$) of the vector field $X_M$ by the diffeomorphism $\Phi_g:M\to M$ (resp. $\Psi_g:M\to M$) is the fundamnetal vector field $(\Ad_gX)_M$ associated to $\Ad_gX$ (resp. the fundamental vector field $(\Ad_{g^{-1}}X)_M$ 
associated to $\Ad_{g^{-1}}X$).
\end{prop}

\begin{proof} For each $x\in M$
 \begin{equation*}
  \begin{split}
  (\Phi_g)_*(X_M)(x)
    &=T\Phi_g\Bigl(X_M\bigl(\Phi(g^{-1},x)\bigr)\Bigr)\\
    &=T\Phi_g\left(\frac{\d}{\d s}\Phi\bigl(\exp(sX)g^{-1},x\bigr)
       \bigm|_{s=0} \right)\\
    &=\frac{\d}{ds}\Bigl(\Phi\bigl(g\exp(sX)g^{-1},x\bigr)\Bigr)\Bigm|_{s=0}
     =(\Ad_gX)_M(x)\,,
  \end{split}
  \end{equation*}
since $g\exp(sX)g^{-1}=\exp(\Ad_gX)$. The proof for the action on the right $\Psi$ is similar.
\end{proof}

\subsection{Linear and affine representations}
In this section, after recalling some results about linear and affine 
transformation groups, we discuss linear and affine representations 
of a Lie group or of a Lie algebra in a finite-dimensional vector space, 
which can be seen as special examples of actions. 

\subsubsection{Linear and affine transformation groups and their Lie algebras}
\label{AffineTransforms} 
Let $E$ be a finite-dimensional vector space. The set of
linear isomorphisms $l:E\to E$ will be denoted by $\GL(E)$. We recall that equipped with the composition of maps 
 $$(l_1,l_2)\mapsto l_1\circ l_2$$ 
as a composition law, $\GL(E)$ is a Lie group whose dimension is $(\dim E)^2$. Its Lie algebra, which will be denoted by
$\gl(E)$, is the set ${\mathcal L}(E,E)$ of linear maps $f:E\to E$, with the commutator
 $$(f_1,f_2)\mapsto [f_1,f_2]=f_1\circ f_2-f_2\circ f_1$$
as a composition law.
\par\smallskip

A map $a:E\to E$ is called an \emph{affine map} if it can be written as
 $$a(x)=l(x)+c\,,\quad x\in E\,,$$
the map $l:E\to E$ being linear, and $c\in E$ being a constant. The affine map $a$ is invertible if and only if its linear part $l$ is invertible, in other words if and only if $l\in \GL(E)$; when this condition is satisfied, its inverse is
 $$a^{-1}(y)=l^{-1}(y-c)\,,\quad y\in E\,.$$
By identifying the invertible affine map $a$ with the pair $(l,c)$, with $l\in\GL(E)$
and $c\in E$, the set of invertible affine maps of $E$ onto itself becomes identified with $\GL(E)\times E$. The composition law and the inverse map on this product (which
is called the \emph{semi-direct product of} $\GL(E)$ \emph{with} $E$) are
 $$(l_1,c_1),(l_2,c_2)\mapsto (\bigl(l_1\circ l_2, l_1(c_2)+c_1)\bigr)\,,\quad
    (l,c)^{-1}=\bigl(l^{-1},-l^{-1}(c)\bigr)\,.$$
The semi-direct product $\Aff(E)=\GL(E)\times E$ is a Lie group whose dimension is
$(\dim E)^2+\dim E$; its Lie algebra is the product $\aff(E)={\mathcal L}(E,E)\times E$, with the composition law
 $$\bigl((f_1,d_1),(f_2,d_2)\bigr)\mapsto\bigl[(f_1,d_1),(f_2,d_2)\bigr]
      =\bigl(f_1\circ f_2-f_2\circ f_1, f_1(d_2)-f_2(d_1)\bigr)\,.
 $$
The adjoint representation is given by the formula
 $$\Ad_{(l,c)}\bigl((f,d)\bigr)=\bigl(l\circ f\circ l^{-1}, 
   l(d)-l\circ f\circ l^{-1}(d)\bigr)\,.$$

\begin{rmk}\label{SignChange} 
The finite-dimensional vector space $E$ can be considered as a smooth manifold on which $E$ itself transitively acts by translations. That action determines
a natural trivialization of the tangent bundle $TE$, the tangent space $T_xE$ at each
point $x\in E$ being identified with $E$. An element $a\in\aff(E)$, in other words
an affine map $a:E\to E$, can therefore be considered as the vector field on
$E$ whose value, at each $x\in E$, is $a(x)\in T_xE\equiv E$. A question
naturally arises: how the bracket of two elements $a_1$ and $a_2\in \aff(E)$, for the Lie algebra structure of $\aff(E)$ defined in \ref{AffineTransforms}, compares with the bracket of these two elements when considered as vector fields on $E$? An easy calculation in local coordinates shows that the bracket $[a_1,a_2]$ defined in
\ref{AffineTransforms} is the \emph{opposite} of the bracket of these two elements 
when considered as vector fields on $E$. Remark \ref{explanation} below will explain
the reason of that change of sign.
\end{rmk}

\begin{defis}\label{DefinitionLinearRepresentation} 
Let $G$ be a Lie group, ${\mathcal G}$ a Lie algebra and $E$ a finite-dimensional vector space.

\par\smallskip\noindent
{\bf 1.\quad}
A \emph{linear representation} (respectively, an \emph{affine representation}) of the Lie group $G$ in the vector space $E$ is a Lie groups homomorphism
$R:G\to\GL(E)$ of $G$ in the Lie group $\GL(E)$ of linear transformations of $E$ (respectively, 
a Lie groups homomorphism $A:G\to\Aff(E)$ of $G$ in the Lie group
$\Aff(E)$ of affine transformations of $E$).

\par\smallskip\noindent
{\bf 2.\quad}
A \emph{linear representation} (respectively, an \emph{affine representation}) of the Lie algebra $\mathcal G$ in the vector space $E$ is a Lie algebras homomorphism 
$r:{\mathcal G}\to \gl(E)$ of the Lie algebra $\mathcal G$ in the Lie algebra
$\gl(E)$ of the group of linear transformations of $E$ (resp, a Lie algebras
homomorphism $a:{\mathcal G}\to\aff(E)$ of the Lie algebra $\mathcal G$ in 
the Lie algebra $\aff(E)$ of the group of affine transformations of $E$).
\end{defis}

\begin{exs}\label{AdjointAndCoadjointRepresentations} Let $G$ be a Lie group. The \emph{adjoint representation} of $G$ is the linear representation of $G$ in its Lie algebra $\mathcal G$ which associates, to each 
$g\in G$, the linear isomorphism $\Ad_g\in\GL(\mathcal G)$
 $$\Ad_g(X)=TL_g\circ TR_{g^{-1}} (X)\,,\quad (X\in{\mathcal G})\,.
 $$
The \emph{coadjoint representation} of $G$ is the contragredient of the adjoint representation. It associates to each $g\in G$ the linear isomorphism
$\Ad^*_{g^{-1}}\in \GL({\mathcal G}^*)$, which satisfies, for each 
$\zeta\in{\mathcal G}^*$ and $X\in{\mathcal G}$,
 $$\bigl\langle\Ad^*_{g^{-1}}(\zeta),X\bigr\rangle
   =\bigl\langle \zeta, \Ad_{g^{-1}}(X)\bigr\rangle\,.
 $$
The \emph{adjoint representation} of the Lie algebra $\mathcal G$ is the linear representation of $\mathcal G$ into itself which associates, to each 
$X\in{\mathcal G}$, the linear map $\ad_X\in\gl({\mathcal G})$
 $$\ad_X(Y)=[X,Y]\,,\quad (Y\in{\mathcal G})\,.
 $$
The \emph{coadjoint representation} of the Lie algebra $\mathcal G$ is the contragredient of the adjoint representation. It associates, to each $X\in{\mathcal G}$,
the linear map $\ad_{-X}^*\in\gl({\mathcal G}^*)$ which satisfies, for each 
$\zeta\in{\mathcal G}^*$ and $X\in{\mathcal G}$,
 $$\bigl\langle \ad^*_{-X}\zeta,Y\bigr\rangle=\bigl\langle\zeta, [-X,Y]\bigr\rangle\,.
 $$
The adjoint representation (respectively, the coadjoint representation) 
of $\mathcal G$ is the Lie algebra representation 
associated to the adjoint representation (respectively, the coadjoint
representation) of the Lie group $G$, in the sense recalled below in the proof of 
\ref{AssociatedLieAlgebraRepresentation}. 
\end{exs}

\begin{prop}\label{LinearPartA}
Let $G$ be a Lie group and $E$ a finite-dimensional vector space.
A map $A:G\to \Aff(E)$ always can be written as
 $$A(g)(x)=R(g)(x)+\theta(g)\,,\quad\hbox{with}\ g\in G\,,\ x\in E\,,$$
where the maps $R:G\to\GL(E)$ and $\theta:G\to E$ are determined by $A$.
The map $A$ is an affine representation of $G$ in $E$ if and only if the following two properties are satisfied:

\begin{itemize}
\item{} $R:G\to \GL(E)$ is a linear representation of $G$ in the vector space $E$,

\item{} the map $\theta:G\to E$ is a one-cocycle of $G$ with values in $E$, for the linear representation $R$; it means that $\theta$ is a smooth map which satisfies,
for all $g$ and $h\in G$, 
 $$\theta(gh)=R(g)\bigl(\theta(h)\bigr)+\theta(g)\,.
 $$ 
\end{itemize}

When these two properties are satisfied, the linear representation $R$ is called the
\emph{linear part} of the affine representation $A$, and $\theta$ is called the
\emph{one-cocycle} of $G$ associated to the affine representation $A$.
\end{prop}

\begin{proof}
Since $\Aff(E)=\GL(E)\times E$, for each $g\in G$ and $x\in E$, we have
 $$A(g)(x)=R(g)(x)+\theta(g)\,,$$
where the maps $R:G\to \GL(E)$ and $\theta:G\to E$ are determined by $A$. 
By comparing $A(gh)$ and $A(g)\circ A(h)$, for $g $ and $h\in G$, using the
composition law of $\Aff(E)$ recalled in subsection \ref{AffineTransforms}, 
we easily check that
$A$ is an affine representation, which means that it is smooth and
satisfies, for all $g$ and $h\in G$, $A(gh)=A(g)\circ A(h)$, and $A(e)=\id_E$, 
if and only if the two above stated properties are satisfied.
\end{proof}

For linear and affine representations of a Lie algebra, we have the following infinitesimal analogue of Proposition \ref{LinearPartA}.

\begin{prop}\label{LinearParta}
Let $\mathcal G$ be a Lie algebra and $E$ a finite-dimensional vector space.
A linear map $a:{\mathcal G}\to \aff(E)$ always can be written as
 $$a(X)(x)=r(X)(x)+\Theta(X)\,,\quad\hbox{with}\ X\in {\mathcal G}\,,\ x\in E\,,$$
where the linear maps $r:{\mathcal G}\to\gl(E)$ and $\Theta:{\mathcal G}\to E$ are determined by $a$. The map $a$ is an affine representation of $G$ in $E$ if and only if the following two properties are satisfied:

\begin{itemize}
\item{} $r:{\mathcal G}\to \gl(E)$ is a linear representation of the Lie algebra
$\mathcal G$ in the vector space $E$,

\item{} the linear map $\Theta:{\mathcal G}\to E$ is a one-cocycle of $\mathcal G$ 
with values in $E$, for the linear representation $r$; it means that $\Theta$ 
satisfies, for all $X$ and $Y\in {\mathcal G}$, 
 $$\Theta\bigl([X,Y]\bigr)=r(X)\bigl(\Theta(Y)\bigr)-r(Y)\bigl(\Theta(X)\bigr)\,.
 $$ 
\end{itemize}

When these two properties are satisfied, the linear representation $r$ is called the
\emph{linear part} of the affine representation $a$, and $\Theta$ is called the
\emph{one-cocycle} of $\mathcal G$ associated to the affine representation $a$.
\end{prop}

\begin{proof}
Since $\aff(E)=\gl(E)\times E={\mathcal L}(E,E)\times E$, for each $X\in{\mathcal G}$ and $x\in E$, we have
 $$a(X)(x)=r(X)(x)+\Theta(X)\,,$$
where the linear maps $r:{\mathcal G}\to \gl(E)={\mathcal L(E,E)}$ and 
$\Theta:{\mathcal G}\to E$ are determined by $a$. 
By comparing $a\bigl([X,Y]\bigr)$ and $\bigl[a(X),a(Y)\bigr]$, for $X$ and 
$Y\in{\mathcal G}$, using the expression of the bracket of $\aff(E)$ recalled
in subsection \ref{AffineTransforms}, we easily check that 
$A$ is an affine representation, which means that it is smooth and
satisfies, for all $X$ and $Y\in{\mathcal G}$, 
$a\bigl([X,Y]\bigr)=\bigl[a(X),a(Y)\bigr]$ 
if and only if the two above stated properties are satisfied.
\end{proof}

\begin{prop}\label{AssociatedLieAlgebraRepresentation}
Let $A:G\to \Aff(E)$ be an affine representation of a Lie group $G$ in a 
finite-dimensional vector space $E$, and $\mathcal G$ be the Lie algebra of $G$.
Let $R:G\to\GL(E)$ and $\theta:G\to E$ be, respectively, the linear part and 
the associated cocycle of the affine representation $A$.
Let $a:{\mathcal G}\to\aff(E)$ be the affine representation of the Lie algebra
$\mathcal G$ associated (in the sense recalled below in the proof) 
to the affine representation $A:G\to\Aff(E)$ of the Lie group 
$G$. The linear part of $a$ is the linear representation $r:{\mathcal G}\to\gl(E)$ associated to the linear representation $R:G\to\GL(E)$, and the associated cocycle
$\Theta:{\mathcal G}\to E$ is related to the one-cocycle $\theta:G\to E$ by
 $$\Theta(X)=T_e\theta\bigl(X(e)\bigr)\,,\quad(X\in{\mathcal G})\,.$$
\end{prop}

\begin{proof} We recall that when we have a Lie groups homomorphism $A:G\to H$ of a Lie group $G$ into another Lie group $H$, the associated Lie algebras homomorphism
$a:{\mathcal G}\to {\mathcal H}$ of Lie algebras associates, to each $X\in{\mathcal G}$
(seen as the space of left-invariant vector fields on $G$) the left-invariant vector field $a(X)$ on $H$ whose value at the neutral element is $T_eA\bigl(X(e)\bigr)$. 
Let $X\in{\mathcal G}$. For each 
$t\in\RR$ and $x\in E$, we have
 $$A\bigl(\exp(tX)\bigr)(x)=R\bigl(\exp(tX)\bigr)(x)+\theta\bigl(\exp(tX)\bigr)\,.$$
By taking the derivative of both sides of this equality with respect to $t$, then setting $t=0$, we get
 $$a(X)(x)=r(X)(x)+T_e\theta(X)\,.$$
Therefore the affine representation $a$ has $r$ as linear part and $\Theta=T_e\theta$ 
as associated one-cocycle.
\end{proof}

\begin{rmk}\label{explanation}
Let $A:G\to \Aff(E)$ 
be an affine representation of a Lie group $G$ in a finite-dimensional 
vector space $E$. The map  $\widetilde A:G\times E\to E$,
 $$\widetilde A(g,x)=A(g)(x)\,,\quad g\in G\,,\ x\in E\,,
 $$
is an action on the left of $G$ on $E$. Proposition 
\ref{AssociatedLieAlgebraRepresentation} shows that $a:{\mathcal G}\to\aff(E)$ 
is a Lie algebras homomorphism, the Lie algebra structure of $\aff(E)$ being 
the structure  defined in \ref{AffineTransforms}. For each $X\in{\mathcal G}$, 
the element $a(X)\in\aff(E)$, when considered
as an affine vector field on $E$, is the fundamental
vector field associated to $X$, for the action on the left $\widetilde A$
of $G$ on $E$. We have seen (\ref{LeftRight}) that for an action on the 
left of $G$ on $E$, the map which associates to each $X\in {\mathcal G}$ the corresponding fundamental vector field on $E$ is a Lie algebras homomorphism 
of the Lie algebra of \emph{right invariant} vector fields on $G$ into the Lie algebra of smooth vector fields on $E$. This explains why, as was observed in \ref{SignChange},
the Lie algebra structure of $\aff(E)$ defined in \ref{AffineTransforms} is the opposite of the Lie algebra structure which exists on the space of
affine vector fields on $E$. Of course, this remark is also valid for a linear representation $R:G\to \GL(E)$, since $\GL(E)$ is a Lie subgroup of $\Aff(E)$.
\end{rmk}

\begin{defis}\label{coboundary}\hfill
\par\noindent
{\bf 1.\quad} Let $R:G\to \GL(E)$ be a linear representation of a Lie group $G$ in a 
finite-dimensional vector space $E$. A \emph{one-coboundary} of $G$ with values in 
$E$, for the linear representation $R$, is a map $\theta:G\to E$ 
which can be expressed as
 $$\theta(g)=R(g)(c)-c\,,\quad (g\in G)\,,
 $$
where $c$ is a fixed element in $E$. 
\par\noindent
{\bf 2.\quad} Ler $r:{\mathcal G}\to \gl(E)$ be a linear representation of a 
Lie algebra $\mathcal G$ in a finite-dimensional vector space $E$. A \emph{one-coboundary} of $\mathcal G$ with values in $E$, for the linear representation $r$,
is a linear map $\Theta:{\mathcal G}\to E$ which can be expressed as
 $$\Theta(X)=r(X)(c)\,,\quad (X\in{\mathcal G})\,,
 $$    
where $c$ is a fixed element in $E$.
\end{defis}

\begin{rmk} The reader will easily check the following properties. A one-coboundary 
of a Lie group $G$ with values in a finite-dimensional vector space 
$E$ for a linear representation $R:G\to \GL(E)$, 
automatically is a one-cocycle in the sense of \ref{LinearPartA}.
Similarly, a one-coboundary of a Lie algebra $\mathcal G$ with values in $E$ for the linear representation $r:{\mathcal G}\to\gl(E)$, automatically is a one-cocycle of 
$\mathcal G$ in the sense of \ref{LinearParta}. When a Lie group
one-cocycle $\theta:G\to E$ is in fact a one-coboundary, the associated Lie algebra
one-cocycle $\Theta=T_e\theta$ is a Lie algebra one-coboundary.
\end{rmk}

\begin{prop}\label{AffineEquivalence} 
Let $A:G\to\Aff(E)$ be an affine representation of a Lie group $G$ in a
finite-dimensional vector space $E$, $R:G\to\GL(E)$ be its linear part and 
$\theta:G\to E$ be the associated Lie group one-cocycle. 
The following properties are equivalent.

\begin{enumerate} 

\item{} There exists an element $c\in E$ such that, 
for all $g\in G$ and $x\in E$,
 $$ A(g)(x)=R(g)(x+c)-c\,.$$

\item{} The one-cocycle $\theta:G\to E$ is in fact a $1$-coboudary, whose expression is
 $$\theta(g)=R(g)(c)-c\,.$$ 
\end{enumerate}
\end{prop}

\begin{proof} Since for each $g\in G$ $R(g)$ is linear, Property 1 can be written
 $$A(g)(x)=R(g)(x)+\bigl(R(g)(c)-c\bigr)\,.$$
Therefore Property 1 is true if and only if $\theta(g)=R(g)(c)-c$, in other words if and only if Property 2
is true.
\end{proof}

The following Proposition is the infinitesimal analogue, for affine representations of a
Lie algebra, of Proposition \ref{AffineEquivalence}.

\begin{prop}\label{affineEquivalence} 
Let $a:{\mathcal G}\to\aff(E)$ be an affine representation of a Lie algebra 
$\mathcal G$ in a finite-dimensional vector space $E$, $r:{\mathcal G}\to\gl(E)$ 
be its linear part and $\Theta:{\mathcal G}\to E$ be the associated 
Lie algebra one-cocycle. The following properties are equivalent.

\begin{enumerate} 

\item{} There exists an element $c\in E$ such that, 
for all $X\in {\mathcal G}$ and $x\in E$,
 $$ a(X)(x)=r(X)(x+c)\,.$$

\item{} The one-cocycle $\Theta:{\mathcal G}\to E$ is in fact a $1$-coboudary, whose expression is
 $$\Theta(X)=r(X)(c)\,.$$ 
\end{enumerate}
\end{prop}

\begin{proof} Since for each $X\in {\mathcal G}$ $r(X)$ is linear, Property (i) can be written
 $$a(X)(x)=r(X)(x)+r(X)(c)\,.$$
Therefore Property 1 is true if and only if $\Theta(X)=r(X)(c)$, in other words if and only if Property 2 is true.
\end{proof}

\begin{rmk} Let us say that an affine representation $A:G\to\Aff(E)$ of a Lie group 
$G$ in a finite-dimensional vector space $E$ is equivalent to its linear part 
$R:G\to\GL(E)$ if there exists a translation $T:E\to E$ such that, for all $g\in G$
and $x\in E$,
 $$A(g)(x)=T^{-1}\circ R(g)\circ T(x)\,.$$
Proposition \ref{AffineEquivalence} expresses the fact that the affine representation 
$A$ is equivalent to its linear part $R$ if and only if its associated Lie group cocycle $\theta$ is a one-coboundary. The reader will easily formulate a similar 
interpretation of Proposition \ref{affineEquivalence}.
\end{rmk} 

\begin{prop}\label{LieAlgebraCocycleIntegration}
Let $G$ be a connected and simply connected Lie group, $R:G\to\GL(E)$ be a linear
representation of $G$ in a finite-dimensional vector space $E$, and
$r:{\mathcal G}\to\gl(E)$ be the associated linear representation of its Lie algebra
$\mathcal G$. For any one-cocycle $\Theta:{\mathcal G}\to E$ of the Lie algebra 
$\mathcal G$ for the linear representation $r$, there exists a unique
one-cocycle $\theta:G\to E$ of the Lie group $G$ for the linear representation $R$
such that $\Theta=T_e\theta$, in other words, which has $\Theta$ as associated
Lie algebra one-cocycle. The  Lie group one-cocycle $\theta$ is a Lie group
one-coboundary if and only if the Lie algrebra one-cocycle $\Theta$ is a 
Lie algebra one-coboundary.
\end{prop}  

\begin{proof} 
If $\theta:G\to E$ is a Lie group one-cocycle such that $T_e\theta=\Theta$ we have,
for any $g\in G$ and $X\in{\mathcal G}$,
 $$\theta\bigl(g\exp(tX)\bigr)=\theta(g)+R(g)\Bigl(\theta\bigl(\exp(tX)\bigr)\Bigr)\,.$$
By taking the derivative of both sides of this equality with respect to $t$, then setting $t=0$, we see that 
 $$T_g\theta\bigl(TL_g(X)\bigr)=R(g)\bigl(\Theta(x)\bigr)\,,
 $$
which proves that if it exists, the Lie group one-cocycle $\theta$ such that 
$T_e\theta=\Theta$ is unique.
\par\smallskip

For each $g\in G$ let $\eta(g):T_gG\to E$ be the map
 $$\eta(g)(X)=R(g)\circ\Theta\circ TL_{g^{-1}}(X)\,,\quad X\in T_gG\,.$$
The map $\eta$ is an $E$-valued differential one-form on $G$. Let us calculate its
exterior differential $\d\eta$, which is an $E$-valued differential two-form on $G$
(if the reader does not feel at ease with $E$-valued differential forms on $G$, he can
consider separately the components of $\eta$ in a basis of $E$, which are ordinary real-valued one-forms). Let $X$ and $Y$ be two left-invariant vector fields on $G$. We have,
for each $g\in G$,	
 $$\d\eta(g)\bigl(X(g),Y(g)\bigr)=\mathcal L(X)\bigl(\langle\eta,Y\rangle(g)\bigr)
                                 -\mathcal L(Y)\bigl(\langle\eta,X\rangle(g)\bigr)
                                 -\bigl\langle\eta,[X,Y]\bigr\rangle(g)\,.
 $$      
But 
 $$\langle\eta,Y\rangle(g)=R(g)\circ\Theta(Y)\,,\quad
   \langle\eta,X\rangle(g)=R(g)\circ\Theta(X)\,,
 $$
therefore
 $$\mathcal L(X)\bigl(\langle\eta,Y\rangle(g)\bigr)=
   \frac{\d}{\d t}\Bigl(R\bigl(g\exp(tX)\bigr)\circ\Theta(Y)\Bigr)\Bigm|_{t=0}
   =R(g)\circ r(X)\circ\Theta(Y)\,.
 $$
Similarly
 $$\mathcal L(Y)\bigl(\langle\eta,X\rangle(g)\bigr)
   =R(g)\circ r(Y)\circ\Theta(X)\,,
 $$
and 
 $$\bigl\langle\eta,[X,Y]\bigr\rangle(g)=R(g)\circ\Theta\bigl([X,Y]\bigr)\,,
 $$
Since the condition which expresses that $\Theta$ is a
Lie algebra one-cocycle for the linear representation $r$ asserts that
 $$r(X)\circ\Theta(Y)-r(Y)\circ\Theta(X)-\Theta\bigl([X,Y]\bigr)=0\,,$$
we conclude that the one-form $\eta$ is closed, \emph{i.e.} satisfies $d\eta=0$.
Since $G$ is assumed to be simply connected, the one-form $\eta$ is exact, and since
$G$ is assumed to be connected, for any $g$ in $G$, there exists a smooth
parametrized curve $\gamma:[0,T]\to G$ such that $\gamma(0)=e$ and $\gamma(T)=g$.
Let us set
 $$\theta(g)=\int_0^T\eta\left(\frac{\d\gamma(t)}{\d t}\right)\,dt\,.$$
Since $\eta$ is exact, the right hand side of the above equality only depends on the end
points $\gamma(0)=e$ and $\gamma(T)=g$ of the parametrized curve $\gamma$, which allows us to define $\theta(g)$ by that equality. So defined, $\gamma:G\to E$ is a smooth map.
Its very definition shows that $T_e\theta=\Theta$. If $g$ and $h$ are two elements in 
$G$, let $\gamma:[0,T_2]\to G$ be a smooth parametrized curve such that
$0<T_1<T_2$, $\gamma(0)=e$, $\gamma(T_1)=g$ and $\gamma(t_2)=gh$. We have 
 $$\theta(gh)=\int_0^{T_2}\eta\left(\frac{\d\gamma(t)}{\d t}\right)\,dt
             =\int_0^{T_1}\eta\left(\frac{\d\gamma(t)}{\d t}\right)\,dt
             +\int_{T_1}^{T_2}\eta\left(\frac{\d\gamma(t)}{\d t}\right)\,dt\,.
 $$
Observe that
 $$\int_0^{T_2}\eta\left(\frac{\d\gamma(t)}{\d t}\right)\,dt=\theta(g)$$
and that
 $$\int_{T_1}^{T_2}\eta\left(\frac{\d\gamma(t)}{\d t}\right)\,dt=
   R(g)\circ\int_{T_1}^{T_2}\eta\left(\frac{\d\bigl(L_{g^{-1}}\circ\gamma(t)\bigr)}
   {dt}\right)\,dt=R(g)\bigl(\theta(h)\bigr)\,,
 $$
which proves that $\theta$ is a Lie group one-cocycle.
\par\smallskip

We already know that if $\theta$ is a Lie group one-coboundary, $\Theta=T_e\theta$ is a
Lie algebra coboundary. Conversely let us assume that $\Theta$ is a Lie algebra 
one-coboundary. We have, for each $X\in {\mathcal G}$,
 $$\Theta(X)=r(X)(c)\,,$$
where $c$ is a fixed element in $E$. Let $g\in G$ and let $\gamma:[0,T]\to G$ be a smooth parametrized curve in $G$ such that
$\gamma(0)=e$ and $\gamma(T)=g$. We have
 $$\theta(g)=\int_0^T\eta\left(\frac{\d\gamma(t)}{\d t}\right)\,dt
            =\int_0^tR\bigl(\gamma(t)\bigr)\circ 
              r\left(TL_{\bigl(\gamma(t)\bigr)^{-1}}\frac{\d\gamma(t)}{\d t}\right)(c)
              \,dt\,.
 $$
But by taking the derivative with respect to $t$ of the two sides of the equality
 $$R\bigl(g\exp(tX)\bigr)=R(g)\circ R\bigl(\exp(tX)\bigr)$$
and then setting $t=0$, we see that
 $$\frac{\d}{\d t}R\bigl(\gamma(t)\bigr)=R\bigl(\gamma(t)\bigr)\circ 
    r\left(TL_{\bigl(\gamma(t)\bigr)^{-1}}\frac{\d\gamma(t)}{\d t}\right)\,.
 $$
Therefore $\theta$ is a Lie group one-coboundary since we have
 $$\theta(g)=\int_{0}^T\frac{\d R\bigl(\gamma(t)\bigr)}{\d t}\,(c)\,dt
            =R(g)(c)-R(e)(c)=R(g)(c)-c\,.\qedhere 
 $$
\end{proof}

\subsection{Poisson, symplectic and Hamiltonian actions}

\begin{defis}\hfill 

\noindent
{\bf 1.\quad} An action $\varphi$ of a Lie algebra $\mathcal G$ on a Poisson manifold $(M,\Lambda)$ is called a \emph{Poisson action} if for any $X\in{\mathcal G}$ the corresponding vector field $\varphi(X)$ is a Poisson vector field. When the Poisson manifold is in fact a symplectic manifold $(M,\omega)$, Poisson vector fields on $M$ are locally Hamiltonian vector fields and a Poisson action is called 
a \emph{symplectic action}.

\par\smallskip\noindent
{\bf 2.\quad} An action $\Phi$ (either on the left or on the right) of a
Lie group $G$ on a Poisson manifold $(M,\Lambda)$ is called a \emph{Poisson action}
when for each $g\in G$, 
 $$(\Phi_g)_*\Lambda=\Lambda\,.$$
When the Poisson manifold $(M,\Lambda)$ is in fact a symplectic manifold 
$(M,\omega)$, a Poisson action is called a \emph{symplectic action}; the fibre bundles isomorphism $\Lambda^\sharp:T^*M\to TM$ being the inverse of
$\omega^\flat:TM\to T^*M$, we also can say that an action $\Phi$ of a Lie group $G$ on a symplectic manifold $(M,\omega)$ is called a \emph{symplectic action}  
when for each $g\in G$, 
 $$(\Phi_g)^*\omega=\omega\,.$$  
\end{defis} 

\begin{prop}
We assume that $G$ is a connected Lie group which acts by an action $\Phi$, 
either on the left or on the right, on a Poisson manifold $(M,\Lambda)$, 
in such a way that the corresponding action of its Lie algebra
$\mathcal G$ is a Poisson action. Then the action $\Phi$ itself is a Poisson action.
\end{prop}

\begin{proof}
Let $X\in{\mathcal G}$. For each $x\in M$, the parametrized curve
$s\mapsto \Phi_{\exp(sX)}(x)$ is the integral curve of the fundamental 
vector field $X_M$ which takes the value $x$ for $s=0$. In other words, the reduced flow of the vector field $X_M$ is the map, defined on $\RR\times M$ and taking its values in $M$,
 $$(s,x)\mapsto\Phi_{\exp(sX)}(x)\,.$$
According to a formula which relates inverse images of multivectors
or differential forms with respect to the flow of a vector field, 
with their Lie derivatives with respect to that vector field 
(see for example \cite{LibermannMarle}, Appendix 1, section 3.4, 
page 351), for any $s_0\in\RR$
 $$\frac{\d}{\d s}\Bigl(\bigl((\Phi_{\exp(sX)})^*(\Lambda)\bigr)
   (x)\Bigr)\Bigm|_{s=s_0}
  =\Bigl((\Phi_{\exp(s_0X)})^*
                 \bigl({\mathcal L(X_M)\Lambda}\bigr)\Bigr)(x)=0\,,
 $$
since ${\mathcal L}(X_M)\Lambda=0$. Therefore for any $s\in\RR$,
 $$(\Phi_{\exp(sX)})^*\Lambda=(\Phi_{\exp(-sX)})_*\Lambda=\Lambda\,.$$
The Lie group $G$ being connected, any $g\in G$ is the product 
of a finite number of exponentials, so $(\Phi_g)_*\Lambda=\Lambda$. 
\end{proof}

\subsubsection{Other characterizations of Poisson actions}\label{exo}
Let $\Phi$ be an action, either on the left or on the right, of a Lie group $G$ on a Poisson manifold $(M,\Lambda)$. The reader will easily prove that
the following properties are equivalent. Therefore any of these properties can be used as the definition of a Poisson action.

\begin{enumerate}

\item{} For each $g\in G$, 
 $$(\Phi_g)_*\Lambda=\Lambda\,.$$

\item{} For each $g\in G$ and $f\in C^\infty(M,\RR)$,
 $$(\Phi_g)_*(X_f)=X_{(\Phi_g)_*(f)}\,.$$

\item{} For each $g\in G$, $\Phi_g:M\to M$ is a \emph{Poisson map}, 
which means that for each pair $(f_1,f_2)$ of smooth functions on $M$,
 $$\bigl\{(\Phi_g)^*f_1,(\Phi_g)^*f_2\bigr\}=(\Phi_g)^*\bigl(\{f_1,f_2\}\bigr)\,;$$

\item{} In the special case when the Poisson manifold $(M,\Lambda)$ is in fact
a symplectic manifold $(M,\omega)$, for each $g\in G$,
 $$(\Phi_g)^*\omega=\omega\,.$$
\end{enumerate}

The reader will easily prove that when these equivalent properties are satisfied, the
action of the Lie algebra ${\mathcal G}$ of $G$ which associates, to each $X\in{\mathcal G}$, the fundamental vector field $X_M$ on $M$,
is a Poisson action.

\begin{defis}\label{HamiltonianActions}\hfill

\noindent
{\bf 1.\quad} An action $\varphi$ of a Lie algebra $\mathcal G$ on a Poisson manifold $(M,\Lambda)$ is called a \emph{Hamiltonian action} if for every $X\in{\mathcal G}$ the corresponding vector field
$\varphi(X)$ is a Hamiltonian vector field on $M$.

\par\smallskip\noindent
{\bf 2.\quad} An action $\Phi$ (either on the left or on the right) 
of a Lie group $G$ on a Poisson manifold $(M,\Lambda)$ is
called a \emph{Hamiltonian action} if it is a Poisson action 
(or a symplectic action when the Poisson manifold $(M,\Lambda)$ 
is in fact a symplectic manifold$(M,\omega)$) and if, 
in addition, the associated action $\varphi$ of its Lie algebra 
is a Hamiltonian action. 
\end{defis}

\begin{rmks}\hfill

\noindent
{\bf 1.\quad} A Hamiltonian action of a Lie algebra on a Poisson manifold is automatically a Poisson action. 

\par\smallskip\noindent
{\bf 2.\quad} An action $\Phi$ of a connected Lie group $G$ on a Poisson manifold such that the corresponding action of its Lie algebra 
is Hamiltonian, automatially is a Hamiltonian action.

\par\smallskip\noindent
{\bf 3.\quad} Very often, Hamiltonian actions of a Lie algebra (or of a Lie group) 
on the cotangent bundle $T^*N$ to a smooth manifold $N$ encountered in applications 
come from an action of this Lie algebra (or of this Lie group) on the 
manifold $N$ itself. Proposition \ref{LiftedHamiltonianAction} explains how an action
on $N$ can be lifted to $T^*N$ into a Hamiltonian action.
\end{rmks}

\begin{prop}\label{LiftedHamiltonianAction}
Let $\varphi:{\mathcal G}\to A^1(N)$ be an action of a finite-dimensional Lie
algebra $\mathcal G$ on a smooth manifold $N$. Let 
$\widehat\varphi:{\mathcal G}\to A^1(T^*N)$ be the map wich associates to each 
$X\in{\mathcal G}$ the canonical lift to $T^*N$ of the vector field $\varphi(X)$
on $N$ (\ref{CanonicalLiftDef}). The map $\widehat\varphi$ is a Hamiltonian action of
$\mathcal G$ on $(T^*N,\d \eta_N)$ (where $\eta_N$ is the Liouville form and
$\d\eta_N$ the canonical symplectic form on $T^*N$). For each $X\in{\mathcal G}$, 
the smooth function $f_X:T^*N\to\RR$
 $$f_X(\xi)=\Bigl\langle\xi, \varphi(X)\bigl(\pi_N(\xi)\bigr)\Bigr\rangle
            =\mathi\bigl(\widehat\varphi(X)\bigr)\eta_N(\xi)\,,\quad
            \xi\in T^*N\,,
 $$
is a Hamiltonian for the vector field $\widehat\varphi(X)$. Moreover, for each pair
$(X,Y)$ of elements in ${\mathcal G}$,
 $$\{f_X,f_Y\}=f_{[X,Y]}\,.$$
\end{prop}

\begin{proof} Proposition \ref{CanonicalLiftHam} proves that for each
 $X\in{\mathcal G}$ the vector field $\widehat\varphi(X)$ is Hamiltonian and admits the function $f_X$ as Hamiltonian. This Proposition also shows that $f_X$ is given by the two equivalent expressions
 $$f_X(\xi)=\Bigl\langle\xi, \varphi(X)\bigl(\pi_N(\xi)\bigr)\Bigr\rangle
            =\mathi\bigl(\widehat\varphi(X)\bigr)\eta_N(\xi)\,,\quad
            \xi\in T^*N\,.
 $$
Let $(X,Y)$ be a pair of elements in $\mathcal G$. Since the vector fields
$\widehat\varphi(X)$ and $\widehat\varphi(Y)$ admit $f_X$ and $f_Y$ as 
Hamiltonians, Lemma \ref{LemmaPoissonBracket} shows that $\bigl[\widehat\varphi(X),\widehat\varphi(Y)\bigr]$ admits
$\{f_X,f_Y\}$ as Hamiltonian. We have
 $$\{f_X,f_Y\}={\mathcal L}\bigl(\widehat\varphi(X)\bigr)f_Y
                   ={\mathcal L}\bigl(\widehat\varphi(X)\bigr)\circ\mathi
                                \bigl(\widehat\varphi(Y)\bigr)\eta_N
                   =\mathi\Bigl[\widehat\varphi(X),\widehat\varphi(Y)\Bigr]\eta_N
 $$
since, using \ref{PropCanonicalLifts}, we see that 
${\mathcal L}\bigl(\widehat\varphi(X)\bigr)\eta_N=0$. Therefore,
for each $\xi\in T^*N$,
 $$\{f_X,f_Y\}(\xi)=\Bigl\langle\xi,T\pi_N\bigl([\widehat\varphi(X),
                    \widehat\varphi(Y)](\xi)\bigr)\Bigr\rangle
                   =\bigl\langle\xi,[X,Y]\circ\pi_N(\xi) 
                    \bigr\rangle
                   =f_{[X,Y]}(\xi)  
 $$  
since $T\pi_N\bigl([\widehat\varphi(X),
                    \widehat\varphi(Y)](\xi)\bigr)
      =[X,Y]\circ\pi_N(\xi)$. Since $\{f_X,f_Y\}=f_{[X,Y]}$, 
the corresponding Hamiltonian vector fields
$\bigl[\widehat\varphi(X),\widehat\varphi(Y)\bigr]$ and 
$\widehat\varphi\bigl([X,Y]\bigr)$ are equal. In other words, $\widehat\varphi$ is a Lie algebra action of $\mathcal G$ on $(T^*N,\d\eta_N)$. 
\end{proof}

\begin{prop}\label{MomentumMap} 
Let $\varphi$ be a Hamiltonian action 
of a Lie algebra $\mathcal G$ on a Poisson manifold 
$(M,\Lambda)$. Let  ${\mathcal G}^*$ be the dual space of 
$\mathcal G$. There exists a smooth map $J:M\to{\mathcal G}^*$ 
such that for each $X\in{\mathcal G}$ the corresponding Hamiltonian
vector field $X_M$  has the function $J_X:M\to\RR$, defined by
 $$J_X(x)=\bigl\langle J(x),X\bigr\rangle\,,\quad\text{with}\ x\in M\,,$$
as Hamiltonian. 
\par\smallskip

Such a map $J:M\to{\mathcal G}^*$ is called a \emph{momentum map} 
for the Hamiltonian Lie algebra action $\varphi$. When $\varphi$ is the Lie algebra action associated to a Hamiltonian action $\Phi$ of a Lie group $G$ on the Poisson manifold $(M,\Lambda)$, $J$ is called a
\emph{momentum map} for the Hamiltonian Lie group action $\Phi$.
\end{prop}

\begin{proof}
Let $(e_1,\ldots,e_p)$ be a basis of the Lie algebra $\mathcal G$ 
and$(\varepsilon^1,\ldots,\varepsilon^p)$ be the dual basis of 
${\mathcal G}^*$. Since $\varphi$ is Hamiltonian, for each 
$i$ ($1\leq i\leq p)$ there exists a Hamiltonian $J_{e_i}:M\to\RR$ for the Hamiltonian vector field $\varphi(e_i)$. The map $J:M\to{\mathcal G}$ defined by
 $$J(x)=\sum_{i=1}^p J_{e_i}\varepsilon^i\,,\quad x\in M\,,$$
is a momentum map for $\varphi$.
\end{proof}

The momentum map was introduced by Jean-Marie Souriau 
\cite{Souriau69} and, in the Lagrangian formalism, by Stephen Smale
\cite{Smale70}.

\subsection{Some properties of momentum maps}

\begin{prop}\label{Theta}
Let $\varphi$ be a Hamiltonian action 
of a Lie algebra $\mathcal G$ on a Poisson manifold 
$(M,\Lambda)$, and $J:M\to{\mathcal G}^*$ be a momentum map for that
action. For any pair $(X,Y)\in{\mathcal G}\times{\mathcal G}$, the smooth function 
$\widetilde\Theta(X,Y):M\to\RR$ defined by
 $$\widetilde\Theta(X,Y)=\{J_X,J_Y\}-J_{[X,Y]}$$
is a Casimir of the Poisson algebra $C^\infty(M,\RR)$, which satisfies,
for all $X$, $Y$ and $Z\in{\mathcal G}$,
 $$\widetilde\Theta\bigl([X,Y],Z\bigr)
   +\widetilde\Theta\bigl([Y,Z],X\bigr)
    +\widetilde\Theta\bigl([Z,X],Y\bigr)=0\,.\eqno(1)
 $$
\par\smallskip
When the Poisson manifold $(M,\Lambda)$ is in fact a connected symplectic manifold $(M,\omega)$, for any pair $(X,Y)\in{\mathcal G}\times{\mathcal G}$
the function $\widetilde\Theta(X,Y)$ is constant on $M$, and the map 
$\widetilde\Theta:{\mathcal G}\times{\mathcal G}\to\RR$ is a skew-symmetric bilinear form, which satisfies the above identity $(1)$.  
\end{prop}

\begin{proof}
Since $J_X$ and $J_Y$ are Hamiltonians for the Hamiltonian vector fields
$\varphi(X)$ and $\varphi(Y)$, the Poisson bracket $\{J_X,J_Y\}$ is a Hamiltonian for $\bigl[\varphi(X),\varphi(Y)]$. Since 
$\varphi:{\mathcal G}\to A^1(M)$ is a Lie algebras homomorphism,
$\bigl[\varphi(X),\varphi(Y)]=\varphi\bigl([X,Y]\bigr)$, and $J_{[X,Y]}$ is a Hamiltonian for this vector field. We have two different Hamiltonians for the same Hamiltonian vector field. Their difference $\widetilde\Theta(X,Y)$ is therefore a Casimir of the Poisson algebra $C^\infty(M,\RR)$.
\par\smallskip

Let $X$, $Y$ and $Z$ be three elements in $\mathcal G$. We have
 \begin{equation*}
 \begin{split}
  \widetilde\Theta\bigl([X,Y],Z\bigr)&=\{J_{[X,Y]},J_Z\}-J_{\bigl[[X,Y],Z\bigr]}\\
                           &=\bigl\{\{J_X,J_Y\}-\widetilde\Theta(X,Y),J_Z\bigr\}
                                  -J_{\bigl[[X,Y],Z\bigr]}\\
                           &=\bigl\{\{J_X,J_Y\},J_Z\bigr\}
                                  -J_{\bigl[[X,Y],Z\bigr]}
 \end{split}
 \end{equation*}
since $\widetilde\Theta(X,Y)$ is a Casimir of the Poisson algebra $C^\infty(M,\RR)$.
Similarly
 \begin{equation*}
  \begin{split}
   \widetilde\Theta\bigl([Y,Z],X\bigr)&=\bigl\{\{J_Y,J_Z\},J_X\bigr\}
                                  -J_{\bigl[[Y,Z],X\bigr]}\,,\\
   \widetilde\Theta\bigl([Z,X],Y\bigr)&=\bigl\{\{J_Z,J_X\},J_Y\bigr\}
                                  -J_{\bigl[[Z,X],Y\bigr]}\,.
 \end{split}
 \end{equation*}
Adding these three terms and using the fact that the Poisson bracket of functions and the bracket in the Lie algebra $\mathcal G$ both satisfy the Jacobi identity, we see that $\widetilde\Theta$ satisfies $(1)$.  
\par\smallskip

When $(M,\Lambda)$ is in fact a connected symplectic manifold $(M,\omega)$, the only Casimirs of the Poisson algebra $C^\infty(M,\RR)$ are the constants, and 
$\widetilde\Theta$ becomes a bilinear skew-symmetric form on $\mathcal G$. 
\end{proof}

\begin{defi}\label{SymplecticCocycle2}
Under the assumptions of Proposition \ref{Theta}, the 
skew-symmetric bilinear map $\widetilde\Theta$, defined on ${\mathcal G}\times{\mathcal G}$ and taking its values in the space of Casimirs of the Poisson algebra
$C^\infty(M,\RR)$ (real-valued when the Poisson manifold $(M,\Lambda)$ 
is in fact a connected symplectic manifold $(M,\omega)$), 
is called the \emph{symplectic cocycle of the Lie algebra} $\mathcal G$ associated to the momentum map $J$. 
\end{defi}

\begin{rmk}\label{SymplecticCocycle3}
Under the assumptions of Proposition \ref{Theta}, let us assume in addition that the Poisson manifold $(M,\Lambda)$ is in fact a connected symplectic 
manifold $(M,\omega)$. The symplectic cocycle $\widetilde\Theta$ is then a real-valued skew-symmetric bilinear form on $\mathcal G$. Therefore it is a symplectic cocycle in the sense of \ref{SymplecticCocycle}. Two different interpretations of this cocycle can be given.
\par\smallskip
 
\begin{enumerate} 

\item{}
Let $\Theta:{\mathcal G}\to{\mathcal G}^*$ 
be the map such that, for all $X$ and $Y\in{\mathcal G}$
 $$\bigl\langle\Theta(X),Y\bigr\rangle=\widetilde\Theta(X,Y)\,.$$
Written for $\Theta$, Equation $(1)$ of \ref{Theta} becomes
 $$\Theta\bigl([X,Y]\bigr)=\ad^*_{-X}\bigl(\Theta(Y)\bigr)
                          -\ad^*_{-Y}\bigl(\Theta(X)\bigr)\,,
                           \quad X\ \hbox{and}\ Y\in{\mathcal G}\,.
 $$  
The map  $\Theta$ is therefore the one-cocycle of the Lie algebra $\mathcal G$ 
with values in ${\mathcal G}^*$, for the coadjoint representation 
(\ref{AdjointAndCoadjointRepresentations}) $X\mapsto\ad^*_{-X}$ of $\mathcal G$, associated to the affine action
of $\mathcal G$ on its dual
 $$a_\Theta(X)(\zeta)=\ad^*_{-X}(\zeta)+\Theta(X)\,,\quad X\in{\mathcal G}\,,\ \zeta\in
  {\mathcal G}^*\,,
 $$
in the sense of \ref{LinearParta}. The reader is referred to the book \cite{HiltonStammbach}
for a more thorough discussion of the cohomology theories of 
Lie groups and Lie algebras.

\item{}
Let $G$ be a Lie group whose Lie algebra is $\mathcal G$. The skew-symmetric
bilinear form $\widetilde\Theta$ on ${\mathcal G}=T_eG$ can be extended, either by left translations or by right translations, into a left invariant (or a right invariant) closed differential two-form on $G$, since the identity $(1)$ of \ref{Theta}  
means that its exterior differential $\d\widetilde\Theta$ vanishes. In other words, 
$\widetilde\Theta$ is a $2$-cocycle for the restriction of the de~Rham cohomology of 
$G$ to left (or right) invariant differential forms.

\end{enumerate}
\end{rmk}

\begin{prop}\label{LiftedHamiltonianAction2}
Let $\varphi:{\mathcal G}\to A^1(N)$ be an action of a finite-dimensional Lie
algebra $\mathcal G$ on a smooth manifold $N$, and let
$\widehat\varphi:{\mathcal G}\to A^1(T^*N)$ be the Hamiltonian action
of $\mathcal G$ on $(T^*N ,\d\eta_N)$ introduced in Proposition 
\ref{LiftedHamiltonianAction}.
The map $J:T^*N\to{\mathcal G}^*$ defined by
 $$\bigl\langle J(\xi),X\bigr\rangle=\mathi\bigl(\widehat\varphi(X)\bigr)\eta_N(\xi)\,,
   \quad X\in{\mathcal G}\,,\ \xi\in T^*N\,,
 $$ 
is a momentum map for the action $\widehat\varphi$ which satisfies, for all $X$ and 
$Y\in{\mathcal G}$,
 $$\bigl\{J_X,J_Y\bigr\}=J_{[X,Y]}\,.
 $$
In other words, the symplectic cocycle of $\mathcal G$ associated to $J$, in the sense of \ref{SymplecticCocycle2}, identically vanishes.
\end{prop}

\begin{proof}
These properties immediately follow from \ref{LiftedHamiltonianAction}.
\end{proof} 

\begin{theo}[First Emmy Noether's theorem in Hamiltonian form]\label{noether}
Let $\varphi$ be a Hamiltonian action of a Lie algebra $\mathcal G$ 
on a Poisson manifold $(M,\Lambda)$, $J:M\to{\mathcal G}^*$ be a momentum map for $\varphi$ and $H:M\to\RR$ be a smooth Hamiltonian.
If the action $\varphi$ leaves $H$ invariant, that means if 
 $${\mathcal L}\bigl(\varphi(X)\bigr)H=0\quad\hbox{for any}\ 
    X\in{\mathcal G}\,,
 $$
the momentum map $J$ is a ${\mathcal G}^*$-valued first integral 
(\ref{FirstIntegrals})
of the Hamiltonian vector field $\Lambda^\sharp(\d H)$, which means that it
keeps a constant value along each integral curve of that vector field.
\end{theo}

\begin{proof} 
For any $X\in{\mathcal G}$, let $J_X:M\to\RR$ be the function
$x\mapsto\bigl\langle J(x),X\bigr\rangle$. Let $t\mapsto\psi(t)$ be an
integral curve of the Hamiltonian vector field $\Lambda^\sharp(\d H)$. We have
 \begin{equation*}
  \begin{split}
   \frac{\d}{\d t}\Bigl(J_X\bigl(\psi(t)\bigr)\Bigr)
    &={\mathcal L}\bigl(\Lambda^\sharp(\d H)\bigr)
        \bigl(J_X\bigr)\bigl(\psi(t)\bigr)
       =\Lambda\bigl(\d H,\d J_X\bigr)(\psi(t))\\
    &=-{\mathcal L}\Bigl(\Lambda^\sharp\bigl(\d J_X\bigr)
       \Bigr)H
     =-{\mathcal L}\bigl(\varphi(X)\bigr)H=0\,.
    \end{split}
   \end{equation*}
Therefore, for any $X\in{\mathcal G}$, the derivative of
$\langle J,X\rangle\bigl(\psi(t)\bigr)$  with respect to the parameter $t$ of the parametrized curve $t\mapsto\psi(t)$ vanishes identically, which means that $J$ keeps a constant value along that curve.
\end{proof}

The reader will find in the book by Yvette Kosmann-Schwarzbach 
\cite{Kosmann2011} a very nice exposition of the history and scientific applications of the Noether's theorems.

\begin{prop}\label{SymplecticOrthogonality1}
Let $\varphi$ be a Hamiltonian action of a Lie algebra $\mathcal G$ 
on a Poisson manifold $(M,\Lambda)$ and $J:M\to{\mathcal G}^*$ 
be a momentum map for that action. Let $S$ be a symplectic leaf of 
$(M,\Lambda)$ and $\omega_S$ be its symplectic form.

\par\smallskip\noindent
{\bf 1.\quad} For each 
$x\in S$, in the symplectic vector space $\bigl(T_xS,\omega_S(x)\bigr)$, each of the two vector subspaces $T_xS\cap\ker(T_xJ)$ and
$\bigl\{\,\varphi(X)(x)\,;\,X\in{\mathcal G}\,\}$  is the symplectic orthogonal of the other.

\par\smallskip\noindent
{\bf 2.\quad} For each $x\in S$, $T_xJ(T_xS)$ is the annihilator of
the isotropy subalgebra 
${\mathcal G}_x=\bigl\{X\in{\mathcal G};\phi(X)(x)=0\bigr\}$ of $x$.
\end{prop} 

\begin{proof} Let $v\in T_xS$. For each $X\in{\mathcal G}$ we have
 $$\omega_S\bigl(v,\varphi(X)(x)\bigr)
   =\bigl\langle\ d\langle J,X\rangle(x),v\bigr\rangle
  =\bigl\langle T_xJ(v),X\bigr\rangle\,. 
 $$
Therefore a vector $v\in T_xS$ belongs to 
$\orth\bigl\{\,\varphi(X)(x)\,;\,X\in{\mathcal G}\,\}$ if and only if
$T_xJ(v)=0$. In other words, in the symplectic vector space $\bigl(T_xS,\omega_S(x)\bigr)$, $T_xS\cap\ker(T_xJ)$ is the symplectic orthogonal of $\bigl\{\,\varphi(X)(x)\,;\,X\in{\mathcal G}\,\}$. Of course, conversely $\bigl\{\,\varphi(X)(x)\,;\,X\in{\mathcal G}\,\}$
is the symplectic orthogonal of 
$T_xS\cap\ker(T_xJ)$. 
\par\smallskip

The same formula shows that $\bigl\langle T_xJ(v),X\bigr\rangle=0$ 
for all $v\in T_xS$ if and only if $X\in{\mathcal G}_x$.
\end{proof}

\begin{rmk} Under the assumptions of \ref{SymplecticOrthogonality1}, when $\varphi$
is the Lie algebra action associated to a Hamiltonian action $\Phi$
of a Lie group $G$, the vector space $\bigl\{\,\varphi(X)(x)\,;\,X\in{\mathcal G}\,\}$ is the space tangent at $x$ to the $G$-orbit of this point.
\end{rmk}

\begin{coro}\label{SymplecticOrthogonal}
Let $\varphi$ be a Hamiltonian action of a Lie algebra $\mathcal G$ 
on a symplectic manifold $(M,\omega)$ and $J:M\to{\mathcal G}^*$ 
be a momentum map for that action. 

\par\smallskip\noindent
{\bf 1.\quad} For each 
$x\in M$, in the symplectic vector space $\bigl(T_xM,\omega(x)\bigr)$ each of the two vector subspaces $\ker(T_xJ)$ and
$\bigl\{\,\varphi(X)(x)\,;\,X\in{\mathcal G}\,\}$  is the symplectic orthogonal of the other.

\par\smallskip\noindent
{\bf 2.\quad} For each $x\in M$, $T_xJ(T_xM)$ is the annihilator
of the isotropy subalgebra ${\mathcal G}_x=\{X\in{\mathcal G}; \varphi(X)(x)=0\}$  of $x$. 
\end{coro}

\begin{proof}
These assertions both follow immediately from \ref{SymplecticOrthogonality1} 
since the symplectic leaves of $(M,\omega)$ are its connected components.
\end{proof}

\begin{prop}\label{LieGroupCocycle}
Let $\Phi$ be a Hamiltonian action of a Lie group $G$ 
on a connected symplectic manifold $(M,\omega)$ and $J:M\to{\mathcal G}^*$ 
be a momentum map for that action. There exists a unique action $A$ of the
Lie group $G$ on the dual ${\mathcal G}^*$ of its Lie algebra for which
the momentum map $J$ is equivariant, that means satisfies for each
$x\in M$ and $g\in G$
 $$J\bigl(\Phi_g(x)\bigr)=A_g\bigl(J(x)\bigr)\,.$$
The action $A$ is an action on the left (respectively, on the right) 
if $\Phi$ is an action on the left (respectively, on the right), and its expression is
\begin{equation*}
\begin{cases}
A(g,\xi)=\Ad^*_{g^{-1}}(\xi)+\theta(g)& \text{if $\Phi$ is an action on the left,}\\
A(\xi,g)=\Ad^*_{g}(\xi) -\theta(g^{-1})& \text{if $\Phi$ is an action on the right,}
\end{cases}\quad g\in G\,,\ \xi\in{\mathcal G}^*\,.
\end{equation*}
The map $\theta:G\to{\mathcal G}^*$ is called the \emph{symplectic cocycle
of the Lie group} $G$ associated to the momentum map $J$.
\end{prop}

\begin{proof}
Let us first assume that $\Phi$ is an action on the left. For each 
$X\in{\mathcal G}$ the associated fundamental vector field $X_M$ is Hamiltonian and the function $J_X:M\to \RR$ defined by
 $$J_X(x)=\bigl\langle J(x),X\bigr\rangle\,,\quad x\in M\,,
 $$
is a Hamiltonian for $X_M$. We know by the characterizations \ref{exo} 
of Poisson actions
that $(\Phi_{g^{-1}})_*(X_M)$, the direct image of $X_M$ by the diffeomorphism 
$\Phi_{g^{-1}}$,
is a Hamiltonian vector field for which the function $J_X\circ\Phi_g$ is a Hamiltonian. Proposition \ref{DirectImageFundamental} shows that 
$(\Phi_{g^{-1}})_*(X_M)$ is the fundamental vector field associated to
$\Ad_{g^{-1}}(X)$, therefore has the function 
 $$x\mapsto \bigl\langle J(x), \Ad_{g^{-1}}(X)\bigr\rangle 
           =\bigl\langle \Ad^*_{g^{-1}}\circ J(x),X\bigr\rangle
 $$
as a Hamiltonian. The difference between these two Hamiltonians for the same Hamiltonian vector field is a constant since $M$ is assumed to be connected.
Therefore the expression 
 $$\bigl\langle J\circ\Phi_g(x)
                     -\Ad^*_{g^{-1}}\circ J(x),X\bigr\rangle
 $$  
does not depend on $x\in M$, and depends linearly on $X\in{\mathcal G}$ 
(and of course smoothly depends on $g\in G$). We can therefore define a smooth map
$\theta:G\to {\mathcal G}^*$ by setting
 $$\theta(g)=J\circ \Phi_g - \Ad^*_{g^{-1}}\circ J\,,\quad g\in G\,.
 $$
It follows that the map $a:G\times{\mathcal G}^*\to {\mathcal G}^*$,
 $$a(g,\xi)=\Ad^*_{g^{-1}}(\xi)+\theta(g)$$
is an action on the the left of the Lie group $G$ on the dual ${\mathcal G}^*$ of its Lie algebra, which renders the momentum map $J$ equivariant.
\par\smallskip

The case when $\Phi$ is an action on the right easily follows by observing that $(g,x)\mapsto\Phi(x,g^{-1})$ is a Hamiltonian action on the left whose
momentum map is the opposite of that of $\Phi$.
\end{proof}

\begin{prop}\label{CocyclePropertytheta}
Under the same assumptions as those of Proposition \ref{LieGroupCocycle}, 
the map
$\theta:G\to{\mathcal G}^*$ satisfies, for all $g$ and $h\in G$,
 $$\theta(gh)=\theta(g)+\Ad^*_{g^{-1}}\bigl(\theta(h)\bigr)\,.$$  
\end{prop}

\begin{proof}
In Proposition \ref{LieGroupCocycle}, the cocycle $\theta$ introduced for an action on the right $\Psi:M\times G\to M$ was the cocycle of the corresponding action on the left $\Phi:G\times M\to M$ defined by
$\Phi(g,x)=\Psi(x,g^{-1})$. We can therefore consider only the case when
$\Phi$ is an action on the left.
\par\smallskip

Let $g$ and $h\in G$.
We have
\begin{align*}
\theta(gh)
&=J\bigl(\Phi(gh,x)\bigr)
-\Ad^*_{(gh)^{-1}}J(x)\\
&=J\Bigl(\Phi\bigl(g,\Phi(h,x)\bigr)\Bigr)
-\Ad^*_{g^{-1}}\circ\Ad^*_{h^{-1}}J(x)\\
&=\theta(g)+\Ad^*_{g^{-1}}\Bigl(
J\bigl(\Phi(h,x)\bigr)-\Ad^*_{h^{-1}}J(x)
\Bigr)\\
&=\theta(g)+\Ad^*_{g^{-1}}\theta(h)\,.\qedhere
\end{align*}
\end{proof}

\begin{prop}\label{MomentumPoissonMap} 
Let $\Phi$ be a Hamiltonian action of a Lie group $G$ 
on a connected symplectic manifold $(M,\omega)$ and 
$J:M\to{\mathcal G}^*$ 
be a momentum map for that action. The symplectic cocycle $\theta:G\to{\mathcal G}^*$ of the Lie group $G$ introduced in Proposition 
\ref{LieGroupCocycle} and the symplectic cocycle 
$\Theta:{\mathcal G}\to{\mathcal G}^*$ of its Lie algebra 
$\mathcal G$ introduced in Definition \ref{SymplecticCocycle2}
and Remark \ref{SymplecticCocycle3} are related by
 $$\Theta=T_e\theta\,,$$
where $e$ is the neutral element of $G$, the Lie algebra ${\mathcal G}$ being identified with $T_eG$ and the tangent space at ${\mathcal G}^*$ at its origin being identified with ${\mathcal G}^*$. Moreover $J$ is a Poisson map when ${\mathcal G}^*$ is endowed with

\begin{itemize}

\item{} its canonical Poisson structure modified by the 
symplectic cocycle $\Theta$ (defined in \ref{SymplecticCocycle}) 
if $\Phi$ is an action on the right,

\item{} the opposite of this Poisson structure if $\Phi$ is an action on the left.
\end{itemize}  
\end{prop}

\begin{proof} As in the proof of Proposition 
\ref{CocyclePropertytheta}, whe have only to consider the case when 
$\Phi$ is an action on the left. The map
which associates to each $X\in{\mathcal G}$ the fundamental vector field
$X_M$ is a Lie algebras homomorphism when $\mathcal G$ is endowed with the Lie algebra structure of \emph{right invariant vector fields} on the Lie group
$G$. We will follow here the more common convention, in which $\mathcal G$ is endowed with the Lie algebra structure of \emph{left invariant vector fields}
on $G$. With this convention the map $X\mapsto X_M$ is a Lie algebras antihomomorphism and we must change a sign in the definition of $\widetilde\Theta$
given in Proposition \ref{Theta} and take 

 $$\widetilde\Theta(X,Y)=\bigl\langle\Theta(X),Y\bigr\rangle 
        =\{J_X,J_Y\}+J_{[X,Y]}\,,\quad X\ \text{and}\ Y\in{\mathcal G}\,.
 $$
We have, for any $x\in M$,
\begin{align*}
\{J_X,J_Y\}(x)
&=\omega(X_M,Y_M)(x)
=i(X_M)d\bigl(\langle J,Y\rangle\bigr)(x)\\
&=\frac{\d}{\d t}\,\Bigl\langle J\bigl(\Phi(\exp(tX),x\bigr),
                Y\Bigr\rangle{\Bigm|}_{t=0}\\
&=\frac{\d}{\d t}\,\Bigl\langle \Ad^*_{\exp(-tX)}
J(x)+\theta\bigl(\exp(tX)\bigr),Y\Bigr\rangle
{\Bigm|}_{t=0}\\
&=\bigl\langle J(x),-[X,Y]\bigr\rangle
+\bigl\langle T_e\theta(X),Y\bigr\rangle\\
&=-J_{[X,Y]}(x)+\bigl\langle T_e\theta(X),Y\bigr\rangle\,.
\end{align*}
We see that $\Theta=T_e\theta$. 
Moreover, the elements $X$ and $Y$ in ${\mathcal G}$ can be considered 
as linear functions on ${\mathcal G}^*$. Their Poisson bracket, when
${\mathcal G}^*$ is equipped with its canonical Poisson structure
modified by $\widetilde\Theta$, is
 $$\{X,Y\}_{\widetilde\Theta}(\xi)
  =\bigl\langle \xi,[X,Y]\bigr\rangle-\widetilde\Theta(X,Y)\,.$$
The formula 
 $\{J_X,J_Y\}(x)=-J_{[X,Y]}(x)+\widetilde\Theta(X,Y)$
can be read as
 $$\{X\circ J,Y\circ J\}(x)=-\{X,Y\}_{\widetilde\Theta}\circ J(x)\,.$$
Since the value taken at a point by the Poisson bracket of two functions only depends on the values of the differentials of these two functions at that point, this result proves that $J$ is a Poisson map when ${\mathcal G}^*$ is equipped with the opposite of the Poisson bracket $\{\ ,\ \}_{\widetilde\Theta}$.  
\end{proof}

\begin{rmks}
Let $\Phi$ be a Hamiltonian action on the left of a Lie group $G$ 
on a connected symplectic manifold $(M,\omega)$, $J:M\to{\mathcal G}^*$ 
be a momentum map for that action and $\theta:G\to{\mathcal G}^*$ be the
symplectic cocycle of the Lie group $G$ introduced in Proposition 
\ref{LieGroupCocycle}.

\par\smallskip\noindent
{\bf 1.\quad} The symplectic cocycle $\theta:G\to{\mathcal G}^*$ is the Lie group
one-cocycle with values in ${\mathcal G}^*$, for the coadjoint representation,
associated to the affine representation $A:G\to\Aff({\mathcal G}^*)$,
 $$A(g)(\zeta)=\Ad^*_{g^{-1}}(\zeta)+\theta(g)\,,\quad \zeta\in{\mathcal G}^*\,,
 $$
in the sense of \ref{LinearPartA}.

\par\smallskip\noindent
{\bf 2.\quad} If instead of $J$ we take for momentum map
 $$J'(x)=J(x)-c\,,\quad x\in M\,,
 $$
where $c\in{\mathcal G}^*$ is constant, the symplectic cocycle $\theta$
is replaced by
 $$\theta'(g)=\theta(g)+\Ad^*_{g^{-1}}(c)-c\,.$$
The map $\theta'-\theta$ is a one-coboundary of $G$ with values in 
${\mathcal G}^*$ for the coadjoint representation (\ref{coboundary}). Therefore
the cohomology class of the symplectic cocycle $\theta$ only depends on the
Hamiltonian action $\Phi$, not on the choice of its momentum map $J$. This property
is used by Jean-Marie Souriau (\cite{Souriau69}, chapter III, p. 153) to offer 
a very nice cohomological interpretation of the total mass of a classical 
(non-relativistic) isolated mechanical system. He proves that the space of all 
possible motions of the system is a symplectic manifold on which the Galilean 
group acts by a Hamiltonian action. The dimension of the symplectic cohomology 
space of the Galilean group (the quotient of the space of symplectic one-cocycles 
by the space of symplectic one-coboundaries) is equal to $1$. The cohomology 
class of the symplectic cocycle associated to a momentum map of the action 
of the Galilean group on the space of motions of the system is interpreted as 
the \emph{total mass} of the system.     
\end{rmks}

\subsubsection{Other properties of the momentum map} The momentum map has several other very remarkable properties. Michael Atiyah
\cite{Atiyah82}, Victor Guillemin and Shlomo Sternberg
\cite{GuilleminSternberg82, GuilleminSternberg84} 
have shown that
the image of the momentum map of a Hamiltonian action of a torus
on a compact symplectic manifold is a convex polytope. Frances Kirwan
\cite{Kirwan84} adapted this result when the torus is replaced by any
compact Lie group. Thomas Delzant \cite{Delzant88} has shown that the convex polytope which is the image of a Hamiltonian action of a torus on a compact symplectic manifold determines this manifold.  

\subsection{Actions of a Lie group on its cotangent bundle}
\label{ActionsOnCotangentBundle}

In this section $G$ is a Lie group, $\mathcal G$ is its Lie algebra and 
${\mathcal G}^*$ is the dual space of $\mathcal G$. The Liouville one-form
on $T^*G$ is denoted by $\eta_G$. 
\par\smallskip

The group composition law $m:G\times G\to G$, $m(g,h)=gh$, can be seen as an
action of $G$ on itself either on the left, or on the right. For each 
$g\in G$ we will denote by $L_g:G\to G$ and $R_g:G\to G$
 the diffeomorphisms
 $$L_g (h)=gh\,,\quad R_g(h)=hg\,,\quad h\in G\,.$$
called, respectively, the \emph{left translation} and the \emph{right translation} of $G$ by $g$.

\begin{defis}\label{CanonicalLiftedAction}
The \emph{canonical lifts} to the tangent bundle $TG$ of the actions of
$G$ on itself by left translations (respectively, by right translations) are, 
repectively, the maps
${\overline{\mathstrut L}}:G\times TG\to TG$ and ${\overline R}:TG\times G\to TG$ 
 $${\overline{\mathstrut L}}(g,v)=TL_g(v)\,,\quad {\overline R}(v,g)=TR_g(v)\,,\quad g\in G\,,
   \quad v\in TG\,.
 $$
The \emph{canonical lifts} to the cotangent bundle $T^*G$ of the actions of
$G$ on itself by left translations (respectively, by right translations) are, respectively, the maps
 $\widehat L:G\times T^*G\to T^*G$ and 
$\widehat R:T^*G\times G\to T^*G$ 
 $$\widehat L(g,\xi)=\bigl(TL_{g^{-1}}\bigr)^T(\xi)\,,\quad
    \widehat R(\xi,g)=\bigl(TR_{g^{-1}}\bigr)^T(\xi)\,,
   \quad g\in G\,,\quad \xi\in T^*G\,.$$
We have denoted by $\bigl(TL_{g^{-1}}\bigr)^T$ and 
$\bigl(TR_{g^{-1}}\bigr)^T$ the transposes of the vector bundles morphisms
$TL_{g^{-1}}$ and $TR_{g^{-1}}$, respectively.
\end{defis}

\begin{prop}
The canonical lifts to the tangent bundle and to the cotangent bundle
of the actions of the Lie group $G$ on itself by left translations
(respectively, by right translations) are actions on the left (respectively, on the right)
of $G$ on its tangent bundle and on its cotangent bundle, which project onto 
the actions of $G$ on itself by left translations (respectively, by right translations). It means that for all $g\in G$ and $v\in TG$
 $$\tau_G\bigl({\overline{\mathstrut L}}(g,v)\bigr)=L_g\bigl(\tau_G(v)\bigr)\,,\quad
   \tau_G\bigl({\overline R}(v,g)\bigr)=R_g\bigl(\tau_G(v)\bigr)\,,
 $$
and that for all $g\in G$ and $\xi\in T^*G$
 $$\pi_G\bigl(\widehat L(g,\xi)\bigr)=L_g\bigl(\pi_G(\xi)\bigr)\,,\quad
   \pi_G\bigl(\widehat R(\xi,g)\bigr)=R_g\bigl(\pi_G(\xi)\bigr)\,.
 $$
\end{prop}

\begin{proof}
It is an easy verification that the properties of actions are indeed satisfied by the maps ${\overline{\mathstrut L}}$, ${\overline R}$, $\widehat L$
and $\widehat R$, which is left to the reader.
\end{proof}

\begin{theo}\label{ActionsOfGonTstarG}
The canonical lifts to the cotangent bundle $\widehat L$
and $\widehat R$ of the actions of the Lie group $G$ on itself
by translations on the left and on the right are two Hamiltonian actions
of $G$ on the symplectic manifold $(T^*G,\d\eta_G)$.
The maps $J^L:T^*G\to{\mathcal G}^*$ and $J^R:T^*G\to{\mathcal G}^*$ defined, for each $\xi\in T^*G$, by
  $$J^L(\xi)=\widehat R\bigl(\xi,\pi_G(\xi)^{-1}\bigr)\,,\quad
    J^R(\xi)=\widehat L\bigl(\pi_G(\xi)^{-1},\xi\bigr)
 $$ 
are momentum maps for the actions $\widehat L$ and $\widehat R$, respectively. 
\par\smallskip
Moreover, the map $J^L$ is constant on each orbit of the action $\widehat R$,
the map $J^R$ is constant on each orbit of the action $\widehat L$ and for each $\xi \in T^*G$ each of the tangent spaces at $\xi$ to the orbits
$\widehat L(G,\xi)$ and $\widehat R(\xi,G)$ is the symplectic orthogonal
of the other. The maps $J^L:T^*G\to{\mathcal G}^*$ and 
$J^R:T^*G\to{\mathcal G}^*$ are Poisson maps when $T^*G$ is equipped with
the Poisson structure associated to its canonical symplectic structure
and when ${\mathcal G}^*$ is equipped, respectively, with its canonical Poisson structure (\ref{CanonicalPoissonStructure}) and with the opposite of its canonical Poisson structure.
\end{theo}
 
\begin{proof} For each $X\in {\mathcal G}$, let $X^L_G$ and $X^R_G$ be 
the fundamental vector fields on $G$ associated to $X$ for the actions 
of $G$ on itself,
respectively by left and by right translations. Similarly, let
$X^L_{T^*G}$ and $X^R_{T^*G}$  be 
the fundamental vector fields on $T^*G$ associated to $X$ for the actions 
$\widehat L$ and $\widehat R$ of $G$ on $T^*G$ defined in 
\ref{CanonicalLiftedAction}. The reduced flows of $X^L$ and of $X^R$
are the maps
 $$\Phi^{X^L}(t,g)=\exp(tX)g\,,\quad \Phi^{X^R}(t,g)=g\exp(tX)\,,
    \quad t\in\RR\,,\ g\in G\,.$$
Therefore
 $$X^L(g)=T R_g(X)\,,\quad X^R(g)=TL_g(X)\,,\quad g\in G\,,$$
and we see that the fundamental vector fields $X^L_{T^*G}$ and $X^R_{T^*G}$
on $T^*G$ are the canonical lifts to the cotangent bundle of the vector fields
$X^L_G$ and $X^R_G$ on the Lie group $G$. Proposition 
\ref{CanonicalLiftHam} proves that $X^L_{T^*G}$ and $X^R_{T^*G}$ are Hamiltonian vector fields which admit as Hamiltonians, respectively,  
the maps
 $$J^L_X(\xi)=\Bigl\langle\xi,X^L_G\bigl(\pi_G(\xi)\bigr)\Bigr\rangle\,,\quad
   J^R_X(\xi)=\Bigl\langle\xi,X^R_G\bigl(\pi_G(\xi)\bigr)\Bigr\rangle\,,\quad
   \xi\in T^*G\,.
 $$
Replacing $X^L_G$ and $X^R_G$ by their expressions given above and using the
definitions of $\widehat R$ and $\widehat L$, we easily get the stated expressions for $J^L$ and $J^R$. These expressions prove that $J^L$ is constant on each orbit of the action $\widehat R$, and that
$J^R$ is constant on each orbit of the action $\widehat L$. 
\par\smallskip

The actions $\widehat L$ and $\widehat R$ being free, each of their orbits  is a smooth submanifold of $T^*G$ of dimension $\dim G$. The ranks of the maps
$J^L$ and $J^R$ are everywhere equal to $\dim G$ since their restrictions to
each fibre of $T^*G$ is a diffeomorphism of that fibre onto 
${\mathcal G}^*$. Therefore, for each $\xi\in T^*G$,
 $$\ker T_\xi J^L=T_{\xi}\bigl(\widehat R(\xi,G)\bigr)\,,\quad
   \ker T_\xi J^R=T_{\xi}\bigl(\widehat L(\xi,G)\bigr)\,.
 $$
Corollary \ref{SymplecticOrthogonal} proves that for each $\xi\in T^*G$ each of the two vector subspaces of $T_\xi(T^*G)$: 
 $$T_{\xi}\bigl(\widehat L(G,\xi)\bigr)\quad\text{and}\quad
   T_{\xi}\bigl(\widehat R(\xi,G)\bigr)
 $$
is the symplectic orthogonal of the other.
\par\smallskip

Finally, the fact that $J^L$ and $J^R$ are Poisson maps when 
${\mathcal G}$ is equipped with its canonical Poisson structure or its opposite is an easy consequence of Proposition \ref{QuotientPoisson}. 
\end{proof}

In \cite{LibermannMarle}, Chapter IV, Section 4, we proposed a generalization
of Proposition \ref{ActionsOfGonTstarG} taking into account a symplectic cocycle
$\theta:G\to{\mathcal G}^*$ in which the action $\widehat L:G\times T^*G\to T^*G$
remained unchanged while the action $\widehat R:T^*G\times G\to T^*G$ was modified.
Below we propose a more general and more symmetrical generalization.
The symplectic form
on $T^*G$ will be the sum of its canonical symplectic form $d\eta_G$ and of 
the pull-back by the canonical projection $\pi_G:T^*G\to G$ of a suitable 
closed two-form on $G$, deduced from $\theta$.
The actions  $\widehat L:G\times T^*G\to T^*G$ and
$\widehat R:T^*G\times G\to T^*G$ will be modified in the following way: 
for each $g\in G$, the map $\widehat L_g:T^*G\to T^*G$ will be composed 
with a translation in the fibres of $T^*G$, determined by addition of a 
\emph{right-invariant} one-form on $G$ depending of the element $g\in G$, deduced
from $\theta$; similarly, the map $\widehat R_g:T^*G\to T^*G$ will be composed 
with a translation in the fibres of $T^*G$, determined by addition of a 
\emph{left-invariant} one-form on $G$ depending of the element $g\in G$, deduced
from $\theta$. As the reader will see, it is possible to modify the action
$\widehat L$ and to keep $\widehat R$ unchanged, or to modify the action $\widehat R$ and to keep $\widehat L$ unchanged; in the first case, the momentum map 
$J^L:T^*G\to{\mathcal G}^*$ remains unchanged, while 
$J^R:T^*G\to{\mathcal G}^*$ must be modified; in the second case, it is
$J^R:T^*G\to{\mathcal G}^*$ which remains unchanged while 
$J^L:T^*G\to{\mathcal G}^*$ must be modified. It is even possible to simultaneously
modify both the actions $\widehat L$ and $\widehat R$; then we get a pair
of actions of $G$ on $T^*G$ depending on two real parameters.

\begin{theo}\label{GeneralizedActionsOfGonTstarG}
Let $G$ be a Lie group, $\theta:G\to{\mathcal G}^*$ be a symplectic cocycle
of $G$, $\Theta=T_e\theta:{\mathcal G}\to{\mathcal G}^*$ be the associated
symplectic cocycle of its Lie algebra $\mathcal G$, and 
$\widetilde\Theta:{\mathcal G}\times{\mathcal G}\to\RR$ be the skew-symmetric
bilinear form $\widetilde\Theta(X,Y)=\bigl\langle\Theta(X),Y\bigr\rangle$. Let
$\widetilde\Theta_L$ and $\widetilde\Theta_R$ be the differential two-forms on $G$,
respectively left-invariant and right-invariant, whose value at the neutral element
is $\widetilde\Theta$. The differential two-form on $T^*G$
 $$\omega_{T^*G}=\d\eta_G+\pi_G^*(\lambda_L\widetilde\Theta_L-\lambda_R\widetilde\Theta_R)
 \,,
 $$
where $\lambda_L$ and $\lambda_R$ are real constants and where $\eta_G$ 
is the Liouville form on $T^*G$, is a symplectic form on $T^*G$.
The formulae, in which $g\in G$, $\xi\in T^*G$,
 \begin{align*}
  \Phi^L(g,\xi)&=\widehat L_g(\xi)
                 +\lambda_R\widehat R_{g\pi_G(\xi)}\bigl(\theta(g)\bigr)\,,\\
  \Phi^R(\xi,g)&=\widehat R_g(\xi)
                 +\lambda_L\widehat L_{\pi_G(\xi)g}\bigl(\theta(g^{-1})\bigr) 
 \end{align*}
define two Hamiltonian actions $\Phi^L:G\times T^*G\to T^*G$
and $\Phi^R:T^*G\times G\to T^*G$ of $G$ on the symplectic manifold 
$(T^*G,\omega_{T^*G})$, respectively on the left and on the right. The maps
 $J^{L,\lambda_L}:T^*G\to{\mathcal G}^*$ and 
$J^{R,\lambda_R}:T^*G\to{\mathcal G}^*$ defined, for each $\xi\in T^*G$, by
 \begin{align*}
  J^{L,\lambda_L}(\xi)&=\widehat R_{\bigl(\pi_G(\xi)\bigr)^{-1}}(\xi)
                       +\lambda_L\theta\bigl(\pi_G(\xi))\,,\\
  J^{R,\lambda_R}(\xi)&=\widehat L_{\bigl(\pi_G(\xi)\bigr)^{-1}}(\xi)
                       +\lambda_R\theta\Bigl(\bigl(\pi_G(\xi)\bigr)^{-1}\Bigr)
 \end{align*}
\noindent
are momentum maps for the actions $\Phi^L$ and $\Phi^R$, respectively.
\par\smallskip

Moreover, the map $J^{L,\lambda_L}$ is constant on each orbit of the 
action $\Phi^R$, the map $J^{R,\lambda_R}$ is constant on each orbit 
of the action $\Phi^L$ and for each $\xi \in T^*G$ each of the tangent spaces at $\xi$ to the orbits $\Phi^L(G,\xi)$ and $\Phi^R(\xi,G)$ is the symplectic orthogonal
of the other (with respect to the symplectic form $\omega_{T^*G}$). 
The maps $J^{L,\lambda_L}:T^*G\to{\mathcal G}^*$ and 
$J^{R,\lambda_R}:T^*G\to{\mathcal G}^*$ are Poisson maps when $T^*G$ is equipped with
the Poisson structure associated to the symplectic form $\omega_{T^*G}$
and when ${\mathcal G}^*$ is equipped, respectively, with its canonical
Poisson structure modified by the cocycle $(\lambda_L+\lambda_R)\widetilde\Theta$
(\ref{SymplecticCocycle})
 $$\{f,g\}_{(\lambda_L+\lambda_R)\widetilde\Theta}(\zeta)
=\Bigl\langle\zeta,\bigl[\d f(\zeta),\d g(\zeta)\bigr]\Bigr\rangle
                        -(\lambda_L+\lambda_R)
                         \widetilde\Theta\bigl(\d f(\zeta),\d g(\zeta)\bigr)
 $$ 
and with the opposite of this Poisson structure.
\end{theo}

\begin{proof} The sum of the canonical symplectic form on $T^*G$ with the pull-back of
any closed two-form on $G$ always is nondegenerate, therefore symplectic. So
$\omega_{T^*G}$ is symplectic.
For $g$ and $h\in G$,
$\xi\in T^*G$, let us calculate
 $$\Phi^L\bigl(g,\Phi^L(h,\xi)\bigr)-\Phi^L(gh,\xi)\quad\hbox{and}\quad
   \Phi^R\bigl(\Phi^R(\xi,g),h\bigr)-\Phi^R(\xi,gh)\,.
 $$
We get
 \begin{align*}
  \Phi^L\bigl(g,\Phi^L(h,\xi)\bigr)-\Phi^L(gh,\xi)
   &=\lambda_R\widehat R_{gh\pi_G(\xi)}\Bigl(\Ad^*_{g^{-1}}\bigl(\theta(h)\bigr)+\theta(g)
                                   -\theta(gh)\Bigr)\\
   &=0
\end{align*}
since $\theta$ is a one-cocycle. The map $\Phi^L$ is therefore an action on the left 
of $G$ on $T^*G$. Similarly
 \begin{align*}
  \Phi^R\bigl(\Phi^R(\xi,g),h\bigr)-\Phi^R(\xi,gh)
   &=\lambda_L\widehat L_{\pi_G(\xi)gh}\Bigl(\Ad^*_{h}\theta(g^{-1})
                                              +\theta(h^{-1})\\
   &\quad\quad\quad\quad\quad\quad\quad\quad
     \quad\quad\quad\quad-\theta(h^{-1}g^{-1})\Bigr)\\
   &=0
 \end{align*}
for the same reason. The map $\Phi^R$ is therefore an action on the right of $G$
on $T^*G$. 
\par\smallskip

Let $X\in{\mathcal G}$ and $\xi=T^*G$. By calculating the derivative with respect to $t$
of $\Phi^L\bigl(\exp(tX),\xi\bigr)$ and of $\Phi^R\bigl(\xi,\exp(tX)\bigr)$, then setting
$t=0$, we get the following expressions for the fundamental vector fields on $T^*G$ associated to the actions $\Phi^L$ and $\Phi^R$:
 \begin{align*}
  X^{L,\lambda_R}_{T^*G}(\xi)&=X^L_{T^*G}(\xi)+ \lambda_R T\widehat R_{\pi_G(\xi)}
                              \Theta(X)\,,\\
  X^{R,\lambda_L}_{T^*G}(\xi)&=X^R_{T^*G}(\xi)-\lambda_L T\widehat L_{\pi_G(\xi)}
                              \Theta(X)\,, 
 \end{align*}
\noindent
the vector fields $X^L_{T^*G}$ and $X^R_{T^*G}$ being, as in the proof of
\ref{ActionsOfGonTstarG}, the canonical lifts to $T^*G$ of the fundamental
vector fields $X^L$ and $X^R$ on $G$, for the actions of $G$ on itself by translations
on the left and on the right, respectively. Using these expressions, we easily check
that
 $$\mathi(X^{L,\lambda_R}_{T^*G})\omega_{T^*G}=-\d J^{L,\lambda_L}_X\,,\quad
   \mathi(X^{R,\lambda_L}_{T^*G})\omega_{T^*G}=-\d J^{R,\lambda_R}_X\,,\quad
 $$ 
which means that the actions $\Phi^L$ and $\Phi^R$ are Hamiltonian and have,
respectively, $J^{L,\lambda_L}$ and $J^{R,\lambda_R}$ as momentum maps.
\par\smallskip
The facts that $J^{R,\lambda_R}$ is constant on each orbit of $\Phi^L$ and that
$J^{L,\lambda_L}$ is constant on each orbit of $\Phi^R$ directly follow from the
expressions of $\Phi^L$, $\Phi^R$, $J^{L,\lambda_L}$ and $J^{R,\lambda_R}$. 
\par\smallskip

Finally, let $X$ and $Y\in {\mathcal G}$. When considered as linear functions on
${\mathcal G}^*$, their Poisson bracket for the Poisson structure on ${\mathcal G}^*$
for which $J^{L,\lambda_L}$ is a Poisson map is easily determined by calculating
the Poisson bracket 
$\{J^{L,\lambda_L}\circ X, J^{L,\lambda_L}\circ Y\}
 =\{J^{L,\lambda_L}_X,J^{L,\lambda_L}_Y\}$, for 
the Poisson  structure on $T^*G$ associated to the symplectic form $\omega_{T^*G}$. This calculation fully determines the Poisson structure on ${\mathcal G}^*$
for which $J^{L,\lambda_L}$ is a Poisson map, and proves that it is indeed the canonical Poisson structure on $T^*G$ modified by the symplectic cocycle 
$(\lambda_L+\lambda_R)\widetilde\Theta$, in the sense of \ref{SymplecticCocycle}. A similar calculation shows that $J^{R,\lambda_R}$ is a Poisson map when 
${\mathcal G}^*$ is equipped with 
the opposite Poisson structure.
\end{proof}

\begin{prop} Under the assumptions and with the notations of 
\ref{GeneralizedActionsOfGonTstarG}, the momentum map 
$J^{L,\lambda_L}:T^*G\to{\mathcal G}^*$ is equivariant when $G$ acts on the left 
on $T^*G$ by the action $\Phi^L$ and on ${\mathcal G}^*$ by the action
 $$(g,\zeta)\mapsto \Ad^*_{g^{-1}}(\zeta)+\theta(g)\,,\quad
   (g,\zeta)\in G\times{\mathcal G}^*\,.
 $$
Similarly, the momentum map $J^{R,\lambda_R}:T^*G\to{\mathcal G}^*$ is equivariant
when $G$ acts on the right on $T^*G$ 
by the action $\Phi^R$ and on ${\mathcal G}^*$ by the action
 $$(\zeta,g)\mapsto \Ad^*_{g}(\zeta)+\theta(g^{-1})\,,\quad 
    (\zeta,g)\in {\mathcal G}^*\times G\,.
 $$
\end{prop}

\begin{proof}
Let $g\in G$ and $\xi\in T^*G$. Using the expressions of $J^{L,\lambda_L}$ and of 
$\Phi^L$, we obtain
 \begin{align*}
  J^{L,\lambda_L}\bigl(\Phi^L(g,\xi)\bigr)
    &=\Ad^*_{g^{-1}}\bigl(J^{L,\lambda_L}(\xi)\bigr)+(\lambda_L+\lambda_R)\theta(g)\,,\\
  J^{R,\lambda_R}\bigl(\Phi^R(\xi,g)\bigr)
    &=\Ad^*_{g}\bigl(J^{R,\lambda_R}(\xi)\bigr)+(\lambda_L+\lambda_R)\theta(g^{-1})\,,
 \end{align*}
which proves that $J^{L,\lambda_L}$ and $J^{R,\lambda_R}$ are equivariant with respect to the indicated actions, respectively on the left and on the right, of $G$ on $T^*G$ and on ${\mathcal G}^*$. 
\end{proof}

\section{Reduction of Hamiltonian systems with symmetries}
\label{Reduction}

Very early, many scientists 
(Lagrange, Jacobi, Poincar\'e,
$\ldots$) used first integrals to facilitate the determination of integral curves of Hamiltonian systems. It was observed that the knowledge of \emph{one} 
real-valued first integral often allows the reduction by \emph{two} units
of the dimension of the phase space in which solutions are searched for.
\par\smallskip

J.~Sniatycki and W.~Tulczyjew \cite{SniatyckiTulczyjew72} and, 
when first integrals come from the momentum map of a Lie group action,
K. Meyer \cite{Meyer73}, 
J. Marsden and A. Weinstein \cite{MarsdenWeinstein74}, developed a geometric presentation of this reduction procedure, widely known now under the name \lq\lq Marsden-Weinstein reduction\rq\rq.
\par\smallskip

Another way in which symmetries of a Hamiltonian system can be used to facilitate the determination of its integral curves was discovered
around 1750 by Leonard Euler (1707--1783) when he derived the 
equations of motion of a rigid body around a fixed point. 
In a short Note published in 1901 
\cite{Poincare1901}, Henri Poincar\'e formalized and generalized this reduction procedure, often called today,  
rather improperly, \lq\lq Lagrangian reduction\rq\rq\,
while the equations obtained by its application are called 
the \lq\lq Euler-Poincar\'e equations\rq\rq\ \cite{cendra,cendramarsden}.
\par\smallskip

We present in the following sections these two reduction procedures.

\subsection{The Marsden-Weinstein reduction procedure}

\begin{theo}\label{MarsdenWeinsteinReduction}
Let $(M,\omega)$ be a connected symplectic manifold 
on which a Lie group $G$ 
acts by a Hamiltonian action
$\Phi$, with a momentum map $J:M\to{\mathcal G}^*$.
Let $\xi\in J(M)\subset {\mathcal G}^*$ be a
possible value of $J$. The subset $G_\xi$ of elements $g\in G$ such 
that $\Phi_g\bigl(J^{-1}(\xi)\bigr)=J^{-1}(\xi)$ is a closed Lie subgroup
of $G$. 
\par\smallskip

If in addition $\xi$ is a weakly regular value of $J$ in the sense 
of Bott \cite{Bott1954}, $J^{-1}(\xi)$ is a submanifold of $M$ on which 
$G_\xi$ acts, by the action $\Phi$ restricted to $G_\xi$ and to
$J^{-1}(\xi)$, in such a way that all orbits are of the same dimension.
For each $x\in J^{-1}(\xi)$ the kernel of the two-form induced by 
$\omega$ on $J^{-1}(\xi)$ is the space tangent at this point to its
$G_\xi$-orbit. Let $M_\xi=J^{-1}(\xi)/G_\xi$ be the set of all these orbits. When $M_\xi$ has a smooth manifold structure for which
the canonical projection $\pi_\xi:J^{-1}(\xi)\to M_\xi$ is a submersion,
there exists on $M_\xi$ a unique symplectic form $\omega_\xi$ such that
$\pi_\xi^*\omega_\xi$ is the two-form induced on $J^{-1}(\xi)$
by $\omega$. The symplectic manifold $(M_\xi,\omega_\xi)$ is called the
\emph{reduced symplectic manifold} (in the sense of Marsden an Weinstein)
for the value $\xi$ of the momentum map.
\end{theo}

\begin{proof}
Proposition \ref{LieGroupCocycle} shows that there exists an affine action
$a$ of $G$ on ${\mathcal G}^*$ for which the momentum map $J$ is equivariant. The subset $G_\xi$ of $G$ is therefore the isotropy subgroup of $\xi$ for the action $a$, which proves that it is indeed
a closed subgroup of $G$. A well known theorem due to \'Elie Cartan allows us to state that $G_\xi$ is a Lie subgroup of $G$.
\par\smallskip

When $\xi$ is a weakly regular value of $J$, $J^{-1}(\xi)$ is a submanifold of $M$ and, for each 
$x\in J^{-1}(\xi)$, the tangent space at $x$ to this submanifold
is $\ker T_xJ$ (it is the definition of a weakly regular value in the sense of Bott). Let $N=J^{-1}(\xi)$ and let $i_N:N\to M$ be the canonical injection. For all $x\in N$, the vector spaces
$\ker T_xJ$  all are of the same dimension 
$\dim N$, and 
$\dim\bigl(T_xJ(T_xM)\bigr)=\dim M-\dim N$.   
Corollary \ref{SymplecticOrthogonal} shows that 
$T_xJ\bigl(T_xM\bigr)$ is the annihilator of ${\mathcal G}_x$.
Therefore for all $x\in N$ the isotropy subalgebras 
${\mathcal G}_x$ 
are of the same dimension $\dim G-\dim M+\dim N$.
The $G_\xi$-orbits of all points $x\in N$ 
are all of the same dimension $\dim G_\xi-\dim G_x$.
\par\smallskip

Corollary \ref{SymplecticOrthogonal} also shows that 
$\orth(\ker T_xJ)=\orth (T_xN)
=T_x\bigl(\Phi(G,x)\bigr)$. Therefore, for each 
$x\in N$,
 $$\ker (i_N^*\omega)(x)=T_xN
       \cap \orth(T_xN)
        =T_xN\cap T_x\bigl(\Phi(G,x)\bigr)
        =T_x\bigl(\Phi(G_\xi,x)\bigr)\,.  
 $$
It is indeed the space tangent at this point to its
$G_\xi$-orbit. When $M_\xi=N/G_\xi$ has a smooth manifold structure such that the canonical projection
$\pi_\xi:N\to M_\xi$ is a submersion, for each 
$x\in N$ the kernel of $T_x\pi_\xi$ is 
$\ker (i_N^*\omega)(x)$, and the existence on $M_\xi$ of a symplectic form 
$\omega_\xi$ such that $\pi_\xi^*(\omega_\xi)=i_N^*\omega$ easily follows.
\end{proof}

\begin{prop}\label{MarsdenWeinsteinHamiltonian}
The assumptions made here are the strongest of those made in
Theorem \ref{MarsdenWeinsteinReduction}: the set $J^{-1}(\xi)/G_\xi$ 
has a smooth manifold structure such that the canonical projection 
$\pi_\xi:J^{-1}(\xi)/G_\xi$ is a submersion. Let $H:M\to\RR$ be a smooth Hamiltonian, invariant under the action $\Phi$. There exists an unique smooth function $H_\xi:M_\xi\to\RR$ such that $H_\xi\circ\pi_\xi$ 
is equal to the restricton of $H$ to $J^{-1}(\xi)$. Each integral 
curve $t\mapsto \varphi(t)$ of the Hamiltonian vector field $X_H$ 
which meets $J^{-1}(\xi)$ is entirely contained in $J^{-1}(\xi)$,
and in the reduced symplectic manifold $(M_\xi,\omega_\xi)$ the
parametrized curve $t\mapsto \pi_\xi\circ \varphi(t)$ is an 
integral  curve of $X_{H_\xi}$.
\end{prop}

\begin{proof}
As in the proof of Theorem \ref{MarsdenWeinsteinReduction}, we set
$N=J^{-1}(\xi)$ and denote by $i_N:N\to M$ the canonical injection.
Let $\omega_N=i_N^*\omega$.
Since $H$ is invariant under the action $\Phi$, it keeps a constant value on each orbit of $G_\xi$ contained in $N$, so there exists
on $M_\xi$ an unique function $H_\xi$ such that 
$H_\xi\circ\pi_\xi=H\circ i_N$. The projection
$\pi_\xi$ being a surjective submersion, $H_\xi$ is smooth.
Noether's theorem (\ref{noether}) proves that the momentum map $J$ remains constant on each integral curve of the Hamiltonian vector field $X_H$. So if one of these integral curves meets $N$ 
it is entirely contained in $N$, and we see that the Hamiltonian vector field $X_H$ is tangent to $N$. We have,
for each $x\in N$,
 \begin{align*}
 \pi_\xi^*\biggl(\mathi\Bigl(T_x\pi_\xi\bigl(X_H(x)\bigr)\Bigr)
   \omega_\xi\bigl(\pi_\xi(x)\bigr)\biggr)
   &=i\bigl(X_H(x)\bigr)\bigl(i_N^*\omega(x)\bigl)=-\d(i_N^*H)(x)\\
   &=-\pi_\xi^*\bigl(\d H_\xi\bigr)(x)=\pi_\xi^*\bigl(\mathi(X_{H_\xi})\omega_\xi\bigr)(x)\,.
 \end{align*}

Since $\pi_\xi$ is a submersion and $\omega_\xi$ a non-degenerate 
two-form, this implies that for each $x\in N$, 
$T_x\pi_\xi\bigl(X_H(x)\bigr)=X_{H_\xi}\bigl(\pi_\xi(x)\bigr)$. The
restriction of $X_H$ to $N$ and $X_{H_\xi}$ are therefore two vector fields compatible with respect to the map $\pi_\xi:N\to M_\xi$, which implies the stated result.
\end{proof}

\begin{rmk}
Theorem \ref{MarsdenWeinsteinReduction} and Proposition
\ref{MarsdenWeinsteinHamiltonian} still hold when instead of the Lie group action $\Phi$ we have an action $\varphi$ of a finite-dimensional Lie algebra. The proof of the fact that the $G_\xi$-orbits in
$J^{-1}(\xi)$ all are of the same dimension can easily be adapted 
to prove that for all $x\in J^{-1}(\xi)$, the vector spaces
$\{\varphi(X)(x);X\in{\mathcal G}_\xi\}$ all are of the same dimension and determine a foliation of $J^{-1}(\xi)$. We have then only to replace the $G_\xi$-orbits by the leaves of this foliation.
\end{rmk}

\subsubsection{Use of the Marsden-Weinstein reduction procedure}
Theorem \ref{MarsdenWeinsteinReduction} and Proposition
\ref{MarsdenWeinsteinHamiltonian} are used to determine the integral
curves of the Hamiltonian vector field $X_H$ contained in $J^{-1}(\xi)$
in two steps:

\begin{itemize}

\item{} their projections on $M_\xi$ are first determined: they are integral curves of the Hamiltonian vector field $X_{H_\xi}$; this step is
often much easier than the full determination of the integral curves of $X_H$, since the dimension of the reduced symplectic manifold $M_\xi$ is smaller than the dimension of $M$;

\item{} then these curves themselves are determined; this second step, 
called \emph{reconstruction}, involves the resolution of a differential equation on the Lie group $G_\xi$.
\end{itemize}

Many scientists (T.~Ratiu, R.~Cushman, J.~Sniatycki, 
L.~Bates, J.-P.~Ortega, $\ldots$) generalized this reduction procedure
in several ways:
when $M$ is a Poisson manifold instead of a symplectic manifold, 
when $\xi$ is not a weakly regular value of $J$, $\ldots$ 
The reader will find more results on the subject in the recent book 
\emph{Momentum maps and Hamiltonian reduction} by J.-P. Ortega and 
T.S.~Ratiu \cite{OrtegaRatiu2004}.
\par\smallskip

Reduced symplectic manifolds occur in many applications 
other than the determination of integral curves of Hamiltonian systems.
The reader will find such applications in the book 
\emph{Symplectic techniques in Physics} by V. Guillemin and S. Sternberg
\cite{GuilleminSternberg84b} and in the papers on the phase space of a particle in a Yang-Mills field \cite{Sternberg77, Weinstein78}.

\subsection{The Euler-Poincar\'e equation}

In his Note \cite{Poincare1901}, Henri Poincar\'e writes the 
equations of motion of a Lagrangian mechanical system when 
a finite-dimensional 
Lie algebra acts on its configuration space
by a locally transitive action. Below we 
adapt his results to the Hamiltonian formalism.

\begin{prop}\label{EulerPoincare} 
Let $\mathcal G$ be a finite-dimensional Lie algebra 
which acts, by an action $\varphi:{\mathcal G}\to A^1(N)$, on a smooth manifold $N$. 
The action $\varphi$ is assumed to be locally transitive, which means 
that for each $x\in N$, 
$\bigl\{\varphi(X)(x)\,;X\in{\mathcal G}\bigr\}=T_xN$. Let 
$\widehat\varphi:{\mathcal G}\to A^1(T^*N)$ be the Hamiltonian action of 
$\mathcal G$ on $(T^*N,\d\eta_N)$ which associates, to each $X\in{\mathcal G}$,
the canonical lift to $T^*N$ of the vector field $\varphi(X)$ on $N$
(\ref{LiftedHamiltonianAction}), and let
$J:T^*N\to{\mathcal G}^*$ be the the momentum map of $\widehat\varphi$  given by the
formula (\ref{LiftedHamiltonianAction2})
 $$\bigl\langle J(\xi),X\bigr\rangle=\mathi\bigl(\widehat\varphi(X)\bigr)\eta_N(\xi)\,,
   \quad X\in{\mathcal G}\,,\ \xi\in T^*N\,.
 $$ 
Let $H:T^*N\to\RR$ be a smooth Hamiltonian, which comes from a hyper-regular Lagrangian 
$L:TN\to\RR$ (hyper-regular means that the associated
Legendre map ${\mathcal L}:TN\to T^*N$is a diffeomorphism). Let
$\psi:I\to T^*N$ be an integral curve of the Hamiltonian vector field $X_H$ defined on an open interval $I$ and 
$V:I\to{\mathcal G}$  be a smooth parametrized curve in $\mathcal G$ which satisfies, for each $t\in I$,
 $$\varphi\bigl(V(t)\bigr)\bigl(\pi_N\circ\psi(t)\bigr)=\frac{d\bigl(\pi_N\circ
    \psi(t)\bigr)}{\d t}\,. \eqno(1)$$
The curve $J\circ\psi:I\to{\mathcal G}^*$, obtained by composition with $J$ of the integral curve $\psi$ of the Hamiltonian vector field $X_H$, satisfies the differential equation in ${\mathcal G}^*$
 $$\left(\frac{\d}{\d t}-\ad^*_{V(t)} \right)\bigl(J\circ \psi(t)\bigr)
    =J\Bigl(d_1\overline{\mathstrut L}\bigl(\pi_N\circ\psi(t),V(t)\bigr)\Bigr)\,.
    \eqno{(2)}
 $$
We have denoted by $\overline{\mathstrut L}:N\times{\mathcal G}\to \RR$ the map
 $$(x,X)\mapsto \overline{\mathstrut L}(x,X)=L\bigl(\varphi(X)(x)\bigr)\,,\quad x\in N\,,\ X\in{\mathcal G}\,,$$
and by $d_1\overline{\mathstrut L}:N\times{\mathcal G}\to T^*N$ the partial differential of 
$\overline{\mathstrut L}$ with respect to its first variable.
\par\smallskip
Equation $(2)$ is called the \emph{Euler-Poincar\'e equation}, while
Equation $(1)$ is called the \emph{compatibility condition}. 
\end{prop}

\begin{proof}
For each $\xi\in T^*N$ and each $X\in{\mathcal G}$
 $$\bigl\langle J(\xi),X\bigr\rangle=\varphi(X)\bigr(\pi_N(\xi)\bigr)\,,$$
therefore
 $$\Bigl\langle \frac{\d}{\d t}\bigl(J\circ\psi(t)\bigr),X\Bigr\rangle
   =\frac{\d}{\d t}\Bigl\langle\psi(t),\varphi(X)\bigl(\pi_N\circ\psi(t)\bigr)
                    \Bigr\rangle\,.
 $$
Let $(x^1,\ldots,x^n)$ be local coordinates on $N$, and $(x^1,\ldots,x^n,
\allowbreak p_1,\ldots,p_n)$ be the associated local coordinates on $T^*N$.
The smooth curves $\psi$ and 
$\pi_N\circ\psi$ can be expressed as
 $$t\mapsto \bigl( y^i(t),\varpi_i(t)\bigr)\quad\hbox{and}\quad 
            t\mapsto\bigl( y^i(t)\bigr)\,,\quad(1\leq i\leq n)\,,
 $$  
so we can write
 $$\Bigl\langle \frac{\d}{\d t}\bigl(J\circ\psi(t)\bigr),X\Bigr\rangle
   =\frac{\d}{\d t}\left(\sum_{i=1}^n\varpi_i(t)\bigl(\varphi(X)\bigr)^i
     \bigl(y^1(t),\ldots,y^n(t)\bigr)\right)\,.
 $$
We have denoted by $\bigl(\varphi(X)\bigr)^i(x^1,\ldots,x^n)$ the value of the $i$-th
component of the vector field $\varphi(X)$, expressed as a function of the local coordinates $x^i$ ($1\leq i\leq n)$.
\par\smallskip
The compatibility condition $(1)$ becomes 
 $$\frac{\d y^k(t)}{\d t}=\Bigl(\varphi\bigl(V(t)\bigr)\Bigr)^k\bigl(y^1(t),
    \ldots,y^n(t)\bigr)\,. 
 $$
In what follows we write $y^i$ for $y^i(t)$, $\varpi_i$ for $\varpi_i(t)$,
$(y)$ for $\big(y^1(t),\ldots,y^n(t)\bigr)$ and $(y,\varpi)$
for $\bigl(y^1(t),\ldots,y^n(t),\varpi_1(t),\varpi_n(t)\bigr)$. We have
  \begin{align*}
  \Bigl\langle \frac{\d}{\d t}\bigl(J\circ\psi(t)\bigr),X\Bigr\rangle
  &=\sum_{i=1}^n\frac{\d\varpi_i}{\d t}\bigl(\varphi(X)\bigr)^i(y)\\
  &\quad\quad +\sum_{(i,k)=(1,1)}^{(n,n)}
     \varpi_i\frac{\partial \bigl(\varphi(X)\bigr)^i(y)}{\partial x^k}
     \Bigl(\varphi\bigl(V(t)\bigr)\Bigr)^k(y)\,.
 \end{align*}
By using the local expression of the bracket of vector fields
 \begin{align*}
  \Bigl[\varphi\bigl(V(t)\bigr),\varphi(X)\Bigr]^i(x)
  &=\sum_{k=1}^n\Bigl(\varphi\bigl(V(t)\bigr)\Bigr)^k(x)
   \frac{\partial \bigl(\varphi(X)\bigr)^i(x)}{\partial x^k}\\
  &\quad -\sum_{k=1}^n\bigl(\varphi(X)\bigr)^k(x)
   \frac{\partial \Bigl(\varphi\bigl(V(t)\bigr)\Bigr)^i(x)}{\partial x^k}
 \end{align*}
and taking into account the fact that, $\varphi$ being a Lie algebras homomorphism,
 $$\Bigl[\varphi\bigl(V(t)\bigr),\varphi(X)\Bigr]
  =\varphi\Bigl(\bigl[V(t),X]\Bigr)\,,$$
we get
 \begin{align*}
  \Bigl\langle \frac{\d}{\d t}\bigl(J\circ\psi(t)\bigr),X\Bigr\rangle
  &=\Bigl\langle\psi(t),\varphi\Bigl(\bigl[V(t),X\bigr]\Bigr)\Bigr\rangle\\
  &\quad+\sum_{i=1}^n \bigl(\varphi(X)\bigr)^i(y)\Biggl(\frac{\d\varpi_i}{\d t}
  +\sum_{k=1}^n\varpi_k\frac{\partial\Bigl(\varphi\bigl(V(t)\bigr)\Bigr)^k(y)}
    {\partial x^i}\Biggr)\,.
 \end{align*}
The first term in the right hand side can be written
 $$\Bigl\langle\psi(t),\varphi\Bigl(\bigl[V(t),X\bigr]\Bigr)\Bigr\rangle
 =\Bigl\langle J\circ\psi(t), \bigl[V(t),X\bigr]\Bigr\rangle
 =\Bigl\langle \ad^*_{V(t)}\bigl(J\circ\psi(t)\bigr),X\Bigr\rangle\,.
 $$
For all $(x,X)\in N\times{\mathcal G}$ we have  
 $$\overline{\mathstrut L}(x, X)=L\bigl(\varphi(X)(x)\bigr)\,.$$
For any point $x\in N$ and any vector $w\in T_xN$, there exists a smooth
curve $s\mapsto z(s)$ in $N$ such that $z(0)=x$ 
and $\displaystyle \frac{\d z(s)}{\d s}\bigm|_{s=0}=w$.   
We easily obtain 
$\Bigl\langle \d_1\overline{\mathstrut L}\bigl(x,V(t)\bigr),w\Bigr\rangle$ 
by taking the derivative of 
$\overline{\mathstrut L}\bigl(x(s),V(t)\bigr)$ with respect to $s$ ($t$ remaining fixed), 
then making $s=0$. We obtain
 $$\Bigl\langle \d_1\overline{\mathstrut L}\bigl(x,V(t)\bigr),w\Bigr\rangle
   =\sum_{i=1}^n w^i\Biggl(\frac{\partial L(x,v)}{\partial x^i}
                         +\sum_{k=1}^n\frac{\partial L(x,v)}{\partial v^k}\,
                          \frac{\partial\Bigl(\varphi\bigl(V(t)\bigr)\Bigr)^k(x)}
                          {\partial x^i}      
                    \Biggr)\,.
 $$
Let us set $x=\pi_N\circ\psi(t)$, $w=\varphi(X)\bigl(\pi_N\circ\psi(t)\bigr)$. 
We observe that
 $$\Bigl\langle \d_1\overline{\mathstrut L}\bigl(\pi_N\circ\psi(t),V(t)\bigr),
               \varphi(X)\bigl(\pi_N\circ\psi(t)\bigr)\Bigr\rangle
   =\Bigl\langle J\Bigl(d_1\overline{\mathstrut L}\bigl(\pi_N\circ\psi(t),V(t)\bigr)\Bigr),
     X
    \Bigr\rangle\,.
 $$
Now we take into account the well known relations which exist between the partial derivatives of the Lagrangian and of the Hamiltonian expressed in local coordinates
 $$\frac{\partial L(x,v)}{\partial x^i}=-\frac{\partial H(x,p)}{\partial x^i}
   =\frac{d\varpi_i}{dt}\,,\quad  \frac{\partial L(x,v)}{\partial v^k}=\varpi_k\,,
 $$ 
and we obtain
 \begin{align*}
  \Bigl\langle J\Bigl(d_1\overline{\mathstrut L}\bigl(\pi_N\circ\psi(t),V(t)\bigr)\Bigr),
     X\Bigr\rangle
   =\sum_{i=1}^n \bigl(\varphi(X)\bigr)^i&(y)\\
           \quad\Biggl(\frac{\d\varpi_i}{\d t}
            +\sum_{k=1}^n&\varpi_k
              \frac{\partial\Bigl(\varphi\bigl(V(t)\bigr)\Bigr)^k(y)}
              {\partial x^i}\Biggr)\,.
 \end{align*}
Since $X$ can be any element in $\mathcal G$, the Euler-Poincar\'e equation follows.
\end{proof}

\begin{rmk}\label{Euler-Poincare2}
The assumptions made by Poincar\'e in \cite{Poincare1901} are less restrictive than those
made in \ref{EulerPoincare}: he uses the Lagrangian formalism for a smooth Lagrangian
$L:TN\to\RR$ which is not assumed to be hyper-regular. The associated Legendre
map ${\mathcal L}:TN\to T^*N$ still exists as a smooth map (it is the vertical differential of $L$, see for example \cite{Godbillon1969}), but may not be a diffeomorphism. Of course the momentum map $J:T^*N\to{\mathcal G}^*$ still exists
and can be used, together with the Legendre map, to express Poincar\'e's 
results intrinsically \cite{Marle2013}, independently of any choice of 
local coordinates. Poincar\'e proves that if a smooth parametrized 
curve $\gamma:[t_0,t_1]\to N$ is an extremal of the action functional
 $$I(\gamma)=\int_{t_0}^{t_1}L\left(\frac{\d\gamma(t)}{\d t}\right)\,\d t$$
for infinitesimal variations of $\gamma$ with fixed end points, and if
$V:[t_0,t_1]\to{\mathcal G}$ is a smooth parametrized curve which satisfies, 
for each $t\in[t_0,t_1]$, the compatibility condition
 $$\varphi\bigl(V(t)\bigr)\bigl(\gamma(t)\bigr)=\frac{d\bigl(\gamma(t)\bigr)}{\d t}\,,
 \eqno(1)
 $$
the parametrized curve $\displaystyle t\mapsto J\circ{\mathcal L}\circ
\frac{\d\gamma(t)}{\d t}$ satisfies the Euler-Poincar\'e equation 
 $$\left(\frac{\d}{\d t}-\ad^*_{V(t)} \right)\left(J\circ{\mathcal L}\circ
    \frac{\d\gamma(t)}{\d t}\right)
    =J\Bigl(d_1\overline{\mathstrut L}\bigl(\gamma(t),V(t)\bigr)\Bigr)\,.\eqno(2)
 $$
The Euler-Poincar\'e equation can be written under a slightly different form
in which, instead of the Legendre map ${\mathcal L}:TN\to T^*N$,
 the partial differential 
$\d_2\overline{\mathstrut L}:N\times{\mathcal G}\to{\mathcal G}^*$
of the map $\overline{\mathstrut L}:N\times{\mathcal G}\to \RR$ with respect to its 
second variable is used. We have indeed, for all $x\in N$ and $X\in{\mathcal G}$,
 $$\d_2\overline{\mathstrut L}(x,X)=J\circ{\mathcal L}\bigl(\varphi(X)(x)\bigr)\,,
 $$
which allows to write the Euler-Poincar\'e equation under the form
 $$\left(\frac{\d}{\d t}-\ad^*_{V(t)} \right)
    \Bigl(\d_2\overline{\mathstrut L}\bigl(\gamma(t),V(t)\bigr)\Bigr)
    =J\Bigl(d_1\overline{\mathstrut L}\bigl(\gamma(t),V(t)\bigr)\Bigr)\,.\eqno(3)
 $$  
\end{rmk}

\subsubsection{Use of the Euler-Poincar\'e equation for reduction}
Poincar\'e observes in his Note \cite{Poincare1901} that the 
Euler-Poincar\'e equation can be useful 
mainly when its right hand side vanishes and when it 
reduces to an autonomous differential equation on 
${\mathcal G}^*$ for the parametrized curve $t\mapsto J\circ \psi(t)$.
We will see in Section \ref{RigidBody} that 
the first condition is satisfied when the Hamiltonian system under 
consideration describes the motion of a rigid body around a fixed 
point in the absence of external forces (Euler-Poinsot problem). The second
condition generally is not satisfied, since
the Euler-Poincar\'e equation involves the parametrized curve 
$t\mapsto V(t)$ in $\mathcal G$, whose dependence on $J\circ\psi(t)$ is
complicated.
\par\smallskip

This simplification occurs when there exists a smooth function 
$h:{\mathcal G}^*\to\RR$ such that
 $$H=h\circ J\,,$$
which implies that $H$ is constant on each level set of $J$. Then it 
can be shown that the Euler-Poincar\'e equation becomes the Hamilton equation on ${\mathcal G}^*$ for the Hamiltonian $h$ and its canonical 
Poisson structure.
\par\smallskip

If we assume that the manifold $N$ is a Lie group $G$ and that the action
$\varphi:{\mathcal G}\to A^1(G)$ of its Lie algebra is the action 
associated to the action of $G$ on itself by translations on the left 
(respectively, on the right), $\widehat\varphi$ is the Lie algebra action
associated to the canonical lift to $T^*G$ of the canonical action of $G$
on itself by translations on the left (respectively, on the right). 
The conditions under which the Euler-Poincar\'e 
equation can be used for reduction are exactly the same as those under 
which the Marsden-Weinstein reduction method can be applied, but for 
the canonical lift to $T^*G$ of the action of $G$ on itself by translations 
\emph{on the right} (respectively, on the left). Moreover, applications of 
these two reduction methods lead to essentially the same equations: 
the only difference is that the Euler-Poincar\'e reduction method leads to a differential equation on ${\mathcal G}^*$, while the Marsden-Weinstein reduction 
method leads, for each value of the momentum map,
to the same differential equation restricted to a coadjoint orbit of ${\mathcal G}^*$. The reader will find the proof of these assertions
in \cite{Marle2003, Marle2013}. 

\section{Examples of Hamiltonian dynamical systems}\label{ExamplesDynamicalSystems}

We present in this section three classical examples of Hamiltonian dynamical 
systems in which the previously discussed concepts (symmetry groups, momentum
maps and first integrals, reduction methods) are illustrated. The configuration
space of the first system (the spherical pendulum) is a sphere embedded in 
physical space; each point of the sphere is a possible position of a material 
point which moves on that sphere. The third example (the Kepler problem) deals with the motion of a material point in the acceleration field created by an attracting centre; the configuration space is the physical space minus one point (the attractive centre). In the second example (the motion of a rigid body around a fixed point) the configuration space is a little more complicated: it is the set of all maps which send the material body onto one of its possible positions in space. 

\subsection{The mathematical description of space and time}
The framework in which the motions of material bodies occur
is the physical space-time. It will be mathematically described here
as it is usually done in classical (non-relativistic) Mechanics.
In a Galilean reference frame, once units of length and of time are chosen,
the physical space and the physical time are mathematically described 
by affine Euclidean
spaces $E$ and $T$, respectively three-dimensional and one-dimensional. 
We will consider $E$ and $T$ as \emph{oriented}: $T$ has a natural 
orientation (towards the future), 
while by convention, an arbitrary orientation of $E$ is chosen. 
The choice of a particular element of $T$ as origin will allow 
us to identify $T$ with the real line $\RR$. 
\par\smallskip

In the three examples treated below there exists a privileged element of $E$ (the centre of the sphere, the fixed point and the attractive centre, respectively in the first, second and third examples) which will be taken as origin. The space $E$ will therefore be considered as an Euclidean three-dimensional
\emph{vector} space. For the same reason the abstract space $S$ of material points
used in the second example will be considered too as an Euclidean three-dimensional
\emph{vector} space.  
\par\smallskip

In our three examples, the configuration space of the system will be denoted 
by $N$: therefore in the first example $N$ is the sphere embedded in $E$ 
centered on the origin on which the material point is moving; in the third 
example $N=E\backslash\{O\}$, where $O$ is the attractive centre; and we will 
see that in the second example, $N=\Isom(S,E)$ is the space of orientation
preserving linear isometries of an abstract three-dimensional Euclidean 
vector space $S$ (the space of material points) onto the physical space $E$.

\subsection{Vector calculus in a three-dimensional oriented Euclidean vector space}
\label{VectorCalculus}
The group $\SO(E)$ of orientation preserving linear isometries
of $E$, isomorphic to $\SO(3)$, acts on the space $E$, and so does its Lie algebra
$\mathfrak{so}(E)$, isomorphic to $\mathfrak{so}(3)$, by the associated action. The Euclidean vector space $E$ being three-dimensional and oriented, there exists
an isomorphism of $\mathfrak{so}(E)$ onto the space $E$ itself widely used
in elementary vector calculus, in which an element $X\in \mathfrak{so}(E)$, which is a linear map $E\to E$ represented, in some orthonormal  positively oriented basis
$(\vect e_1,\vect e_2,\vect e_3)$
of $E$, by the skew-symmetric $3\times 3$-matrix
 $$\begin{pmatrix}
          0&-c&b\\
           c&0&-a\\
           -b&a&0
   \end{pmatrix}\,
 $$    
is identified with the vector $\vect X=a\vect e_1+b\vect e_2+c\vect e_3$. With this identification, the bracket in $\mathfrak{so}(E)$, in other words the map
$(X,Y)\mapsto [X,Y]=X\circ Y-Y\circ X$, corresponds to the \emph{vector product}
$(\vect X,\vect Y)\mapsto\vect X\times\vect Y$. Expressed in terms of the vector
product, the Jacobi identity becomes
 $$\vect X\times(\vect Y\times \vect Z)
  +\vect Y\times(\vect Z\times \vect X)
  +\vect Z\times(\vect X\times \vect Y)=0\,.\eqno(1)
 $$

Let us recall another very useful formula which expresses, in terms of the scalar and vector products, the $\ad$-invariance of the pairing between $\mathfrak{so}(E)$ and its dual by means of the scalar product. For any triple 
$(\vect u,\vect v,\vect w)\in E\times E\times E$, we have 
 $$\vect u.(\vect v\times\vect w)=\vect v.(\vect w\times\vect u)
                                 =\vect w.(\vect u\times\vect v)\,.\eqno(2)
 $$  
The map $(\vect u,\vect v,\vect w)\mapsto \vect u.(\vect v\times\vect w)$
is therefore a skew-symmetric trilinear form on $E$ sometimes called the
\emph{mixed product}.
\par\smallskip

The dual $E^*$ of $E$ will be identified with $E$, with the scalar product
$(\vect u, \vect v)\mapsto\vect u.\vect v$ as pairing by duality. The tangent and cotangent bundles $TE$ and $T^*E$ will therefore both be identified with 
$E\times E$, the canonical projections $\tau_E:TE\to E$ and $\pi_E:T^*E\to E$ 
both being the projection of $E\times E$ onto its first factor. 
The Lie algebra action
of $\mathfrak{so}(E)$ on $E$ associates, to each $\vect X\in\mathfrak{so}(E)\equiv E$, 
the vector field $\vect{X_E}$ on $E$ whose value at an element $\vect x\in E$ is 
 $$\vect{X_E}(\vect x)=(\vect x, \vect X\times \vect x)\,.\eqno(3)$$ 
\par\smallskip

Since we have identified $\mathfrak{so}(E)$ with $E$, its dual space $\mathfrak{so}(E)^*$ is identified with $E^*$, which we have identified with $E$ by means of the scalar product.
Therefore $\mathfrak{so}(E)^*$ too will be identified with $E$.
\par\smallskip

The canonical lift to the cotangent bundle of the action of $\SO(E)$ on $E$ is
a Hamiltonian action (\ref{LiftedHamiltonianAction}) whose momentum map
$J_E:T^*E\equiv E\times E\to \mathfrak{so}(E)^*\equiv E$
can easily be expressed in terms of the vector product. Indeed the map $J_E$
must satisfy, for each $\vect X\in\mathfrak{so}(E)\equiv E$ and each $(\vect x,\vect p)\in T^*E\equiv E\times E$,
 $$\bigl\langle J_E(\vect x,\vect p),X\bigr\rangle
   =\bigl\langle(\vect x,\vect p),\vect{X_E}(\vect x)\bigr\rangle
   =\vect p.(\vect X\times\vect x)
   =\vect X.(\vect x\times\vect p)\,,$$
the last equality being obtained by using  the above formula $(2)$. We therefore see that
 $$J_E(\vect x,\vect p)=\vect x\times\vect p\,.\eqno(4)$$
Expressed in terms of the vector product, the adjoint and coadjoint actions become
 $$\ad_{\vect X}\vect Y=\vect X\times \vect Y\,,\quad 
    \ad^*_{\vect X}\vect\xi=-\vect X\times\vect \xi=\vect \xi\times\vect X\,,
     \eqno(5)
 $$
where $\vect X$ and $\vect Y\in\mathfrak{so}(E)\equiv E$ and $\vect\xi\in\mathfrak{so}(E)^*\equiv E$.
\par\smallskip

Of course all the above properties hold for the three-dimensional Euclidean oriented
vector space $S$ of material points which is used in the second example, for the
group $\SO(S)$ of its linear orientation preserving isometries and for its
Lie algebra $\mathfrak{so}(S)$.

\subsection{The spherical pendulum}\label{SphericalPendulum}

\subsubsection{Mathematical description of the problem}
Let us consider a heavy material point of mass $m$ constrained, by an ideal 
constraint, on the surface of a sphere $N$ of centre $O$ and radius $R$ 
embedded in the physical space $E$. Since the action of $\SO(E)$ on $E$ maps $N$
onto itself, $\SO(E)$ acts on $N$ on the left, and so does its Lie algebra
$\mathfrak{so}(E)$ by the associated action, which is locally (and globally) transitive. 
The configuration space $N$ is the set of vectors $\vect x\in E$ which satisfy 
$\vect x.\vect x=R^2$
and its tangent bundle $TN$ is the subset of $TE\equiv E\times E$ of pairs
$(\vect x,\vect v)$ of vectors which satisfy
 $$\vect x.\vect x=R^2\,,\quad \vect x.\vect v=0\,.$$
We assume that the material point is submitted to a constant acceleration field
$\vect g$ (which, in most applications, will be the vertical gravity field 
directed downwards). The Lagrangian of the system is
 $$L(\vect x,\vect v)=\frac{m\Vert \vect v\Vert^2}{2}+m\vect g.\vect x\,.
 $$
The Legendre map ${\mathcal L}:TN\to T^*N$ is expressed as
 $${\mathcal L}(\vect x,\vect v)=(\vect x,\vect p)\quad\hbox{with}\ \vect p=m\vect v\,.
 $$
The Hamiltonian of the system is therefore
 $$H(\vect x,\vect p)=\frac{\Vert \vect p\Vert^2}{2m}-m\vect g.\vect x\,.$$
The momentum map $J_E$ of the canonical lift to the cotangent bundle of the Lie algebra action $\varphi$, expressed in terms of the vector product, is given by
Formula $(4)$ in Section \ref{VectorCalculus}.

\subsubsection{The Euler-Poincar\'e equation}
The map $\widetilde\varphi:N\times\mathfrak{so}(E)\to TN$ defined by
$\widetilde\varphi(\vect x,\vect X)=\vect{X_N}(\vect x)$, expressed,
in terms of the vector product, is 
$\widetilde\varphi(\vect x, \vect X)=(\vect x,\vect X\times\vect x)$. Using
Formula $(2)$ of \ref{VectorCalculus}, we easily obtain the expression of
$\overline{\mathstrut L}=L\circ\widetilde\varphi:N\times\mathfrak{so}(E)\to \RR$:   
 $$\overline{\mathstrut L}(\vect x,\vect X)=\frac{mR^2}{2}\left(\Vert\vect X\Vert^2
    -\frac{(\vect X.\vect x)^2}{R^2}\right)+m\vect g.\vect x\,.$$
The partial differentals of $\overline{\mathstrut L}$ with respect to its first and 
second variables are
 \begin{align*}
  \d_1\overline{\mathstrut L}(\vect x,\vect X)&=\Biggl(\vect x,m\Bigl(\vect g
                                   -(\vect X.\vect x)\vect X +
                                   \frac{(\vect x.\vect X)^2-\vect g.\vect x}{R^2}\,
                                    \vect x\Bigr)
                                      \Biggr)\,,\\
 \d_2\overline{\mathstrut L}(\vect x,\vect X)&=m\bigl(R^2\vect X-
                                      (\vect X.\vect x)\vect x\bigr)\,.
 \end{align*}
Let $t\mapsto \vect{x(t)}$ be a smooth curve in $N$, parametrized by the time $t$, solution of the Euler-Lagrange equation for the Lagrangian $L$. The compatibility
condition $(1)$ of \ref{Euler-Poincare2} becomes, for a smooth map 
$t\mapsto V(t)$ in $\mathfrak{so}(E)$,
 $$\frac{\d\vect{x(t)}}{\d t}=\vect V(t)\times \vect{x(t)}\,,$$ 
and the Euler-Poincar\'e equation $(3)$ of \ref{Euler-Poincare2} is
 $$\frac{\d}{\d t}\Bigl(mR^2\vect V(t)-m\bigl(\vect{x(t)}.\vect V(t)\bigr)
    \vect{x(t)}\Bigr)=m\vect{x(t)}\times \vect g\,.
 $$
This equation can easily be obtained by much more elementary methods: it expresses the fact that the time derivative of the angular momentum at the origin is equal to the moment at that point of the gravity force (since the moment at the origin of the constraint force which binds the material point to the surface of the sphere vanishes).
\par\smallskip

The Euler-Poincar\'e equation allows a reduction of the problem if and only if its right hand side vanishes, which occurs if and only if $\vect g=0$. When that condition is satisfied, it can be written as
 $$m\frac{\d}{\d t}\left(\vect{x(t)}\times\frac{\d\vect{x(t)}}{\d t}\right)=0\,,
 $$ 
which implies that the material point moves 
on a great circle of the sphere $N$, in the plane through its centre orthogonal to the constant vector
$\displaystyle\vect{x(t)}\times\frac{\d\vect{x(t)}}{\d t}$. Using the conservation of energy $H$, we see that $\Vert \vect v(t)\Vert$ remains constant during the motion.

\subsubsection{Reduction by the use of first integrals}
\label{SphericalPendulumFirstIntegrals}
Equation $(4)$ of Section \ref{VectorCalculus} shows that the map
 $$J(\vect x,\vect p)=\vect x\times\vect p\,,
  \quad\text{with}\ (\vect x,\vect p)\in T^*N\equiv N\times E\,,
 $$
is a momentum map for the canonical lift to $T^*N$ of the action of $\SO(E)$
on $N$. When $\vect g\neq 0$ that action does not leave invariant
the Hamiltonian $H$, but its restriction to the subgroup $G_1$ of rotations around the vertical line through the centre of the sphere $N$ does leave $H$ invariant. The Lie algebra of $G_1$ and its dual being identified with $\RR$,
the momentum map of this restricted action is
 $$J_1(\vect x,\vect p)=\vect e_g.(\vect x\times\vect p)$$
where $\vect e_g$ is the unit vector such that 
$\vect g=g\vect e_g$, with $g>0$. The only singular value of $J_1$ is $0$.
It is reached when the three vectors $\vect x$, $\vect p$ and $\vect e_g$ 
lie in the same vertical plane. Therefore,
for any $\zeta\neq 0$, $J_1^{-1}(\zeta)$ is a three-dimensional submanifold
of $T^*N$ which does not contain $T^*_{R\vect e_g}N\cup T^*_{-R\vect e_g}N$
and which remains invariant under the action of $G_1$. The set of orbits of 
this action is the Marsden-Weinstein reduced symplectic manifold for the value 
$\zeta$ of the momentum map. On this two-dimensional reduced symplectic manifold 
all the integral curves of the Hamiltonian vector field associated to the reduced Hamiltonian $H_\zeta$ are periodic.

\subsection{The motion of a rigid body around a fixed point}\label{RigidBody}

\subsubsection{Mathematical description of the problem}\label{MathRigidBody}
We consider the motion of a rigid body containing at least three non-aligned material points. A configuration of the body in space is 
mathematically represented by an affine, 
isometric and orientation preserving map defined on an abstract 
Euclidean three-dimensional oriented
affine space $S$ (called the \emph{space of material points}), 
with values in $E$, the three-dimensional Euclidean oriented 
affine space which mathematically describes the physical space.
When the configuration of the body 
is represented by the map $x:S\to E$, the position
in space of the material point of the body represented
by $z\in S$ is $x(z)$.
\par\smallskip

We assume that one geometric point of the rigid body is constrained, by an ideal constraint, to keep a fixed position in the physical space. By using this fixed point as origin, both for $S$ and for $E$, we can now consider these spaces
as \emph{vector} spaces. Each  configuration of the 
body in space is therefore represented by a \emph{linear}
isometry. The set $N$ of all possible configurations of 
the material body in space is therefore $\Isom(S,E)$, the set of linear 
orientation-preserving isometries of $S$ onto $E$.  
\par\smallskip 
 
The Lie groups $\SO(S)$ and $\SO(E)$ of linear orientation-preserving isometries, respectively of $S$ and of $E$, both isomorphic to 
${\SO}(3)$, act on $N$, respectively on the left and on the right, by the two commuting actions
$\Phi_S$ and $\Phi_E$
 $$\Phi_S(x, g_S)=x\circ g_S\,,\ \Phi_E(g_E,x)=g_E\circ x\,,
    \quad g_E\in\SO(E)\,,\ g_S\in\SO(S)\,,\ x\in N\,.  
 $$ 
\par\smallskip

The values at $x\in N$ of the fundamental vector fields on $N$ associated to 
$X\in\mathfrak{so}(S)$ and $Y\in\mathfrak{so}(E)$ are 
 $$X_N(x)=\frac{\d \bigl(x\circ\exp(sX)\bigr)}{ds}\Bigm|_{s=0}\,,
   \quad Y_N(x)=\frac{\d (\exp(sY)\circ x)}{ds}\Bigm|_{s=0}\,.$$  
The Lie algebra actions $\varphi_S:\mathfrak{so}(S)\to A^1(N)$ and 
$\varphi_E:\mathfrak{so}(E)\to A^1(N)$ associated to the Lie group actions
$\Phi_S$ and $\Phi_E$ are, respectively, the maps
 $$\varphi_S(X)=X_N\,,\quad \varphi_E(Y)=Y_N\,,
    \quad X\in\mathfrak{so}(S)\,,
    \quad Y\in\mathfrak{so}(E)\,.
 $$
One should be careful with signs: since $\Phi_S$ is an action of $\SO(S)$ on the right, the bracket of elements in the Lie algebra $\mathfrak{so}(S)$ for which
$\varphi_S$ is a Lie algebras homomorphism is the bracket of 
\emph{left-invariant} vector fields on the Lie group $\SO(S)$; similarly, 
since $\Phi_E$ is an action of $\SO(E)$ on the left, the bracket of elements 
in the Lie algebra $\mathfrak{so}(E)$ for which $\varphi_E$ is a Lie algebras
homomorphism is the bracket of \emph{right-invariant} vector fields on the Lie group $\SO(E)$. 
\par\smallskip

Let $\widetilde\varphi_S:N\times\mathfrak{so}(S)\to TN$ and 
$\widetilde\varphi_E:N\times\mathfrak{so}(E)\to TN$ be the vector bundles 
isomorphisms 
 $$\widetilde\varphi_S(x,X)=\varphi_S(X)(x)\,,\quad 
    \widetilde\varphi_E(x,Y)=\varphi_E(Y)(x)\,.
 $$
Let $\Omega_S:TN\to\mathfrak{so}(S)$ and
$\Omega_E:TN\to\mathfrak{so}(E)$ be the vector bundles maps
($\mathfrak{so}(S)$ and $\mathfrak{so}(E)$ 
being considered as trivial vector bundles over a base reduced to a singleton)
 $$\Omega_S(v)=\pi_{\mathfrak{so}(S)}\circ\widetilde\varphi_S^{-1}(v)\,,\quad
   \Omega_E(v)=\pi_{\mathfrak{so}(E)}\circ\widetilde\varphi_E^{-1}(v)\,,\quad
                                    v\in TN\,,
 $$
where $\pi_{\mathfrak{so}(S)}:N\times\mathfrak{so}(S)\to\mathfrak{so}(S)$ and
$\pi_{\mathfrak{so}(E)}:N\times\mathfrak{so}(E)\to\mathfrak{so}(E)$ are the projections of these two products on their respective second factor.
\par\smallskip

A \emph{motion} of the rigid body during a time interval $[t_0,t_1]$ is mathematically described by a smooth 
parametrized curve $\gamma:[t_0,t_1]\to N$. 
In his beautiful paper~\cite{arnold}, Vladimir Arnold clearly explained the
physical meaning, for each $t\in[t_0,t_1]$, of 
$\displaystyle\frac{\d\gamma(t)}{\d t}$, 
$\displaystyle\Omega_S\left(\frac{\d\gamma(t)}{\d t}\right)$ and
$\displaystyle\Omega_E\left(\frac{\d\gamma(t)}{\d t}\right)$:

\begin{itemize}
\item{} $\displaystyle \frac{\\d\gamma(t)}{\d t}\in T_{\gamma(t)}N$ is the value, at time $t$, of the 
\emph{true angular velocity} of the body, 

\item{} $\displaystyle\Omega_S\left(\frac{\d \gamma(t)}{\d t}\right)$ is the
value, at time $t$, of the 
\emph{angular velocity of the body seen by an observer bound to the moving body}
and moving with it,

\item{}and $\displaystyle\Omega_E\left(\frac{\d\gamma(t)}{\d t}\right)$ is the 
value, at time $t$, of the 
\emph{angular velocity of the body seen by an observer bound to the Galilean
reference frame in which the motion is studied} and at rest with respect to 
that reference frame.
\end{itemize}
\par\smallskip

The following comments may be useful to explain Arnold's assertions. To shorten the notations, let us state, for some time $t\in[t_0,t_1]$, $x=\gamma(t)\in N$,
$\displaystyle v=\frac{\d\gamma(t)}{\d t}\in T_xN$, $X=\Omega_S(v)\in\so(S)$ 
and $Y=\Omega_E(v)\in\so(E)$. We have
 $$\widetilde\varphi_S(x,X)=\widetilde\varphi_E(s,Y)=v\,.$$ 
Let $z\in S$ be some material point of the moving body. Its position at time $t$
is $x(z)\in E$ and its velocity is
$\displaystyle \frac{\d}{\d t}\bigl(\gamma(t)(z)\bigr)\in T_{x(z)}E$.
It depends only of $\displaystyle v=\frac{\d\gamma(t)}{\d t}\in T_xN$, not of the whole curve $\gamma$. We can therefore replace $\gamma$ by the parametrized curve $s\mapsto \exp(sY)\circ x$, since we have
 $$\frac{\d\bigl(\exp(sY)\circ x)}{\d s}\Bigm|_{s=0}=Y_N(x)=v\,.
 $$   
Therefore the velocity at time $t$ of the material point $z\in S$ is
 $$\frac{\d\bigl(\exp(sY)\circ x(z)\bigr)}{\d s}\Bigm|_{s=0}
   =Y_E\bigl(x(z)\bigr)=\bigl(\vect{x(z)},\vect Y\times \vect{x(z)})
     \in TE\equiv E\times E\,,
 $$
where we have denoted by $Y_E$ the fundamental vector field on $E$ associated
to $Y\in\so(E)$for the action $\Phi_E$, and used Formula $(3)$ of Section \ref{VectorCalculus}. This proves that the fundamental vector field $Y_E$ is 
the velocity field of the rigid body as it appears in space $E$ at time $t$, and explains why Arnold calls $Y$ the angular velocity of the body seen by an observer bound to the Galilean frame in which the motion is studied.
\par\smallskip

Since $\displaystyle\frac{\d\gamma(t)}{\d t}=\frac{\d\bigl(x\circ\exp(sX)
        \bigr)}{\d s}\Bigm|_{s=0}$,
the value at $x(z)\in E$ of the fundamental vector field $Y_E$ is also given by
$\displaystyle \frac{\d\bigl(x\circ\exp(sX)(z)\bigr)}{\d s}\Bigm|_{s=0}$.
The pull-back $x^*(Y_E)$ by the isomorphism $x:S\to E$ of the fundamental 
vector field $Y_E$, \emph{i.e.} of the velocity field of the moving body
in space $E$ at time $t$, is the vector field on $S$ whose value at $z\in S$ is
 $$x^*(Y_E)(z)=\frac{\d\bigl(\exp(sX)(z)\bigr)}{\d s}\Bigm|_{s=0}=X_S(z)
   =(\vect z,\vect X\times\vect z)\in TS\equiv S\times S\,,
 $$
where we have again used Formula $(3)$ of Section \ref{VectorCalculus}. The 
pull-back $x^*(Y_E)$ of the velocity field of the moving body in space $E$ at time $t$, by the isomorphism $x:S\to E$,  is therefore the fundamental vector field $X_S$ associated to $X\in\so(S)$, for the action $\Phi_S$. That 
explains why Arnold calls $X$ the angular velocity of the body seen by an observer bound to the moving body.
\par\smallskip

The above observations prove that for any $x\in N$ and $v\in T_xN$,
 $$x\bigl(\vect{\Omega_S(v)}\bigr)=\vect{\Omega_E(v)}\,,\eqno(1)
 $$
the arrows over $\Omega_S(v)\in\so(S)$ and $\Omega_E(v)\in\so(E)$ 
indicating that they are here considered as vectors in $S$ and in 
$E$, respectively.
\par\smallskip

When the true angular velocity of the body is $v\in TN$, its 
\emph{kinetic energy} is 
 $${\mathbb T}(v)=\frac{1}{2} I\bigl(\Omega_S(v),\Omega_S(v)\bigr)\,,
 $$
where $I:\mathfrak{so}(S)\times\mathfrak{so}(S)\to{\mathbb R}$ is a symmetric, 
positive definite bilinear form
which describes the \emph{inertia properties} of the body. 
The assumed rigidity 
of the body is mathematically described by the fact that the bilinear 
form $I$ does not depend on time, nor on the configuration $\tau_N(v)$ 
of the body. Let us set, for each pair $(v,w)$ of vectors in $TN$ such that 
$\tau_N(v)=\tau_N(w)$,
 $$\widetilde{\mathbb T}(v,w)=\frac{1}{2}I\bigl(\Omega_S(v),\Omega_S(w)\bigr)\,,
    \quad\text{so we can write}\ {\mathbb T}(v)=\widetilde{\mathbb T}(v,v)\,.$$
\par\smallskip
The symmetric bilinear form $\widetilde{\mathbb T}$ is a Riemannian 
metric on the manifold $N$. Let us  consider the effects on 
$\widetilde{\mathbb T}$ of the canonical lifts 
to $TN$ of the actions $\Phi_E$ and $\Phi_S$ on the manifold $N$. 
For each $g_E$ in $\SO(E)$, $g_S\in \SO(S)$, we denote
by $\Phi_{E\,g_E}:N\to N$ and by $\Phi_{S\,g_S}:N\to N$ the diffeomorphisms
 $$\Phi_{E\,g_E}(x)=\Phi_E(g_E,x)=g_E\circ x\,,\quad
   \Phi_{S\,g_S}(x)=\Phi_S(x,g_S)=x\circ g_S\,,\quad x\in N\,.
 $$
For each $v\in TN$, with $\tau_N(v)=x\in N$, we have of course
 $$\widetilde\varphi_S\bigl(x,\Omega_S(v)\bigr)=v\,.
 $$   
Since the actions
$\Phi_E$ and $\Phi_S$ commute we have, for any $g_E\in \SO(E)$, 
$v\in TN$, $t\in\RR$ and $x=\tau_N(v)\in N$, 
 $$\Phi_E\Bigl(g_E,x\circ\exp\bigl(t\Omega_S(v)\bigr)\Bigr)=g_E\circ x\circ
          \exp\bigl(t\Omega_S(v)\bigr)
  =\Phi_S\Bigl(g_E\circ x,\exp\bigl(t\Omega_S(v)\bigr)\Bigr)\,.
 $$
By taking the derivative with respect to $t$, then setting $t=0$, we get
 $$T\Phi_{E\,g_E}(v)=\widetilde\varphi_S\bigl(g_E\circ x,\Omega_S(v)\bigr)\,,
 $$    
which means that
 $$\Omega_S\bigl(T\Phi_{E\,g_E}(v)\bigr)=\Omega_S(v)\,.$$
The Riemannian metric $\widetilde{\mathbb T}$ therefore satisfies, 
for each $g_E\in \SO(E)$ and each pair $(v,w)$ of vectors in $TN$ which 
satisfy $\tau_N(v)=\tau_N(w)$,
 $$\widetilde{\mathbb T}\bigl(T\Phi_{E\,g_E}(v),T\Phi_{E\,g_E}(w)\bigr)
   =\widetilde{\mathbb T}(v,w)\,.
 $$
This result means that the Riemannian metric $\widetilde{\mathbb T}$ 
remains invariant under the canonical lift to $TN$ of the action $\Phi_E$.
\par\smallskip

A similar calculation, in which $g_E\in \SO(E)$ is replaced by
$g_S\in \SO(S)$, proves that, for each $v\in TN$,
 $$\Omega_S\bigl(T\Phi_{S\,g_S}(v)\bigr)=\Ad_{g_S^{-1}}\bigl(\Omega_S(v)\bigr)\,,
 $$ 
so we have, for $v$ and $w\in TN$ satisfying $\tau_N(v)=\tau_N(w)$,
 $$\widetilde{\mathbb T}\bigl(T\Phi_{Sg_S}(v),T\Phi_{Sg_S}(w)\bigr)
   =\frac{1}{2}I\bigl(\Ad_{g_S^{-1}}\circ\Omega_S(v),
                       \Ad_{g_S^{-1}}\circ\Omega_S(w)\bigr)\,.
 $$
For a general rigid body, the kinetic energy $\mathbb T$ and the
Riemannian metric $\widetilde{\mathbb T}$ 
do not remain invariant under the canonical lift to $TN$ 
of the action $\Phi_S$. However, 
let us define an action on the left of $G_S$ on the vector space 
of bilinear forms forms on $\mathfrak{so}(S)$ by setting, for each 
such bilinear form $B$ and each $g_S\in G_S$
 $$(g_S.B)(X_S,Y_S)=B\bigl(\Ad_{g_S^{-1}}(X_S),
                            \Ad_{g_S^{-1}}(Y_S)\bigr)\,,\quad
         X_S\ \text{and}\ Y_S\in\mathfrak{so}(S)\,.
 $$ 
We see that the kinetic energy $\mathbb T$ and the Riemannian metric
$\widetilde{\mathbb T}$ remain invariant under the action
of an element $g_S\in G_S$ if and only if $g_S.I=I$, \emph{i.e.} if and only if 
$g_S$ is an element of the isotropy subgroup if $I$ for the above defined action
of $G_S$ on the space of bilinear forms on $\mathfrak{so}(S)$. This happens, for example, when the body has a symmetry axis, the isotropy subgroup of $I$ being the group of rotations around that axis.   
\par\smallskip
 
When the configuration of the body is $x\in N$, its \emph{potential energy} is
 $$U(x)=-\bigl\langle P,x(\vect a)\bigr\rangle\,,$$
where $\vect a\in S$ is the vector whose origin is the fixed point $O_S$ and extremity the centre of mass of the body,
and $P\in E^*$ is the gravity force. Since $E$ is identified with its dual $E^*$, 
the pairing by duality being the scalar product, $P$ can be seen as a fixed 
vertical vector $\vect P\in E$ directed downwards, equal to the weight of the body (product of its mass with the gravity acceleration), and the potential energy can be written
 $$U(x)=-\vect P.\vect{x(a)}=-\vect{x^{-1}(P)}.\vect a\,,$$
where we have written $\vect{x(a)}$ for $x(\vect a)$ 
and $\vect{x^{-1}(P)}$ for $x^{-1}(\vect P)$ 
to stress the fact that they are vectors, elements of $E$ and of $S$, respectively.
We also used the fact that the transpose $x^T:E^*\to S^*$ of the orthogonal linear map $x:S\to E$ is expressed, when $S$ and $E$ are identified with their dual spaces by means of the scalar product, as $x^{-1}:E\to S$.
\par\smallskip

When either $\vect a=0$ or $\vect P=0$ the potential energy vanishes, therefore 
remains invariant under the actions $\Phi_E$ of $G_E$ and $\Phi_S$ of $G_S$ on
the manifold $N$. When both $\vect a\neq 0$ and $\vect P\neq 0$, the above formulae
show that the potential energy remains invariant by the action
of an element $g_E\in \SO(E)$ if and only if $\vect{g_E(P)}=\vect P$, which
means if and only if $g_E$ is an element of the isotropy group of $\vect P$ for the natural action of $G_E$ on $E$. This isotropy subgroup is the group of rotations
of $E$ around the vertical straight line through the fixed point. Simlilarly, the potential energy remains invariant by the action
of an element $g_S\in \SO(S)$ if and only if $\vect{g_S(a)}=\vect a$, which
means if and only if $g_S$ is an element of the isotropy group of $\vect a$ for the natural action of $G_S$ on $S$. This isotropy subgroup is the group of rotations
of $S$ around the straight line which joins the fixed point and the centre of mass
of the body. 
\par\smallskip
  
The motion of the rigid body can be mathematically described
by a Lagrangian system whose Lagrangian $L:TN\to\RR$ is given, for $v\in TN$, by 
 $$L(v)=\widetilde{\mathbb T}(v,v)+\vect P.\vect{\tau_N(v)(a)}
       =\widetilde{\mathbb T}(v,v)+\vect{\bigl(\tau_N(v)\bigr)^{-1}(P)}.\vect a\,.
 $$
We denote by $\widetilde{\mathbb T}^\flat:TN\to T^*N$ the map determined by
the equality, in which $v$ and $w\in TN$ satisfy $\tau_N(v)=\tau_N(w)$,
 $$\bigl\langle\widetilde{\mathbb T}^\flat(v),w\bigr\rangle
   =\widetilde{\mathbb T}(v,w)\,.  
 $$
The Legendre map ${\mathcal L}:TN\to T^*N$ determined by the Lagrangian $L$ is
 $${\mathcal L}=2\widetilde{\mathbb T}^\flat\,.$$
Its linearity and the positive definiteness 
of $I$ ensure that it is a vector bundles isomorphism. The motion of the 
rigid body can therefore be described by a Hamiltonian system whose Hamiltonian
$H:T^*N\to\RR$ is given, for $p\in T^*N$, by
 $$H(p)=\frac{1}{4}\bigl\langle p,(\widetilde{\mathbb T}^\flat)^{-1}(p)\bigr\rangle
   -\vect P.\vect{\pi_N(p)(a)}
       =\frac{1}{4}\bigl\langle p,(\widetilde{\mathbb T}^\flat)^{-1}(p)\bigr\rangle
   -\vect{\bigl(\pi_N(p)\bigr)^{-1}(P)}.\vect a\,.
 $$

\subsubsection{The Hamiltonian in terms of momentum maps}
\label{HamiltonianMomentum}
Let $x\in N$ be fixed. The maps 
 $$\Omega_{Sx}=\Omega_S\bigm|_{T_xN}:T_xN\to\mathfrak{so}(S)\quad\text{and}\quad
   \Omega_{Ex}=\Omega_E\bigm|_{T_xN}:T_xN\to\mathfrak{so}(E)
 $$
are vector spaces isomorphisms. Their transpose
 $$\Omega_{Sx}^T:\mathfrak{so}(S)^*\to T^*_xN\quad\text{and}\quad 
    \Omega_{Ex}^T:\mathfrak{so}(E)^*\to T^*_xN
 $$
are too vector spaces isomorphisms. Their inverses are closely linked to the momentum maps $J_S:T^*N\to\mathfrak{so}(S)^*$ and
$J_E:T^*N\to\mathfrak{so}(E)^*$ of the canonical lifts to $T^*N$ of the actions
$\Phi_S$ of $G_S$ and $\Phi_E$ of $G_E$, respectively, on the manifold $N$.
We have indeed, for any $x\in N$,
 $$J_S\bigm|_{T^*_xN}=(\Omega_{Sx}^T)^{-1}\,,\quad
   J_E\bigm|_{T^*_xN}=(\Omega_{Ex}^T)^{-1}\,.
 $$
As above, let $x\in E$ be fixed and let $v$ and $w\in T_xN$. The Legendre map
${\mathcal L}:TN\to T^*N$ satisfies
 \begin{align*}
 \bigl\langle {\mathcal L}(v),w\bigr\rangle
 &=\frac{1}{2}\frac{\d}{\d s}
   I\bigl(\Omega_S(v+sw),\Omega_S(v+sw)\bigr)\bigm|_{s=0}\\
 &=I\bigl(\Omega_S(v),\Omega_S(w)\bigr)\\
 &=\bigl\langle I^\flat\circ\Omega_S(v),\Omega_S(w)\bigr\rangle\\
 &=\bigl\langle\Omega_{Sx}^T\circ I^\flat\circ\Omega_S(v),w\bigr\rangle\,,
 \end{align*} 
where $I^\flat:\mathfrak{so}(S)\to\mathfrak{so}(S)^*$ is the map defined by
$$\bigl\langle I^\flat(X_S),Y_S\bigr\rangle=I(X_S,Y_S)\,,
   \quad X_S\ \hbox{and}\ Y_S\in\mathfrak{so}(S)\,.
$$
So we can write
 $${\mathcal L}\bigm|_{T_xN}=\Omega_{Sx}^T\circ I^\flat\circ\Omega_S\bigm|_{T_xN}\,,
 $$
which shows that the momentum map $J_S$ composed with the Legendre map 
$\mathcal L$ has the very simple expression 
 $$J_S\circ{\mathcal L}=I^\flat\circ \Omega_S\,.$$
The momentum map $J_E$ composed with $\mathcal L$ has a slightly more complicated 
expression, valid for each $x\in N$,
 $$J_E\circ{\mathcal L}\bigm|_{T_xN}=(\Omega_{Ex}^T)^{-1}\circ\Omega_{Sx}^T
       \circ I^\flat\circ\Omega_S\bigm|_{T_xM}\,.
 $$ 
Let $I^*:\mathfrak{so}(S)^*\times\mathfrak{so}(S)^*\to\RR$ be the symmetric, positive definite bilinear form on ${\mathcal G}^*$
 $$I^*(\xi,\eta)=I\bigl((I^\flat)^{-1}(\xi),(I^\flat)^{-1}(\eta) \bigr)
                =\bigl\langle\xi,(I^\flat)^{-1}(\eta)
                 \bigr\rangle
                =\bigl\langle\eta,(I^\flat)^{-1}(\xi)
                 \bigr\rangle\,.
 $$
The above expression of $J_S\circ{\mathcal L}$ and the bilinear form $I^*$ allow us to write the Hamiltonian $H$ as
 $$H(p)=\frac{1}{2}I^*\bigl(J_S(p),J_S(p)\bigr)-\vect{\pi_N(p)^{-1}(P)}.\vect a\,,
        \quad p\in T^*N\,.
 $$
Although the kinetic energy remains invariant under the canonical lift to $T^*N$
of the action $\Phi_E$,
the expression of $H$ in terms of the other momentum map $J_E$ is too 
complicated to be useful.

\subsubsection{The Euler-Poincar\'e equation}\label{EulerPoincareRigidBody}
We use the vector bundles isomorphism 
$\widetilde\varphi_S:N\times\mathfrak{so}(S)\to TN$ to derive the 
Euler-Poincar\'e equation. The map 
$\overline{\mathstrut L}=L\circ\widetilde\varphi_S:N\times\mathfrak{so}(S)\to\RR$ 
is
 $${\overline{\mathstrut L}}(x,X)=\frac{1}{2}I(X,X)
   +\vect P.\vect{x(a)}\,,
     \quad X\in\mathfrak{so}(S)\,,\ x\in N\,.
 $$
Its partial differential 
$\d_2\overline{\mathstrut L}$ 
with respect to its second variable is
 $$\d_2{\overline{\mathstrut L}}(x, X)=I^\flat(X)\in\mathfrak{so}(S)^*
         \equiv S\,.$$
A calculation similar to those of Section \ref{HamiltonianMomentum} leads to
the following expression of $J_S$ composed with the partial differential of 
$\overline{\mathstrut L}$ with respect to its first variable:
 $$J_S\circ\d_1\overline L(x,X)=\vect a\times\vect{x^{-1}(P)}
       \in\mathfrak{so}(S)^*\equiv S\,.
 $$
Let $t\mapsto x(t)$ be a smooth curve in $N$ solution of the Euler-Lagrange equation
for the Lagrangian $L$, and $t\mapsto X(t)$ a smooth curve in 
$\mathfrak{so}(S)$ which satisfies the compatibility condition $(1)$ of 
\ref{Euler-Poincare2}
 $$\frac{\d x(t)}{dt}=\widetilde\varphi_S\bigl(x(t),X(t)\bigr)
    \,.\eqno(1)
 $$
The Euler-Poincar\'e equation $(3)$ of \ref{Euler-Poincare2}, 
satisfied by the smooth curve
$t\mapsto\bigl(x(t),X(t)\bigr)$ in $N\times\mathfrak{so}(S)$, is
 $$\left(\frac{\d }{\d t}-\ad^*_{X(t)}\right)\Bigl(I^\flat\bigl(X(t)\bigr)\Bigr)
   =\vect a\times\vect{x^{-1}(P)}\,.
 $$
Using the expression of $\ad^*$ given by Formula $(4)$ of \ref{VectorCalculus}, 
we can write the Euler-Poincar\'e equation as 
 $$\frac{\d}{\d t}\vect{I^\flat\bigl(X(t)\bigr)}
   -\vect{I^\flat\bigl(X(t)\bigr)}\times\vect{X(t)}
   =\vect a\times\vect{P_S(t)}\,,\eqno(2)
 $$
where we have set $\vect{P_S(t)}=\vect{x(t)^{-1}(P)}$. 
The physical meaning of the quantities which appear in this equation
is the following: $\vect{X(t)}$ is the angular velocity,
$\vect{I^\flat\bigl(X(t)\bigr)}$ the angular momentum 
and $\vect{P_S(t)}$ the weight of the moving body, all three a time $t$ and seen by an observer bound to the body, therefore considered as vectors 
in $S$. We recognize the classical \emph{Euler equation} for the motion
of a rigid body around a fixed point. 
\par\smallskip

Of course $x(t)\bigl(\vect{P_S(t}\bigr)=\vect P$ is a constant vector in $E$, therefore
 $$\frac{\d\Bigl(x(t)\bigl(\vect{P_S(t}\bigr)\Bigr)}{\d t}
   =\frac{\d x(t)}{\d t}\bigl(\vect{P_S(t}\bigr)
    +x(t)\left(\frac{\d \vect{P_S(t}}{\d t}\right)=0\,.
 $$
The first term in the right hand side, 
$\displaystyle\frac{\d x(t)}{\d t}\bigl(\vect{P_S(t}\bigr)$, is the value at
$\vect P\in E$ of the velocity field in $E$ of the moving body. Therefore
 $$\frac{\d x(t)}{\d t}\bigl(\vect{P_S(t}\bigr)
   =\vect{\Omega_E(v)}\times\vect P\,,\quad\text{with}\ v=\frac{\d x(t)}{\d t}
     \in T_{x(t)}N\,. 
 $$
Therefore we have
 $$\frac{\d \vect{P_S(t)}}{\d t}=-x(t)^{-1}
    \Bigl(\vect{\Omega_E(v)}\times\vect P\Bigr)
    =-\vect{X(t)}\times \vect{P_S(t)}\,,
 $$
since, by Formula $(1)$ of \ref{MathRigidBody},
$x(t)^{-1}\bigl(\vect{\Omega_E(v)}\bigr)=\vect{\Omega_S(v)}=\vect{X(t)}$.
The compatibility condition and the Euler-Poincar\'e equation 
(Equations $(1)$ and $(2)$ of this Section) have lead us to the differential
equation on $S\times S$, for the unknown parametrized curve 
$t\mapsto\bigl(\vect{X(t)},\vect{P_S(t}\bigr)$,
 \begin{equation*}
  \left\{
  \begin{aligned}
   \frac{\d\vect{I^\flat\bigl(X(t)\bigr)}}{\d t}
   &=\vect{I^\flat\bigl(X(t)\bigr)}\times\vect{X(t)}
     +\vect a\times\vect{P_S(t)}\,,\\
     \frac{\d \vect{P_S(t)}}{\d t}
   &=-\vect{X(t)}\times \vect{P_S(t)}\,.
   \end{aligned}
 \right.\eqno(3)
 \end{equation*}

When the right hand side of equation $(2)$ vanishes, 
which occurs when  the 
fixed point is the centre of mass of the body ($\vect a=0$) or when there 
is no gravity field ($\vect P=0$), the Euler-Poincar\'e  equation yields an 
important reduction, since the first equation of $(3)$ 
becomes an autonomous differential 
equation on the three-dimensional vector space $S$ for the smooth curve 
$t\mapsto \vect{I^\flat\bigl(X_S(t)\bigr)}$, while the
Euler-Lagrange equation or the Hamilton equation live on the six-dimensional manifolds $TN$ or $T^*N$, respectively. Under these assumptions, the study 
of all possible motions of the rigid body is is known in Mechanics as the
\emph{Euler-Poinsot problem}. The reader will find in 
\cite{CushmanBates1997} a very nice and thorough geometric presentation 
of the phase portrait of this problem. 
\par\smallskip

When $\vect a\neq 0$ and $\vect P\neq 0$, the first equation of $(3)$ 
is no more autonomous: one has to solve $(3)$ on the six-dimensional 
vector space $S\times S$. The use of the Euler-Poincar\'e equation
does not allow a reduction of the dimension of the phase space, but
$(3)$ may be easier to solve than the Euler-Lagrange equation or the Hamilton equation, because it lives on a vector space instead of on the tangent or cotangent bundle to a manifold.   

\subsubsection{Use of the Lie algebra of Euclidean displacements}
As explained for example in Theorem 4.1 of \cite{Marle2003} or in
Proposition 13 and Example 14 of \cite{Marle2013}, there exists a
Hamiltonian action on $T^*N$ of the semi-direct product $G_S\times S$, 
(the group of Euclidean displacements, generated by rotations and translations, 
of the Euclidean affine space $S$)
which extends the canonical lift to $T^*N$ of the action $\Phi_S$, 
such that the Hamiltonian $H$ can be expressed as composed of the
momentum map of that action with a smooth function 
$h:\mathfrak{so}(S)^*\times S^*\to\RR$. We briefly explain below the construction of that action. 
\par\smallskip

For each $\vect b\in S$, let
$f_{\vect b}:N\to\RR$ be the smooth function
 $$f_{\vect b}(x)=
       \bigl\langle x^T(P),\vect b\bigr\rangle=\vect b.\vect{x^{-1}(P)}\,,
          \quad x\in N\,.
 $$
The map $\Psi:T^*N\times S\to T^*N$ defined by
 $$\Psi(p,\vect b)=p-\d f_{\vect b}\circ\pi_N(p)\,,\quad
    p\in T^*N\,,\ \vect b\in S\,,
 $$
is a Hamiltonian action of $S$ on the symplectic 
manifold $(T^*N,\d\eta_N)$:
the Lie algebra of $S$ can indeed be identified with $S$, 
the exponential map becoming the identity of $S$, and 
for each $\vect b\in S$, the
vector field on $T^*N$ whose flow is the one-parameter group of transformations
of $T^*N$  
 $$\Bigl\{p\mapsto p-t\d f_{\vect b}\bigl(\pi_N(p)\bigr)\,;t\in\RR\Bigr\}$$ 
is Hamiltonian an admits as Hamiltonian the function 
 $$p\mapsto f_{\vect b}\circ\pi_N(p)==\vect b.\vect{\pi_N(p)^{-1}(P)}
     \,,\quad p\in T^*N.$$
This formula proves that $\Psi$ is a Hamiltonian action which 
admits
 $$J_\Psi:T^*N\to S^*\equiv S\,,\quad J_\Psi(p)=\vect{\pi_N(p)^{-1}(P)}
 $$ 
as a momentum map. Gluing together $\Psi$ with the canonical lift $\widehat\Phi_S$ of $\Phi_S$
to the cotangent bundle, we obtain a Hamiltonian action on
the right $\Xi$ of the semi-direct product $G_S\times S$ on the symplectic
manifold $(T^*N,\d\eta_N)$:
 $$\Xi\bigl(p,(g_S,\vect b)\bigr)=\Psi(\widehat\Phi_S(p,g_S),\vect b)$$
with $(J_S,J_\psi):T^*N\to\mathfrak{so}(S)^*\times S^*
\equiv S\times S$ as a momentum map. The function 
$h:\mathfrak{so}(S)^*\times S^*\equiv S\times S\to \RR$
 $$h(\vect\eta,\vect\zeta)=\frac{1}{2}I^*(\vect\eta,\vect\eta)-
      \vect\zeta\vect\eta
 $$
is such that the Hamiltonian $H:T^*N\to \RR$ can be written as
$H=h\circ(J_S,J_\Psi)$, and  
Equation $(3)$ of Section \ref{EulerPoincareRigidBody} is the Hamilton equation 
on $\mathfrak{so}(S)^*\times S^*
     \equiv S\times S$ (endowed with its canonical Poisson structure)
for the Hamiltonian $h$. This result is in agreement with the fact that
$(J_S,J_\Psi)$ is an $\ad^*$-invariant Poisson map 
(\ref{MomentumPoissonMap}).

\subsubsection{Reduction by the use of first integrals}
\label{RigidBodyFirstIntegrals}
The effects on the kinetic and potential energies of the Hamiltonian 
actions $\widehat\Phi_E$ and $\widehat\Phi_S$ were discussed in Section 
\ref{MathRigidBody}. When $\vect a\neq 0$ and $\vect P\neq 0$, the Hamiltonian $H$ remains 
invariant under the restriction of the action $\widehat\Phi_E$ to the subgroup of rotations around the vertical axis through the fixed point. The corresponding momentum map, which is the orthogonal projection of the momentum map $J_E$
on the vertical direction, is therefore a first integral. Another first integral is the total energy, \emph{i.e.} the Hamiltonian $H$ itself. For a general rigid body, no other independent first integrals are known. However, in two special cases of particular rigid bodies, there exists another independent first integral. 
\par\smallskip

The first case, known as the 
\emph{Euler-Lagrange problem} in Mechanics,
is when the straight line which joins the fixed point and the centre of mass of the body is an axis of symmetry for the inertia properties of the body. The Hamiltonian $H$ remains then invariant under the restriction of the action 
$\widehat\Phi_S$ to the subgroup of $\SO(S)$ of rotations around this straight line. The corresponding momentum map is the orthogonal projection of the momentum map $J_S$ on the direction of the symmetry axis.
\par\smallskip

The second case, discovered by the Russian mathematican Sonya Kovalevskaya    
(1850--1891) \cite{Kowalevski1889} is when two of the principal moments of inertia of the body are equal to twice the third and when the centre of mass
of the body lies in the plane of the two equal moments of inertia. The explanation of the existence, in this very special case, of an additional integral is much more complicated than that of the existence of an 
additional integral
for the Euler-Lagrange problem, and involves mathematical tools which are not discussed in the present paper. The reader is referred to the book by 
Mich\`ele Audin \cite{Audin1996} for a discussion of these tools and to the beautiful other book by the same author \cite{Audin2008} for a very moving presentation of the life of Sonya Kovalevskaya.
\par\smallskip

When $\vect a= 0$ or $\vect P= 0$ (the Euler-Poinsot problem) 
the Hamiltonian $H$ remains invariant under the action $\widehat\Phi_E$ 
of the full group $SO(E)$, so the corresponding momentum map $J_E$ is (as already seen
in Section \ref{EulerPoincareRigidBody}) a (vector valued) first integral.

\subsection{The Kepler problem}\label{KeplerProblem}

\subsubsection{Mathematical description of the problem}
We consider the motion in space of a material point of mass $m$  
submitted to the gravitational field created by an attractive centre $O$. 
Taking $O$ as origin allows us to consider $E$ as a \emph{vector}
Euclidean three-dimensional oriented space. The configuration space, 
\emph{i.e.} the set of all possible positions of the material point, 
is $N=E\backslash\{O\}$. The tangent bundle $TN$ and the cotangent bundle 
$T^*N$ will both be identified with $N\times E$. An element of $TN$ is therefore a pair
$(\vect x,\vect v)\in E\times E$ satisfying $\vect x\neq 0$.
Similarly an element of $T^*N$ is a pair
$(\vect x,\vect p)\in E\times E$ satisfying $\vect x\neq 0$.
\par\smallskip

The kinetic energy ${\mathbb T}:TN\equiv N\times E\to\RR$
and the potential energy $U:N\to\RR$ are

 $${\mathbb T(\vect x,\vect v)}=\frac{1}{2}m\Vert\vect v\Vert^2\,,\quad
     U(\vect x)=\frac{mk}{\Vert\vect x\Vert}\,.
 $$
The Lagrangian $L:TN\equiv N\times E\to\RR$ of the Kepler problem is therefore
 $$L(\vect x,\vect v)=\frac{1}{2}m\Vert\vect v\Vert^2
   +\frac{mk}{\Vert\vect x\Vert}\,.
 $$ 
The Legendre map ${\mathcal L}:TN\equiv N\times E\to T^*N\equiv N\times E$ is
 $${\mathcal L}(\vect x,\vect v)=(\vect x,\vect p)\,,\quad\text{with}\ 
     \vect p=m\vect v\,.
 $$
The Kepler problem can therefore be mathematically formulated as a Hamiltonian
dynamical system on $T^*N\equiv N\times E$, with the Hamiltonian
 $$H(\vect x,\vect p)=\frac{1}{2m}\Vert\vect p\Vert^2
   -\frac{mk}{\Vert\vect x\Vert}\,.
 $$
The natural action $\Phi_E$ of ${\SO}(E)$ on $E$ leaves invariant 
$N=E\backslash\{O\}$, therefore is an action of $\SO(E)$ on $N$. 
With the identifications of $TN$ and
$T^*N$ with $N\times E$ which we have made, the canonical lifts $\overline\Phi_E$
and $\widehat\Phi_E$
of that action to the tangent and cotangent bundles, respectively, 
are expressed as 
 $$\overline\Phi_E\bigl(g_E,(\vect x,\vect v)\bigr)
   =\bigl(g_E(\vect x),g_E(\vect v)\bigr)\,,\
    \widehat\Phi_E\bigl(g_E,(\vect x,\vect p)\bigr)
   =\bigl(g_E(\vect x),g_E(\vect p)\bigr)\,.
 $$
Since the norm of a vector in $E$ remains invariant under the action $\Phi_E$, the Lagrangian $L$ and the Hamiltonian $H$ remain invariant under the actions
$\overline\Phi_E$ and $\widehat\Phi_E$, respectively. The action $\widehat\Phi_E$ is Hamiltonian, and we know (Formula $(4)$ of \ref{VectorCalculus}) that its momentum map $J_E:T^*N\equiv N\times E\to \mathfrak{so}(E)^*\equiv E$ is
 $$J_E(\vect x,\vect p)=\vect x\times\vect p\,.
 $$
The map $J_E$ is the \emph{angular momentum} 
of the moving material point with respect to the attractive centre. 
Noether's theorem (\ref{noether}) shows that
it is a first integral of the Kepler problem. Another first integral of 
the Kepler problem is the total energy $H$, as shown by \ref{FirstIntegrals}.

\subsubsection{The Euler-Poincar\'e equation}
The Lie group $\SO(E)$ does not act transitively on $N$ by the action $\Phi_E$, since the orbits of this action are spheres centered on $O$. However, extending
this action by homotheties of strictly positive ratio, we obtain a transitive action on $N$ of the direct product $\SO(E)\times\,]0,+\infty[$
 $$\Psi_E\bigl((g_E,r),\vect x\bigr)=rg_E(\vect x)\,,\quad g_E\in\SO(E)\,,\ 
    r\in\,]0,+\infty[\,,\ \vect x\in N\,.
 $$
Let $\psi_E:\mathfrak{so}(E)\times\RR\to A^1(N)$ be the associated action of
the Lie algebra $\mathfrak{so}(E)\times\RR$. The map
$\widetilde\psi_E:N\times\bigl(\mathfrak{so}(E)\times\RR\bigr)\to TN$, 
$\widetilde\psi_E\bigl(\vect x,(\vect X,\lambda)\bigr)
=\psi_E(\vect X,\lambda)(\vect x)$,
can be written, with the identifications of $\mathfrak{so}(E)$ with $E$ and of $TN$ with $N\times E$,
 $$\widetilde\psi_E\bigl(\vect x,(\vect X,\lambda)\bigr)
   =\bigl(\vect x, \vect X\times\vect x+\lambda\vect x)\,.
 $$
The function $\overline{\mathstrut L}=L\circ\widetilde\psi_E$ is therefore
 $$\overline{\mathstrut L}\bigl(\vect x,(\vect X,\lambda)\bigr)
   =\frac{1}{2}m\Vert\vect x\Vert^2\bigl(\Vert\vect X\Vert^2+\lambda^2\bigr)
    -\frac{1}{2}m(\vect X.\vect x)^2
    +\frac{mk}{\Vert\vect x\Vert}\,.
 $$
Its partial differentials $\d_1\overline{\mathstrut L}$ and 
$\d_2\overline{\mathstrut L}$ with respect to its first variable $\vect x$ and to its second variable $(\vect X,\lambda)$ are, with the identifications of 
$E^*$ and $\bigl(\mathfrak{so}(E)\times\RR)^*$ with, respectively,
$E$ and $E\times\RR$,
 \begin{align*}
  \d_1\overline{\mathstrut L}\bigl(\vect x,(\vect X,\lambda)\bigr)
 &=\left(m(\Vert\vect X\Vert^2+\lambda^2)-
    \frac{mk}{\Vert\vect x\Vert^3}\right)\vect x
     -m(\vect X.\vect x)\vect X\,,\\
     \d_2\overline{\mathstrut L}\bigl(\vect x,(\vect X,\lambda)\bigr)
    &=\bigl(m\Vert\vect x\Vert^2\vect X
      -m(\vect X.\vect x)\vect x,
       m\Vert\vect x\Vert^2\lambda\bigr)\,.
 \end{align*}
The canonical lift $\widehat\Psi_E$ of $\Psi_E$ to the cotangent bundle
is a Hamiltonian action, whose momentum map 
$(J_E, K_E):T^*N\to \mathfrak{so}(E)^*\times\RR$ has $J_E$ as first component. Its second component is
 $$K_E(\vect x,\vect p)=\vect x.\vect p\,.
 $$
Let $t\mapsto \vect{x(t)}$ be a smooth curve in $N$, parametrized by the time $t$, solution of the Euler-Lagrange equation for the Lagrangian $L$. The compatibility
condition $(1)$ of \ref{Euler-Poincare2}, for a smooth map 
$t\mapsto\bigl(\vect{X(t)},\lambda(t)\bigr)$ in $\mathfrak{so}(E)\times\RR$, 
can be written as
 $$\frac{\d\vect{x(t)}}{\d t}=\vect{X(t)}\times \vect{x(t)}+\lambda(t)\vect{x(t)}\,.
    \eqno(1)$$
This equation does not involve the component of
$\vect{X(t)}$ parallel to $\vect{x(t)}$, since the vector product of this component with $\vect{x(t)}$ vanishes.  
\par\smallskip

The Euler-Poincar\'e equation $(3)$ of \ref{Euler-Poincare2} has now two
components, on $\so(E)^*$ and  on $\RR^*$ identified, respectively,
with $E$ and with $\RR$.
With the above expressions of $\d_1\overline{\mathstrut L}$,
$\d_2\overline{\mathstrut L}$, $J_E$ and $K_E$, we obtain for its first
component
 \begin{align*}
 \frac{\d}{\d t}&\left(m\Vert\vect{x(t)}\Vert^2\left(\vect{x(t)}
                         -\frac{\vect{x(t)}.\vect{x(t)}}
                              {\Vert\vect x\Vert^2}\vect{x(t)}\right)\right)
\\
&=m\bigl(\vect{x(t)}.\vect{x(t)}\bigr)\vect{x(t)}\times\vect{x(t)}
  -m\bigl(\vect{x(t)}.\vect{x(t)}\bigr)\ad^*_{\vect{x(t)}}\vect{x(t)}
   =0\,,
 \end{align*}
where we have used Formula $(4)$ of \ref{VectorCalculus}.
Its second component is
 $$\frac{\d}{\d t}\bigl(m\Vert\vect{x(t)}\Vert^2\lambda\bigr)
   =m\left(\Vert\vect{x(t)}\Vert^2+\lambda^2
       -\frac{\bigl(\vect{x(t)}.\vect{x(t)}\bigr)^2}{\Vert\vect x\Vert^2}
       \right)\Vert\vect{x(t)}\Vert^2
    -\frac{mk}{\Vert\vect{x(t)}\Vert}\,.
 $$
The vector $\vect{X(t)}$ is the sum of two components $\vect{X_1(t)}$ orthogonal to $\vect{x(t)}$ and $\vect{X_2(t)}$ parallel to $\vect{x(t)}$. Since
 $$\vect{X_1(t)}
  =\vect{X(t)}-\frac{\vect{X(t)}.\vect{x(t)}}{\Vert\vect{x(t)}\Vert^2}
    \vect{x(t)}\,,
 $$  
the two components of the Euler-Poincar\'e equation become
 \begin{equation*}\left\{
  \begin{aligned}
  &\frac{\d}{\d t}\bigl(m\Vert\vect{x(t)}\Vert^2\vect{X_1(t)}\bigr)=0\,,\\
  &\frac{\d}{\d t}\bigl(m\Vert\vect{x(t)}\Vert^2\lambda\bigr)
  =m\left(\Vert\vect{X_1(t)}\Vert^2+\lambda^2
    \right)\Vert\vect{x(t)}\Vert^2
    -\frac{mk}{\Vert\vect{x(t)}\Vert}\,.
  \end{aligned}
\right.
\end{equation*}
The first equation expresses the fact that $J_E$ is a first integral of the Kepler problem, since we have
 $$m\Vert\vect{x(t)}\Vert^2\vect{X_1(t)}=\vect{x(t)}\times\vect{p(t)}
     =J_E\bigl(\vect{x(t)},\vect{p(t)}\bigr)\,.$$
Similarly, the second equation can be written
 $$\frac{\d}{\d t}\bigl(\vect{p(t)}.\vect{x(t)}\bigr)
   =\frac{\Vert\vect{p(t)}\Vert^2}{m}-\frac{km}{\Vert\vect x\Vert}\,,
 $$
which is a direct consequence of Hamilton's equations for the Hamiltonian $H$ of the Kepler problem.
\par\smallskip
 
Neither the Euler-Poincar\'e equation nor the compatibility condition
involve the component $\vect{X_2(t)}$ of $\vect{X(t)}$ parallel to 
$\vect{x(t)}$. This illustrates the fact that the system made by these equations is underdetermined when the dimension of the Lie algebra which acts on the configuration space is strictly larger than the dimension of this space.

\subsubsection{Hamilton's method of solving the Kepler problem}
The Hamiltonian $H$ of the Kepler problem remains invariant under the canonical lift to $T^*N$ of the action of $\SO(E)$. Noether's theorem (\ref{noether})
shows that the corresponding momentum map $J_E$ is a first integral. Of course
the total energy, \emph{i.e.} the Hamiltonian $H$, is too a first integral 
(\ref{FirstIntegrals}).
Following the method due to Hamilton 
\cite{Hamilton1846},
we explain below how  the three Kepler laws can easily be deduced 
from the first integrals $J_E$ and $H$.
\par\smallskip

Let us assume that at a particular time $t_0$, $\vect{x(t_0)}$ 
and $\vect{p(t_0)}$ are not collinear. The vector 
$\vect\Omega=J_E\bigl(\vect{x(t)},\vect{p(t)}\bigr)=\vect{x(t)}\times\vect{p(t)}$
does not depend on $t$ since $J_E$ is a first integral, and is $\neq 0$ since for $t=t_0$, $\vect{x(t)}$ and $\vect{p(t)}$ are not collinear.
We choose an orthonormal positively oriented basis 
$(\vect{e_x},\vect{e_y},\vect{e_z})$ of $E$ 
such that $\vect\Omega=\Omega\vect{e_z}$, with $\Omega>0$.
The vectors $\vect{x(t)}$ and $\vect{p(t)}$ remain for all times
in the two-dimensional vector subspace $F$ spanned by $(\vect{e_x},\vect{e_y})$. Let $\theta(t)$ be the polar angle made by $\vect{x(t)}$ with $\vect{e_x}$. We have
\begin{equation*}
 \begin{split}
  {\vect{x(t)}}&=r(t)\cos\theta(t) \vect{e_x}+r(t) \sin\theta(t) \vect{e_y}\,,\\
  {\vect{p(t)}}&=m\left(\frac{\d r(t)}{\d t}\cos\theta(t)
    -r(t)\frac{\d\theta(t)}{\d t}\sin\theta(t)\right)\vect{e_x}\\
    &\ +m\left(\frac{\d r(t)}{\d t}\sin\theta(t) 
      + r(t)\frac{\d\theta(t)}{\d t}\cos\theta(t)\right)\vect{e_y}\,\\
  \vect \Omega&=mr^2\frac{\d\theta(t)}{\d t}\vect{e_z}\,.
 \end{split}
\end{equation*}
Therefore
 $$mr^2\frac{d\theta}{dt}=\Omega=\hbox{Constant}\,.$$
This is the \emph{second Kepler law}, also called \emph{law of areas}, since
$\displaystyle\frac{\Omega}{2m}$ is the area swept by the straight line segment joining the moving material point to the attractive centre
during an unit time. Since
$t\mapsto\theta(t)$ is a strictly increasing function whose derivative never vanishes, we can take $\theta$ instead of time $t$ as independent variable.
Using Hamilton's equation (or Newton's equation), we can write
 $$\frac{\d\vect p(\theta)}{\d\theta}=\frac{\d\vect{p(t)}}{dt}\,
    \frac{\d t}{d\theta}=\frac{mr(\theta)^2}{\Omega}
     \left(-\frac{mk}{r(\theta)^3}\vect{x(\theta)}\right)=
    -\frac{m^2k}{\Omega}(\cos\theta\vect{e_x}+\sin\theta\vect{e_y})\,.
 $$
This ordinary differential equation for the unknown $\vect p(\theta)$, 
which no more involves $\vect{x(\theta)}$, can be readily integrated:
 $$\vect p(\theta)=\frac{m^2k}{\Omega}(-\sin\theta \vect{e_x}+\cos\theta\vect{e_y})+\vect c\,,$$
where $\vect c$ is a (vector) integrating constant.
We will choose $\vect{e_y}$ such that $\vect c = c \vect{e_y}$,
where $c$ is a numeric constant which satisfy $c\geq 0$.

With $O$ as origin let us draw two vectors in the plane $xOy$, the first one (constant) being equal to
$\vect c$, and the second one (which varies with $\theta$) equal to $\vect p$.
The end point of that second vector moves on a circle whose centre is the end point of the vector equal to
$\vect c$, and whose radius is
$\displaystyle{\mathcal R}=\frac{m^2k}{\Omega}$. The part of this circle swept
by the end point of this second vector is (up to multiplication by $m$) 
the \emph{hodograph} of the Kepler problem.
A short calculation leads to the following very simple relation between the energy $H$ of a motion, the
radius $\mathcal R$ of its hodograph and the distance $c$ from the attracting centre $O$ to the centre of the
hodograph:
 $$2mH=c^2-{\mathcal R}^2\,.$$
The right-hand side $c^2-{\mathcal R}^2$ is the \emph{power}\footnote{In plane Euclidean geometry,
the \emph{power}\label{power} of a point $O$ with respect to a circle $\mathcal C$ is the real number $\vect{OA}.\vect{OB}$, where $A$
and $B$ are the two intersection points of $\mathcal C$ with a straight line $\mathcal D$ through $O$. That number does not depend
on $\mathcal D$ and is equal to $\Vert \vect{OC}\Vert^2-{\mathcal R}^2$, where $C$ is the centre and $\mathcal R$ the radius of $\mathcal C$.} 
of $O$ with respect to the hodograph. 
\par\smallskip

We also obtain $r=\Vert\vect x\Vert$ as a function of $\theta$
 $$r(\theta)=\frac{\Omega^2}{m^2k+\Omega c\cos\theta}
            =\frac{\Lambda}{1+\varepsilon\cos\theta}\,,\quad
\hbox{with}\ \Lambda=\frac{\Omega^2}{m^2k}\,,\ 
              \varepsilon=\frac{\Omega c}{m^2k}\,.
$$
It is the polar equation of a conic section with $O$ as focus point and
$\varepsilon$ as eccentricity. This conic section (or, when 
$\varepsilon>1$, the arc of this conic swept by the moving material point)
is the orbit in $E$ of the moving material point. This result 
is the \emph{first Kepler law}.
\par\smallskip
 
The modulus $\Omega$ of the angular momentum, the total energy $H$
and the eccentricity $\varepsilon$ satisfy
 $$\varepsilon^2-1=\frac{2\Omega^2H}{m^3k^2}\,.$$
This formula shows that the orbit in $E$ of the moving material point is an ellipse ($0\leq\varepsilon<1)$ if $H<0$, a parabola ($\varepsilon=1$)
if $E=0$ and a connected component of a hyperbola ($\varepsilon>1$)
if $H>0$.
\par\smallskip

When $H<0$, the orbit in $E$ of the moving point is an ellipse and its motion is periodic. The period $T$ is easily obtained by writing that the area swept in a time $T$ by the straight line segment which joins the moving point to the attractive centre is the area $A$ delimited by the orbit:
 $$T=\frac{2mA}{\Omega}=\frac{2\pi m a^2\sqrt{1-\varepsilon^2}}{\Omega}
 $$  
where $a$ is the length of the half major axis of the orbit. By using the formula
 $$\Omega^2=m^2ka\sqrt{1-\varepsilon^2}
 $$
we obtain
 $$T^2=\frac{4\pi^2}{k}\,a^3\,.
 $$
We conclude that the square of the period is proportional to the third power of the length of the half major axis. This result is the \emph{third Kepler law}.
\par\smallskip

Hamilton's method of solving the Kepler problem is much easier than 
the Marsden-Weinstein reduction procedure, to which it is only very 
loosely related. A 
non-zero vector $\vect \Omega$ is a regular value of $J_E$, 
so $J_E^{-1}(\vect \Omega)$ is a smooth three-dimensional submanifold of 
$T^*N$:
it is the set of pairs of vectors $(\vect x,\vect p)\in F\times F$ such that
$\vect x\times\vect p=\vect\Omega$, where $F$ is the two-dimensional vector subspace of $E$ orthogonal to $\vect\Omega$. This submanifold remains invariant under the action on $T^*N$ of the one-dimensional subgroup of $\SO(E)$,
isomorphic to the circle $S^1$, of rotations around the straight line 
through $O$ parallel to $\vect\Omega$. The reduced Marsden-Weinstein symplectic manifold is the set of orbits of this action. It is isomorphic to the open half-plane $\bigl\{(r,\lambda)\in\RR^2;r>0\bigr\}$, and 
the projection of $J^{-1}(\vect\Omega)$ onto the reduced symplectic manifold
is the map $(\vect x,\vect p)\mapsto(r,\lambda)$, with $r=\Vert\vect x\Vert$,
$\lambda=\vect x.\vect p$. The reduced symplectic form and Hamiltonian are, respectively, 
 $$\omega_{\vect\Omega}=\frac{1}{r}\d\lambda\wedge \d r\,,\quad
   H_{\vect \Omega}=\frac{m(\Omega^2+\lambda^2)}{2r^2}-\frac{mk}{r}\,.
 $$
Instead of using this reduced symplectic manifold and this reduced Hamiltonian,
Hamilton's method uses a clever choice of independent and dependent variables 
on $J_E^{-1}(\vect\Omega)$ which leads to an easy to solve autonomous
differential equation for $\vect p$ as a function of the polar angle 
$\theta$ of $\vect x$. It is successful essentially because the hodograph 
of the Kepler problem is a circle (or, when $H\geq 0$, a part of a circle).      

\subsubsection{The eccentricity vector}
There exists still another vector valued first integral 
$\vect\varepsilon$ of the Kepler problem called the 
\emph{eccentricity vector}, discovered by
Jakob Hermann (1678--1753) three centuries ago
\cite{Bernoulli1710, Herman1710}, often improperly called the
\emph{Laplace vector} or the \emph{Ruge-Lenz vector}, whose expression is
 $$\vect \varepsilon=-\frac{\vect x}{\Vert\vect x\Vert}
                      +\frac{\vect p\times(\vect x\times\vect p)}{m^2k}
   = \left(\frac{\Vert\vect p\Vert^2}{m^2k}-\frac{1}{\Vert\vect x\Vert}\right)
      \,\vect x-\frac{\vect p.\vect x}{m^2k}\,\vect p\,.
 $$
For each motion of the moving material point, the eccentricity vector 
$\vect\varepsilon$ is a dimensionless vector parallel
to the straight line segment which joins the attractive centre $O$ to the perihelion of the orbit (\emph{i.e.} the point of the orbit which is 
the nearest to the attractive centre), of length numerically equal
to the eccentricity $\varepsilon$ of the orbit. 
When the orbit is a circle, the perihelion is undetermined and 
$\vect\varepsilon=0$. We briefly explain below the group theoretical origin of the eccentricity vector. A more detailed explanation can be found for example in
\cite{Marle2012}. Many other interesting results about the Kepler problem 
can be found in the excellent books \cite{Cordani2003, 
GuilleminSternberg1990, CushmanBates1997, Feynman1996, 
Guichardet2012}.
\par\smallskip

Motions $t\mapsto(\vect{x(t)},\vect{p(t)})$ of the Kepler problem in which 
$\vect{x(t)}$ and $\vect{p(t)}$ are parallel are not defined for all values 
of the time $t$: the curves drawn in $E$ by the vectors $\vect{x(t)}$ and 
$\vect{p(t)}$ both are supported by the same straigh line through the 
attractive centre $O$, so the motion finishes, or begins, at a finite time, 
when $\vect{x(t)}$ reaches $0$, \emph{i.e.} when the moving point collides 
with the attractive centre or is expelled by it. When $t$ tends towards 
that final (or initial) instant, $\Vert \vect p(t)\Vert$ tends towards 
$+\infty$. This fact complicates the study of the global topological 
properties of the set of all possible motions of the Kepler problem.
\par\smallskip

For any motion $t\mapsto(\vect{x(t)},\vect{p(t)})$, the curves drawn in $E$ 
by the vectors $\vect{x(t)}$ and $\vect{p(t)}$ are, respectively, the orbit and the hodograph of the motion. The exchange $(\vect x,\vect p)\mapsto
(\vect p,\vect x)$ is an anti-symplectic map, which allows us, at the price of a change of sign of the symplectic form, to consider the curve drawn by 
$\vect{p(t)}$ as the orbit of some Hamiltonian dynamical system and the curve drawn by $\vect{x(t)}$ as the corresponding hodograph. This remark offers a way of studying the global properties of the set of all possible motions: for a motion $t\mapsto(\vect{x(t)},\vect{p(t)})$ which starts or ends at finite instant by a collision with the attractive centre or an ejection by that point,
the curve drawn by $\vect{p(t)}$, now considered as an orbit rather than a hodograph, goes to infininy when $t$ tends towards this limit instant. By a inverse stereographic projection, $E$ can be mapped on a three-dimensional sphere $Q$ minus a point (the pole of the stereographic projection), and the curve drawn by $\vect{p(t)}$ is mapped onto a curve which tends towards the pole $P$ of the stereographic projection. The canonical prolongation of the inverse stereographic projection to the cotangent bundles allows us to map the phase space of the Kepler problem onto the open subset of $T^*Q$ complementary to 
the fibre $T^*_PQ$ over the pole of the stereographic projection. 
On $T^*Q$, motions which reach $T^*_PQ$ can be prolongated and no more 
appear as starting, or ending, at a finite instant of time. This idea, due to Fock \cite{Fock1935} who applied it to the study of the hydrogen atom in 
quantum mechanics, was used by Moser \cite{Moser1970} for the regularization 
of the Kepler problem for negative values of the Hamioltonian $H$.
Gy\"orgyi \cite{Gyorgyi1968} used a similar idea. Since the inverse 
stereographic projection maps circles onto circles, the image of $\vect{p(t)}$ draws a circle on the three-dimensional sphere 
$Q$ and, for a particular value of the total energy $H$, this circle is a great circle, \emph{i.e.} a geodesic of $Q$. Gy\"orgyi \cite{Gyorgyi1968}
proved that the cylindrical projection onto $E$ of these great circles 
are ellipses centered on $O$ whose eccentricity is the same as those of the orbits drawn on $E$ by the corresponding vector $\vect{x(t)}$. 
The group $\SO(4)$ acts on the three-dimensional sphere $Q$ and, 
by the canonical lift to the cotangent bundle,
on $T^*Q$ by a Hamiltonian action. The transformed Hamiltonian is not really invariant under that action and some more work (a reparametrization of time) 
is still needed, but finally Noether's theorem can be used. The eccentricity vector $\vect\varepsilon$ is (modulo the identification of the phase space of 
the Kepler problem with an open subset of $T^*Q$) the momentum map of 
that action, which explains why it is a first integral.
\par\smallskip

For motions with a positive value of the total energy, there exists a
similar construction in which instead of a three-dimensional sphere, $Q$ 
is a two-sheeted revolution three-dimensional hyperboloid. The symmetry group
is the Lorentz group $\SO(3,1)$; the eccentricity vector $\vect\varepsilon$
still is the momentum map of its action on $T^*Q$
\cite{Milnor1983, Anosov2002}. For a motion with a zero value of $H$, 
the circle drawn in $E$ by the vector $\vect{p(t)}$ contains the attractive
centre $O$, so an inversion with $O$ as pole transforms this circle into a straight line, \emph{i.e.} a geodesic of $E$. The symmetry group is then 
the group of Euclidean displacements in $E$ (generated by rotations and translations); the eccentricity vector $\vect\varepsilon$
still is the momentum map of its action on $T^*E$.
\par\smallskip

Ligon ans Schaaf \cite{LigonSchaaf1976} used these results to construct a global
symplectic diffeomorphism of the phase space of the Kepler problem
(for negative values of $H$) onto an open subset of the cotangent bundle to a
three-dimensional sphere. Gy\"orgyi had done that earlier
\cite{Gyorgyi1968} but it seems that his work was not known by mathematicians. 
Later several other authors pursued these studies 
\cite{CushmanDuistermaat1997, HeckmanDeLaat}.
\par\smallskip

Souriau \cite{Souriau1983} used a totally different approach. He built
the regularized manifold of motions of the Kepler problem in a single step,
for all values of the energy, by successive derivations of the equations of motion and analytic prolongation, calculated its symplectic form
and directly determined its symmetry groups. The eccentricity vector appears again as a momentum map for the Hamiltonian actions of these groups.

\section{Acknowledgements}
I address my thanks to Iva\"{\i}lo Mladenov for his kind invitation to present lectures at the International Conference on Geometry, Integrability and Quantization held in Varna in June 2013.
\par\smallskip

For their moral support and many stimulating scientific discussions,
I thank my colleagues and friends Alain Albouy,
Marc Chaperon, Alain Chenciner, Maylis Irigoyen, 
Jean-Pierre Marco, Laurent Lazzarini, Fani Petalidou, 
G\'ery de Saxc\'e, 
Wlodzimierz Tulczyjew, Pawe\l\  Urba\'nski, Claude
Vall\'ee and Alan Weinstein.

\end{document}